\documentclass[11pt,a4paper,reqno]{amsart}
\usepackage{amsmath,amsfonts,amssymb,amsthm,amscd,color,pdfsync,stmaryrd,mathrsfs}
\usepackage{mathrsfs}
\usepackage[english]{babel}
\usepackage{graphicx}
\usepackage{hyperref,pdfsync}

\newcommand{\C}{{\mathbb C}}
\newcommand{\N}{{\mathbb N}}
\newcommand{\R}{{\mathbb R}}

\newcommand{\ds}{\displaystyle}
\newcommand{\norm}[1]{\left\Vert#1\right\Vert}
\newcommand{\mc}[1]{\mathcal{#1}}
\newcommand{\mb}[1]{\mathbb{#1}}
\newcommand{\mf}[1]{\mathfrak{#1}}

\newcommand{\curl}{{\rm curl}\,}

\renewcommand{\div}{{\rm div}\,}

\renewcommand{\phi}{{\varphi}}

\newcommand{\pd}{\partial}

\newtheorem{theorem}{Theorem}[section]
\newtheorem{proposition}[theorem]{Proposition}
\newtheorem{lemma}[theorem]{Lemma}
\newtheorem{corollary}[theorem]{Corollary}

\theoremstyle{remark}
\newtheorem{remark}[theorem]{Remark}

\numberwithin{equation}{section}

\definecolor{green}{rgb}{0.,0.6,0}

\begin{document}

\title[Long-time behavior of fluid-solid interaction]{Long-time behavior for the two-dimensional motion of a disk in a viscous fluid}
\author{S. Ervedoza, M. Hillairet \& C. Lacave}

\address[S. Ervedoza]
{CNRS ; Institut de Math\'ematiques de Toulouse UMR 5219 ; F-31062 Toulouse, France, \\
 Universit\'e de Toulouse ; UPS, INSA, INP, ISAE, UT1, UTM ; IMT ; F-31062 Toulouse, France.
}
\email{ervedoza@math.univ-toulouse.fr}

\address[M. Hillairet]{Universit\'e Paris-Dauphine \\
CEREMADE \\
UMR 7534 - CNRS \\
Place du Mar\'echal De Lattre De Tassigny
75775 Paris Cedex 16.
} \email{hillairet@ceremade.dauphine.fr}

\address[C. Lacave]{Universit\'e Paris-Diderot (Paris 7)\\
Institut de Math\'ematiques de Jussieu - Paris Rive Gauche\\
UMR 7586 - CNRS\\
B\^atiment Sophie Germain \\
Case 7012\\
75205 PARIS Cedex 13\\
France.} \email{lacave@math.jussieu.fr}

\date{\today}

\begin{abstract}
In this article, we study the long-time behavior of solutions of the two-dimensional fluid-rigid disk problem. The motion of the fluid is modeled by the two-dimensional Navier-Stokes equations, and the disk moves under the influence of the forces exerted by the viscous fluid. We first derive $L^p$-$L^q$ decay estimates for the linearized equations and compute the first term in the asymptotic expansion of the solutions of the linearized equations. We then apply these computations to derive time-decay estimates for the  solutions to the full Navier-Stokes fluid-rigid disk system.
\end{abstract}

\maketitle

\tableofcontents

\section{Introduction}

We consider the system formed by a rigid disk and a viscous fluid filling the whole plane $\mathbb R^2.$
We assume that the body initially occupies the disk $B_0$ and rigidly moves so that at  time $t$  it occupies a domain denoted by $B(t)$ that is isometric to $B_0.$
We denote $\mathcal{F} (t) := \R^2  \setminus B(t) $ the domain occupied by the fluid  at  time $t$ starting from the initial domain $\mathcal{F}_{0}  := \R^2 \setminus B_{0} $.

When the fluid has constant viscosity $\nu>0,$ the equations modeling the dynamics of the system fluid-rigid disk read
\begin{eqnarray}
	 \frac{\partial u }{\partial t}+(u  \cdot \nabla)u -\nu \Delta u  + \nabla p =0 && \text{for} \ t \in (0,\infty), \, x \in  \mathcal{F} (t), \label{NS1}\\
	\div u   = 0 && \text{for} \ t \in (0,\infty), \, x \in \mathcal{F}(t) , \label{NS2} \\
	u(t,x)  =   h' (t)+ \omega (t) (x-  h (t))^\perp && \text{for}  \  t\in  (0,\infty), \, x \in \partial B  (t),   \label{NS3} \\
	\lim_{|x|\to \infty} |u(t,x)| =0& & \text{for}  \  t \in  (0,\infty), \\
	m h'' (t) =  -\int_{\partial  B (t)} \Sigma n \, {\rm d}s & & \text{for}  \  t \in  (0,\infty), \label{Solide1} \\ 
	\mathcal{J} \omega' (t) =   - \int_{\partial B(t)} (x-  h (t) )^\perp \cdot \Sigma n   \, {\rm d}s & & \text{for}  \  t \in  (0,\infty), \label{Solide2} \\
	u |_{t= 0} = u_0 & &  \text{for}  \  x\in  \mathcal{F}_0 ,  \label{NSci2} \\
	h (0)= h_0 , \ h' (0)=  \ell_0 ,\    \omega (0)=  \omega_0.&&  \label{Solideci}
\end{eqnarray}
Here, $u=(u_1,u_2)$ denotes the velocity-field, $p$ the pressure and $\Sigma$ is the Cauchy stress tensor of the fluid:
\begin{equation}
	\label{Tensor} 
	\Sigma  = -p {\rm Id} + 2 \nu D(u),
\end{equation}
where ${\rm Id}$ is the $2\times 2$ identity matrix and:
\[ 
	D(u)_{k,\ell} = \frac12 \Bigl( \frac{\pd u_k}{\pd x_\ell} + \frac{\pd u_\ell}{\pd x_k}\Bigl) \qquad 1 \leq k,\ell \leq 2\,.
\]
The constants  $m$ and $ \mathcal{J}$ denote respectively the mass and the inertia of the body  while the fluid  is supposed to be  homogeneous, of density $1$ to simplify notations. In this work, we assume that the solid is homogeneous of density $m/\pi$, implying in particular $\mathcal{J}=m/2$ (we discuss in Section \ref{sect comments} about a generalization).
When $x=(x_1,x_2) \in \mathbb R^2,$ the vector $x^\perp $ stands for $x^\perp =( -x_2 , x_1 )$, 
$n$ denotes  the unit normal vector to $\partial B(t)$ pointing outside the fluid domain,  $h'(t)$
is the velocity of the center of mass  $h (t)$ of the body and $\omega(t)$ denotes the angular velocity of the rigid body.
Indeed, since  $B (t)$ is isometric to $B_0$ there exists a rotation matrix 
\begin{equation*}
S_{\theta(t)} :=  
	\begin{bmatrix}
		\cos  \theta (t) & - \sin \theta (t)
	\\  
		\sin  \theta (t) & \cos  \theta (t)
	\end{bmatrix},
\end{equation*}
such that  the lagrangian coordinates $\eta (t,x)$ associated to the body read:
\begin{equation*}
	\eta (t,x) := h (t) + S_{\theta(t)} (x- h_0) .
\end{equation*}
Furthermore, the angle $\theta$ satisfies
$	\theta'(t) = \omega (t), $
and is chosen such that $\theta (0) =  0$. Without loss of generality, we assume that $B_0$ is the unit disk centered at the origin: $h(0) = 0$.

\bigskip

Given $(u_0,\ell_0,\omega_0) \in H^1(\mathcal{F}_0) \times \mathbb R^2 \times \mathbb R,$ satisfying the compatibility condition:
\begin{equation*}
	 \div u_0=0 \text{ in }\mathcal{F}_0,\quad u_0  =   \ell_0+ \omega_0 x^\perp \text{ on }  \partial B_0,
\end{equation*}
T. Takahashi and M. Tucsnak prove in \cite{Takahashi&Tucsnak04} that there exists a  unique global strong solution $(u,p,h,\omega)$ of \eqref{NS1}-\eqref{Solideci}. 
The construction is based on the change of variable:
\begin{equation}\label{translation}
	v (t,x) = u(t,x-h(t))\,, \qquad \tilde p(t,x) = p(t,x-h(t))\,, \qquad \ell(t) = h'(t)\, .
\end{equation}
The new unknowns $(v,\tilde p)$ are then defined in the fixed domain $[0,\infty)\times (\mathbb R^2 \setminus B_0)$ and system \eqref{NS1}-\eqref{Solideci} reads,
in terms of $(v,\tilde p,\ell,\omega)$:
\begin{eqnarray}
	 \frac{\partial v }{\partial t}+\Bigl((v-\ell)  \cdot \nabla\Bigl)v -\nu \Delta v  + \nabla \tilde{p} =0 && \text{for} \ (t,x)\in (0,\infty) \times \mathcal{F}_0, \label{NS11}\\
	\div v   = 0 && \text{for} \ (t,x)\in (0,\infty) \times \mathcal{F}_0 , \label{NS12} \\
	v(t,x)  =   \ell(t)+ \omega (t) x^\perp && \text{for} \ (t,x)\in (0,\infty) \times \pd {B}_0,   \label{NS13} \\
	\lim_{|x|\to \infty} |v(t,x)| =0,& & \text{for} \ t\in (0,\infty), \\
	m \ell' (t) =  -\int_{\partial  B_{0}} \Sigma n \, {\rm d}s , & & \text{for} \ t\in (0,\infty), \label{Solide11} \\ 
	\mathcal{J} \omega' (t) =  - \int_{\partial B_{0}} x ^\perp \cdot \Sigma n   \, {\rm d}s ,& & \text{for} \ t\in (0,\infty),\label{Solide12} \\
	v |_{t= 0} = v_0 & &  \text{for}  \  x\in  \mathcal{F}_0 ,  \label{NSci12} \\
	\ell(0)=  \ell_0\, ,\,   \omega (0)=  \omega_0,  & &  \label{Solideci1}
\end{eqnarray}
with
\[ 
	\Sigma  = -\tilde{p} \, {\rm Id} + 2 \nu D(v).
\]
These solutions verify the following energy decay estimate: 
\begin{multline} \label{formalenergy}
\dfrac{1}{2} \left[ \int_{\mathcal {F}_0} |v(t,x)|^2 \, {\rm d}x + \left( m|\ell(t)|^2 + \mathcal{J} |\omega(t)|^2 \right) \right] + 2\nu \int_{0}^t \int_{\mathcal{F}_0} |D(v)(\tau,x)|^2 \, {\rm d}\tau {\rm d}x \\ \leq 
 \dfrac{1}{2} \left[ \int_{\mathcal {F}_0} |v_0(x)|^2 \, {\rm d}x + \left( m|\ell_0|^2 + \mathcal{J} |\omega_0|^2 \right) \right] \,, \quad \forall \, t >0\,.
\end{multline}
Relying on this estimate,  T. Takahashi and M. Tucsnak prove the existence and uniqueness of a global weak solution to \eqref{NS11}--\eqref{Solideci1} for initial data $(v_0,\ell_0,\omega_0)$ such that $v_0 \in L^2(\mathcal F_0)$ and 
\begin{equation} \label{eq_assumv_0}
	 \div v_0=0, \text{ in }\mathcal{F}_0, \qquad v_0 \cdot n  = (  \ell_0+ \omega_0 x^\perp ) \cdot n,  \text{ on }  \partial B_0.
\end{equation} 
{\em In this article, we aim at studying the long-time behavior of  these weak solutions.}

\bigskip

The long-time behavior of solutions to fluid-structure interaction systems has already been tackled in different ways. In a series of papers, several authors study the asymptotics of systems without pressure, \emph{i.e.} where the Navier Stokes equations are replaced by a heat equation   \cite{MunnierZuazua2,MunnierZuazua,VazquezZuazua,Vazquez&Zuazua04}. In this simplified case, the force applied by the fluid on a solid is modeled by the circulation of the normal derivative of the velocity-field $u$ on the solid boundaries. In the one-dimensional case in \cite{VazquezZuazua,Vazquez&Zuazua04}, and then in several dimensions in \cite{MunnierZuazua,MunnierZuazua2}, the authors show that the multiplier method introduced in \cite{EscobedoZuazua} to study the asymptotic behavior of solutions to convection-diffusion equations (also applied to the porous medium equation in \cite{Vazquez}) enables to compute sharp decay estimates and asymptotic expansions of solutions up to the second order. Even if the divergence-free condition \eqref{NS12} significantly modifies the equations, we will strongly use the results in \cite{MunnierZuazua,MunnierZuazua2}.

The long-time behavior of solutions for the full Navier Stokes equations in the whole space is also a long-standing question that has motivated numerous studies.
Applying a  Fourier decomposition, M.~E. Schonbek and M. Wiegner show in \cite{Schonbek,Wiegner} that the $L^2$ norm of the Navier-Stokes solution decreases to zero, which was a question raised  by J. Leray \cite{Leray}. In \cite{Carpio}, A. Carpio obtains a sharp description of  the pressure, which is given  by $p= \Delta^{-1}(\div(u\cdot \nabla u)).$  Representing then the velocity-field by a Duhamel formula and using a scaling argument, she computes the development of the solution for long times up to the second order.

Another approach consists in removing the pressure by taking the curl of the momentum equation:
\begin{equation} \label{eq_vorticite}
	\pd_{t}\curl u + u\cdot \nabla \curl u - \nu \Delta \curl u =0,
\end{equation}
where $\curl u$ is the vorticity of the fluid.
Without boundaries, such an equation yields the decay of the $L^p$ norms of the vorticity $\curl u$. 
For $\curl u_0 \in L^1(\mathbb R^2),$ such that $\int \curl u_0 \neq 0,$ T. Gallay and C.E. Wayne prove in \cite{GallayWayne}  that the vorticity 
behaves as $t \to \infty$ like the heat kernel 
\[
	\left(\int \curl u_{0}\right)\frac{e^{-\frac{|x|^2}{4t}}}{4\pi t} .
\] 
Note that if $\curl u_{0}$ is compactly supported and integrable, then thanks to the Biot-Savart law, we have, for large $x$,
\[
	u_{0}(x) =\frac{\int \curl u_{0}}{2\pi} \frac{x^\perp}{|x|^2}  + \mathcal{O}_{|x| \to \infty} \Bigl( \frac1{|x|^2} \Bigl).
\]
Of course, this implies that $u_{0}$ does not belong to $L^1(\mathbb R^2)\cup L^2(\mathbb R^2)$ if $\int \curl u_{0}\neq 0$. Consequently, this theory corresponds
to solutions with infinite energy. For instance, in \cite{GallayWayne}, T. Gallay and C.E. Wayne deduce that the velocity behaves asymptotically as $t \to \infty$ 
like the Lamb-Oseen vector field:
\[
\frac{\int \curl u_{0}}{2\pi}\frac{x^\perp}{|x|^2}\Bigl(1-e^{-\frac{|x|^2}{4t}}\Bigr)
\]
which has infinite energy.

In a domain with boundaries, system \eqref{eq_vorticite} has to be completed with boundary conditions. When Dirichlet boundary conditions are imposed for the velocity-field, one might compute Robin boundary conditions for the vorticity but with non-dissipative coefficients. Therefore, working on the vorticity seems difficult.
In the case of one obstacle surrounded by a viscous fluid (\emph{i.e.}, when $B_0$ is fixed and the system reduces to the Navier-Stokes equations in the exterior domain $\mathcal F_0$ completed with homogeneous Dirichlet boundary condition on $\partial \mathcal{F}_0$), the recent works \cite{GallayMaekawa,Iftimieetcie} prove that the first term in the long-time behavior of the velocity-field is given by the Lamb-Oseen vector field. Their proofs consist in a perturbative argument showing that the decay estimates for the solutions of the Stokes problem, which were established in \cite{DS99a,DS99,MS97}, implies that the nonlinear terms tend faster to zero than the Stokes solution. To our knowledge, such decay estimates on the Stokes semigroup are only known for fixed domains with homogeneous Dirichlet boundary conditions for the velocity-field. 

The only result considering the long-time behavior of a moving particle inside a Navier Stokes fluid is due to E. Feireisl and S. Ne\v{c}asov\`a \cite{FeireislNecasova}. However, they assume the whole system to be confined in a bounded container and they take into account the influence of gravity. Hence, they obtain completely different results with completely different methods. 
Broadly speaking, they prove that, if the container has no vertical wall and contains only one particle, the particle reaches the bottom of the container asymptotically in time. 

\bigskip

One of the main steps in \cite{GallayMaekawa,Iftimieetcie} is to establish $L^p-L^q$ decay estimates for solutions to the linear Stokes equations underlying the Navier Stokes system. Such results are known for fixed domains with homogeneous Dirichlet boundary condition (see \cite{DS99a,MS97}),  but in our case, the linearized Stokes fluid-solid system reads:
\begin{eqnarray}
	 \frac{\partial v }{\partial t} -\nu \Delta v  + \nabla p = 0 && \text{for} \ (t,x)\in (0,\infty) \times \mathcal{F}_0, \label{S1}\\
	\div v   = 0 &&\text{for} \ (t,x)\in (0,\infty) \times \mathcal{F}_0 , \label{S2} \\
	v  =   \ell (t)+ \omega (t) x^\perp &&  \text{for}  \  (t,x)\in  (0,\infty) \times \partial B _0,   \label{S3} \\
	\lim_{|x|\to \infty} |v(t,x)| =0& & \text{for}  \  t \in  (0,\infty), \\
	m \ell' (t) =  -\int_{\partial  B_0} \Sigma n \, {\rm d}s & & \text{for}  \  t \in  (0,\infty),  \label{S-Solide1} \\ 
	\mathcal{J} \omega' (t) =   - \int_{\partial B_0} x ^\perp \cdot \Sigma n   \, {\rm d}s& & \text{for}  \  t \in  (0,\infty),\label{S-Solide2} \\
	v |_{t= 0} = v_0 & &  \text{for}  \  x\in  \mathcal{F}_0 ,  \label{Sci2} \\
	 \ell (0)=  \ell_0 , \   \omega (0)=  \omega_0.&&  \label{S-Solideci}\, 
\end{eqnarray}
To our knowledge, $L^p-L^q$ ($p,\,q\neq 2$) estimates are not available in the literature for solutions of \eqref{S1}--\eqref{S-Solideci}. 
To be more precise, in \cite{Takahashi&Tucsnak04} , T. Takahashi and M. Tucsnak construct solutions of \eqref{S1}--\eqref{S-Solideci} {\em via} a semigroup approach. They show that the semigroup is analytic on $\mc{L}^2$ (to be defined in \eqref{Def-Lp} below) in dimension 2. In \cite{Yun&Wang}, the semigroup is also proved to be analytic in the counterpart of the spaces $\mc{L}^{6/5}\cap\mc{L}^p$ in 3D. However, in both papers the subsequent decay estimates are not sufficient for our purpose. In the first case, the $L^2$ framework only is considered. In the second case, the authors do not obtain sharp decay estimates.

\bigskip

To state our results precisely we introduce shortly some notations.
From a triplet $(v_0,\ell_0,\omega_0)$ verifying \eqref{eq_assumv_0}, we define a divergence-free vector field denoted $V_0$ on $\R^2$ obtained by extending $v_0$ by $\ell_0 + \omega_0  x^\perp$ in $B_{0}$. Adapted to such $V_0$, we introduce the functional spaces $\mathcal L^p$ defined as follows: 
\begin{equation}
	\label{Def-Lp}
	\mc{L}^p= \{ V\in L^p(\R^2),\ \div V=0 \text{ in }\R^2,\ D(V)=0 \text{ in }B_{0} \}, \quad (p \in [1,\infty]).
\end{equation}
When $p \in [1,\infty),$ we endow these spaces with the norms 
\begin{equation*}
	\| V \|_{\mc{L}^p}^p =  \int_{\mathcal{F}_0} |V|^p + \frac{m}{\pi} \int_{B_0} |V|^p\,.
\end{equation*}
It is easy to check that, if $V\in \mc{L}^p,$ then $V=\ell_{V}+\omega_{V} x^\perp$ on $B_0,$ where
\begin{equation}\label{l-omega}
	\ell_{V}= \frac{1}{\pi}\int_{B_{0}}V(x) \, {\rm d}x,\quad \omega_{V} = \frac{2}{\pi}\int_{B_{0}} V(x)\cdot x^\perp \, {\rm d}x,
\end{equation}
and the normal component of $V$ is continuous across $\partial B_0$ (as in \eqref{eq_assumv_0}). In particular, we remark that, setting $v = V|_{\mc{F}_0}$, there holds:
\begin{equation*}
	\| V \|_{\mc{L}^p}^p  \sim \|v\|_{L^p(\mc{F}_{0})}^p + |\ell_{V}|^p +  |\omega_{V}|^p.
\end{equation*}
Such a space is obviously a Banach space as a closed subspace of $L^p(\mathbb R^2).$ A straightforward extension
of \cite[Theorem III.2.3]{Galdi} yields that $\mathcal{L}^p \cap \mathcal{C}^{\infty}_c(\mathbb R^2)$ is dense in $\mathcal L^p$
for arbitrary $p \in (1,\infty).$ For $p =2,$ the space  $\mathcal{L}^p$ is a Hilbert space as the norm is associated with the scalar product:
\begin{equation} \label{eq_ps}
\langle  V ,  W \rangle_{\mc{L}^2} = \int_{\mc{F}_0}   V \cdot   W + \frac{m}{\pi} \int_{B_0}  V\cdot  W\,.
\end{equation}
For $p \in (1,\infty)\setminus \{2\},$  the same bilinear form enables to identify the dual of $\mathcal{L}^p$ with $\mathcal{L}^{{p'}}$ where ${p'}$ is the conjugate exponent of $p.$

Naturally, we endow $\mathcal L^{\infty}$ with the norm:
\begin{equation*}
	\| V \|_{\mc{L}^{\infty}} = \|V\|_{L^{\infty}(\mathbb R^2)}\,.
\end{equation*}
For any $V \in \mathcal L^{\infty},$ we still have $V(x) = \ell_V + \omega_V x^{\bot}$ in $B_0$ with $\ell_V$ and $\omega_V$ defined by \eqref{l-omega}.
Hence, there holds again: 
\begin{equation*}
	\| V \|_{\mc{L}^{\infty}}  \sim  \max \left\{ \|v\|_{L^{\infty}(\mc{F}_{0})}, |\ell_{V}| ,   |\omega_{V}|\right\}\,.
\end{equation*}

\bigskip

Our first results concern the Cauchy problem for \eqref{S1}-\eqref{S-Solideci} in $\mc{L}^p$ and the decay rates of the constructed solutions. As in \cite{Takahashi&Tucsnak04, Yun&Wang}, we use a semigroup approach:  
\begin{theorem} \label{theo Stokes} 
	For each $q\in (1,\infty)$, the Stokes operator of the linear problem \eqref{S1}-\eqref{S-Solideci} generates a semigroup $S(t)$ on $\mc{L}^q$ which satisfies the following decay estimates:

	$\bullet$ For $p \in [q, \infty]$, there exists $K_1 = K_1(p,q)>0$ such that for every $V_0 \in \mc{L}^{q}$:
		\begin{equation}\label{Lp-Lq}
			\|S(t)V_0\|_{\mc{L}^p} \leq K_1 (\nu t)^{\frac{1}{p} - \frac{1}{q}}\|V_0\|_{\mc{L}^q}
			\qquad \text{for all}\quad  t>0.
		\end{equation} 

	$\bullet$ If $q \leq 2$, for  $p \in [q,2]$, there exists $K_2 = K_2( p,q)>0$ such that for every $V_0 \in \mc{L}^{q}$,
		\begin{equation}\label{est grad 1}
			 \|\nabla S(t)V_0\|_{L^p(\mc{F}_0)} \leq K_2 (\nu t)^{-\frac{1}{2} + \frac{1}{p} - \frac{1}{q}}\|V_0\|_{\mc{L}^q}
			\qquad \text{for all}\quad  t>0.
		\end{equation}

	$\bullet$ For $p \in [\max\{2,q\},\infty)$, there exists $K_3 = K_3(p,q)>0$ such that for every $V_0 \in \mc{L}^{q}$,
		\begin{equation}\label{est grad 2}
			 \|\nabla S(t)V_0\|_{L^p(\mc{F}_0)} \leq 
			\left\{ \begin{array}{ll} 
				K_3 (\nu t)^{-\frac{1}{2} + \frac{1}{p} - \frac{1}{q}}\|V_0\|_{\mc{L}^q} &\qquad \text{for all}\quad  0< t< \frac1{\nu}, \\
				 K_3 (\nu t)^{ - \frac{1}{q}}\|V_0\|_{\mc{L}^q} &\qquad \text{for all}\quad   t \geq \frac1{\nu}.
			 \end{array}\right.
		\end{equation} 
\end{theorem}
Our approach is based on the decomposition of the velocity in spherical harmonics. We show that the 0-spherical harmonic verifies a heat equation (without pressure) with dynamical boundary conditions.  This  enables us to compute decay estimates using the multiplier method of Escobedo-Zuazua \cite{EscobedoZuazua} in the same way as in \cite{MunnierZuazua2,MunnierZuazua}. 
The 1-spherical harmonic is the hardest part. \emph{A priori}, it verifies an equation with pressure and non-standard boundary conditions.
However, we show that there exists an underlying algebra which enables to reduce this equation to a heat equation (without pressure) with dynamical boundary 
conditions. So, we can again reproduce the method of  \cite{EscobedoZuazua} in the spirit of \cite{MunnierZuazua2,MunnierZuazua}. 
We do not expand the remainder ({\em i.e.}, the $k$-spherical harmonic for $k\geq 2$) in this part, as we show it satisfies the Stokes equations with Dirichlet boundary condition on $B_{0}$ which has been studied formerly in several papers \cite{DS99a,MS97}.  

Going further in the spherical-harmonic decomposition, we are also able to compute an asymptotic expansion of the solution to the Stokes system \eqref{S1}-\eqref{S-Solideci}  for well-localized initial data:

\begin{theorem} \label{thm_AsymptoticStokes}
For all $p \in [2,\infty]$, and for any $V_{0}\in \mc{L}^1\cap  L^2(\R^2,\exp(|x|^2/4){\rm d}x)$, setting $\ell_0 = \ell_{V_0}$ and 
\begin{equation*}
	\vec{\mathcal{M}} =( m - \pi) \ell_{0},
\end{equation*}
we have
\begin{align}
	& \lim_{t \to \infty} t^{1-1/p} \norm{S(t)V_{0} - U_{\vec{\mathcal{M}}}(t,\cdot)}_{L^p(\mc{F}_{0})}  =  0 ,  \label{eq_asymStokes1}
	\\
	& \lim_{t \to \infty} t \left|\ell_{S(t)V_{0}} - \frac{\vec{\mathcal{M}} }{8 \pi\nu t} \right| =  0, \label{eq_asymStokes2}
	\\
	& \limsup_{t \to \infty} t^2 \left| \omega_{S(t)V_{0}}    \right|  <   +\infty, \label{eq_asymStokes4}
\end{align}
where
\begin{equation*}
	U_{\vec{\mathcal{M}}}(t,x)  = \nabla^{\perp}\left[ \frac{1-e^{-\frac{|x|^2}{4\nu t}}}{2\pi |x|^2} \vec{\mathcal{M}} \cdot x^{\perp}\right].
\end{equation*}
\end{theorem}

Before going further, let us emphasize the following important remark, which can be easily deduced from explicit computations: provided $\vec{\mathcal{M}} \neq 0$, for all $p >1$, 
$$
	0 <  \liminf_{t \to \infty } t^{1 - 1/p} \norm{U_{\vec{\mathcal{M}}}(t,\cdot)}_{L^p(\mc{F}_{0})}= \limsup_{t \to \infty } t^{1 - 1/p} \norm{U_{\vec{\mathcal{M}}}(t,\cdot)}_{L^p(\mc{F}_{0})}=  \norm{U_{\vec{\mathcal{M}}}(1,\cdot)}_{L^p(\R^2)} < \infty.
$$

The quantity $\vec{\mathcal{M}}$ represents the total momentum of the system. Indeed, since any $V_0 \in L^1(\mathbb{R}^2)$ satisfying the divergence free condition has $0$ mean value,
$$
	\int_{\mc{F}_0} V_0 \, {\rm d}x + m \ell_0 = - \int_{B_0} V_0 \, {\rm d}x + m \ell_0 = ( m - \pi) \ell_0 = \vec{\mathcal{M}}.
$$ 

Therefore, if $\vec{\mathcal{M}}\neq 0,$ \eqref{eq_asymStokes1} shows that $ U_{\vec{\mathcal{M}}}$ is the first term in the asymptotic expansion of $S(t)V_0$. 
We deduce also from \eqref{eq_asymStokes2}--\eqref{eq_asymStokes4} that, provided $\vec{\mathcal{M}} \neq 0$, the solid, whose center of mass corresponds to $h_{S(t)V_0} = \int_0^t \ell_{S(s)V_0} \, {\rm d}s$,  goes logarithmically to infinity and stops turning. 

If $\vec{\mathcal{M}}=0$, then a careful reading of the proof of Theorem \ref{thm_AsymptoticStokes} yields
\begin{equation}\label{eq_asymStokes5}
\limsup_{t \to \infty} \left(\frac{t^{5/4}}{|\log(t)|^{1/2}} \left| \ell_{S(t)V_{0}}  \right| \right) < +\infty 
\end{equation}
which implies that the disk converges to a fixed state when considering the linearized equations \eqref{S1}--\eqref{S-Solideci}.
Note that the condition $\vec{\mathcal{M}} = 0 $ is satisfied in the following two cases:
\begin{itemize}
	\item $m = \pi$, that is the case of a solid having exactly the same density as the fluid.
	\item $\ell_0 = 0$, that is the case of a solid whose center of mass has zero initial velocity.
\end{itemize}
Thus, when $\vec{\mathcal{M}} = 0$, we expect a different behavior as $t \to \infty$ of the solutions of the Stokes system \eqref{S1}--\eqref{S-Solideci}. 

\medskip

With the decay estimates of Theorem \ref{theo Stokes} at hand, we study the long-time behavior of solutions to the Navier Stokes system \eqref{NS11}--\eqref{Solideci1}.
We prove that, for small initial data, such solutions satisfy decay estimates similar to the one of the solutions to the Stokes equations:

\begin{theorem}\label{Thm-q-dans(1,2)}
	Let $q \in (1, 2]$. Then there exists $\lambda_0(q)>0$ such that, for all initial data $V_0 \in \mc{L}^q \cap \mc{L}^2$ satisfying the smallness assumption
	\begin{equation}
		\label{Smallness-Intro}
		\|  V_{0} \|_{\mc{L}^2} \leq \lambda_{{0}}(q),
	\end{equation}
	the unique weak solution $V$ of \eqref{NS11}--\eqref{Solideci1} with initial data $V_0$ satisfies the following decay estimates:
	\begin{itemize}
		\item for all $p \in [2, \infty)$, there exists $H(p,q, V_0)$ such that
		\begin{equation}
			\label{H(p,q)}
			\sup_{t >0} \{ t^{\frac{1}{q} - \frac{ 1}{p}} \| V(t) \|_{\mc{L}^p}\} \leq H(p,q,V_0).
		\end{equation}
		\item there exists $H_\ell(q,V_0)$ such that 
		\begin{equation}
			\label{C-ell-q}
			\sup_{t >0} \{t^{\frac{1}{q}} |\ell_V(t)| \} \leq H_\ell(q, V_0).
		\end{equation}
	\end{itemize}
	Besides, the function $q \mapsto \lambda_0(q)$ can be chosen as an increasing function of $q \in (1,2]$ which goes to zero as $q \to 1$.
\end{theorem}

The proof of Theorem \ref{Thm-q-dans(1,2)} consists of two steps. First, we consider the case $q = 2$. Following the idea developed by Kato in \cite{Kato84}, we construct successive approximations $Y_{n}$ which verify the decay estimates \eqref{H(p,q)} uniformly in $n$ for $p=2$, $p = 8$ and \eqref{C-ell-q}. Next, we pass to the limit to get a solution to the Navier-Stokes equations with such a time-decay. To reach $p \in [2,\infty)$ we use a bootstrap argument based on the Duhamel formula as in \cite{Carpio}. We then develop the case $q \in (1,2)$ by showing that estimates \eqref{H(p,q)}  are satisfied uniformly by the sequence $Y_n$.

Using then a bootstrap argument allows us to quantify the proximity of the solution of the non-linear system \eqref{NS11}--\eqref{Solideci1} and of the linear system \eqref{S1}--\eqref{S-Solideci}:

\begin{theorem}\label{theo kato} 
	Let $q \in (1,2]$. Taking $\lambda_0(q) >0$ as in Theorem \ref{Thm-q-dans(1,2)}, for any $V_0 \in \mc{L}^q \cap \mc{L}^2$ verifying \eqref{Smallness-Intro},
	the unique global solution $V$ of \eqref{NS11}-\eqref{Solideci1} with initial data $V_0$ verifies:
	 for all  $p \in [2,\infty)$ there exist constants $C(p,q, V_0)>0$ for which:
	\begin{eqnarray}
			&\ds \sup_{t >2}  \{ t^{1-\frac1p}  \| V(t) -S(t) V_0  \|_{\mc{L}^p}\}  \leq C(p,q,V_0),
			\quad & \text{if $q \in (1, 4/3)$},
			\label{eq_firstorder-a}
			\\
			&\ds \sup_{t >2}  \Big\{ \frac{t^{1-\frac1p}}{\log(t)}   \| V(t) -S(t) V_0 \|_{\mc{L}^p} \Big\}  \leq C(p,q,V_0),
			\quad & \text{if $q =4/3$},
			\label{eq_firstorder-b}
			\\
			& \ds \sup_{t >2}  \{ t^{\frac2q - \frac12-\frac1p}  \| V(t)  - S(t) V_0 \|_{\mc{L}^p}\}  \leq C(p,q,V_0)
			\quad & \text{if $q\in ( 4/3,2]$.}\quad
			\label{eq_firstorder-c}
	\end{eqnarray}
	Similarly, there exist constants $C_\ell(q,V_0)>0$ such that
	\begin{eqnarray} 
			&\sup_{t >2}  \{ t |\ell_{V}(t)- \ell_{S(t)V_0}|\}  \leq C_{\ell}(q,V_0),
			\quad & \text{if $q \in (1, 4/3)$},
			\label{eq_firstorder2-a}
			\\
			&\ds \sup_{t >2}  \Big\{ \frac{t}{\log(t)}   | \ell_V(t)-\ell_{S(t)V_0}  | \Big\}  \leq C_\ell(q,V_0),
			\quad & \text{if $q = 4/3$},
			\label{eq_firstorder2-b}
			\\
			& \ds \sup_{t >2}  \{ t^{\frac{2}{q} - \frac{1}{2}}  | \ell_V(t)- \ell_{S(t)V_0}  |\}  \leq C_\ell(q,V_0).
			\quad & \text{if $q \in ( 4/3,2]$}.
			\label{eq_firstorder2-c}
	\end{eqnarray}
\end{theorem}

Let us comment the fact that if the initial data $V_0$ belongs to $\mc{L}^q \cap \mc{L}^2$ for some $q \in (1,2)$ and satisfies the smallness condition \eqref{Smallness-Intro}, the $\mc{L}^p$-norm of the difference between the solution of the complete non-linear system \eqref{NS11}--\eqref{Solideci1} and the linear one, given by $S(t) V_0$, decays faster than the a priori decay estimates predicted by Theorem \ref{theo Stokes}. Indeed, we check easily that $1-\frac1p > \frac1q-\frac1p$ for any $q >1$ and that $\frac2q-\frac12-\frac1p>\frac1q-\frac1p$ for any $q<2$.

 Combining Theorem \ref{theo kato} and Theorem \ref{theo Stokes} and taking $\lambda_0 = \lambda_0(5/4)$, we can guarantee that, for all $q\in (1,2]$ and all $V_0 \in \mc{L}^q \cap \mc{L}^2$ satisfying 
$\| V_0 \|_{\mc{L}^2} \leq \lambda_0,$
we have 
$$\sup_{t >2}  \{ t^{\frac1q-\frac1p}  \| V(t)\|_{\mc{L}^p}\} < \infty.$$
Indeed, for $q >5/4$, it relies on Theorem \ref{Thm-q-dans(1,2)} and the fact that $q \mapsto \lambda_0(q)$ is increasing. For $q \in (1,5/4]$, it is a simple combination of Theorem \ref{theo kato} (for $\| V_0 \|_{\mc{L}^2} \leq \lambda_0(5/4)$ because $V_0\in \mc{L}^{5/4}$) with the decay estimates of Theorem \ref{theo Stokes} (because  $V_0\in \mc{L}^{q}$). If $V_0$ satisfies the further assumption $V_0 \in \mc L^1 \cap L^2(\R^2,\exp(|x|^2/4)\, {\rm d}x)$ with $\| V_0 \|_{\mc{L}^2} \leq \lambda_0(5/4)$, we can also combine Theorem \ref{theo kato} with Theorem \ref{thm_AsymptoticStokes} yielding that, for all $p \geq 2$:  
$$
\sup_{t>2} t^{1-\frac 1p} \|V(t)\|_{\mc{L}^p} < \infty.
$$
In all these cases, we obtain that the solution to the Navier Stokes system thus decays with time (at least) as fast as the solution to the Stokes system. 

\bigskip

The paper is organized as follows. In next section, we collect some preliminary results. We explain the decomposition of the velocity-field in spherical harmonics.
We then compute the different equations satisfied by the different modes of the velocity-field and we end up the section by several elliptic lemmas that 
will be used further. Section \ref{sect2} is devoted to the proof of {Theorem \ref{theo Stokes}} and {Theorem \ref{thm_AsymptoticStokes}}. Section \ref{sect3} contains
the proof of Theorem \ref{Thm-q-dans(1,2)} and Theorem \ref{theo kato}. The article ends by some comments and open problems, in particular the lack of the first asymptotic term of the Navier-Stokes solution.

\bigskip

{\bf Notations.}	In the whole article, we use classical notations for function spaces. The symbol $L^p(\Omega, {\rm d}\mu)$ stands for the Lebesgue space with respect to measure ${\rm d}\mu$ defined on an open set $\Omega \subset \mathbb R^n$.  
If ${\rm d}\mu$ is the Lebesgue measure, we drop ${\rm d}\mu$. Sobolev spaces are denoted by $H^m(\Omega),$ $m\in\mathbb Z.$  
Further notations for function spaces are introduced along the paper. We shall use extensively symbol $L$ in different fonts (such as $\mathcal L^p, \ \mathfrak{L}^p$). This will correspond to variants of Lebesgue spaces. The only exception concerns $\mathscr{L}_c(X)$ (resp. $\mathscr{L}_c(X \to Y)$), which represent the Banach space of continuous linear operators from a Banach space $X$ to itself (resp. a Banach space $X$ to another Banach space $Y$).

In what follows, we will use capital letters to denote functions defined on $\R^2$, as we did for the velocity $V$ above, and denote by the corresponding small characters the restriction on $\mc{F}_0$. To be more precise, for $V,\,W,\, Z$ ($\cdots$) defined on $\R^2$, functions $v, \, w, \, z$ denote the corresponding restrictions of $V,\, W, \, Z$ on $\mc{F}_0$ and $\ell_V,\, \ell_W, \ell_Z$ the mean value of $V, \, W, \, Z$ on $B_0$. In the sequel, when considering functions $W,\, Z$ ($\cdots$) which are constant on $B_0$, we will identify them with the couple $(w,\ell_W),\ (z, \ell_Z)$ ($\cdots$) and write $W \doteq (w, \ell_W)$, $Z \doteq (z, \ell_Z)$. In the case of the velocity $V$ in $\mc{L}^p$, the restriction of $V$ is $\ell_V+ \omega_V x^\perp$ and thus we also identify $V$ with the triplet $(v, \ell_V, \omega_V)$ and note $V \doteq  (v, \ell_V, \omega_V)$.

\section{Preliminary results}\label{sectpas2maispresk}

We first recall how the Cauchy problem for \eqref{S1}--\eqref{S-Solideci} has been tackled in \cite{Takahashi&Tucsnak04}.
Formal energy estimates imply that, for a sufficiently smooth and localized initial data, $V(t) \in \mathcal L^2$ for all $t.$
In this framework, system \eqref{S1}--\eqref{S-Solideci} reduces to the abstract ODE: $\partial_t V+ A V = 0,$
where $A$ is the unbounded operator with domain:
\begin{equation} \label{domainStokes}
	 \mathcal{D}(A) = \{ V \in H^1(\R^2), \, V_{|\mc{F}_0} \in H^2(\mc{F}_{0}),\ \div V=0 \text{ in }\R^2,\ D(V)=0\text{ in } B_{0} \} 
\end{equation}
such that for $V \in \mathcal{D}(A)$, $A V:=\mb{P}_{2}\mc{A} V$  with 
\begin{equation*}
\mc{A} V =
\left\{\begin{array}{cc} - \nu \Delta V &\text{ in } \mc{F}_{0}\,, \\[8pt] \dfrac{2 \nu}m \displaystyle{\int_{\partial B_{0}}} D(V) n\, {\rm d}s +\frac{2 \nu}{\mc{J}}x^\perp \displaystyle{\int_{\partial B_{0}}} y^\perp\cdot D(V) n\, {\rm d}s(y) &\text{ in } B_{0}\,,\end{array}\right.
\end{equation*}
and where $\mathbb{P}_2$ is the orthogonal projector from $L^2(\R^2)$ onto $\mc{L}^2$.

For arbitrary $p \in [1,\infty),$  $\mc{L}^p$ is a closed subspace of $L^p(\R^2).$ Hence we can define $\mathbb{P}_{p}$ the projector operators from $L^p(\R^2)$ onto $\mc{L}^p$ which coincide with $\mathbb{P}_2$ on $L^p(\R^2) \cap L^2(\R^2)$ (see e.g. \cite{Yun&Wang}). These projectors are  obviously continuous and satisfy $\mathbb{P}_{p}V=\mathbb{P}_{q}V$ for all $V\in L^p\cap L^q.$ In what follows, we omit the index $p$. 
We emphasize that the pressure does not appear in the abstract ODE. But, once a solution $V$  
is constructed, one shows the existence of a pressure such that \eqref{S1}--\eqref{S-Solideci} holds true (see the proof of Corollary 4.3 in \cite{Takahashi&Tucsnak04}). According to this, for sake of simplicity we omit to mention the pressure when considering solutions of \eqref{S1}--\eqref{S-Solideci}.

\medskip

Proposition 4.2 in \cite{Takahashi&Tucsnak04} shows that $A$ is a self-adjoint maximal monotone operator. Therefore, applying Hille-Yosida theorem (see e.g. \cite[Theorem 7.7]{Brezis}) yields global solutions to \eqref{S1}--\eqref{S-Solideci} for arbitrary initial
data $V_0 \in \mathcal L^2.$
Furthermore, there holds: 
\begin{equation*}
\|V(t)\|_{\mathcal L^2} \leq \|V_0\|_{\mathcal L^2}\,, \qquad \forall \, t \in \mathbb R^+ \,,
\end{equation*}
and there exists a constant $C$ such that 
\begin{equation*}
	\norm{\partial_t V(t)}_{\mathcal{L}^2}=\norm{A V(t)}_{\mathcal{L}^2} \leq \frac{C}{t} \norm{V_{0}}_{\mathcal{L}^2}.
\end{equation*}
Using the identity
\[
	\nu \norm{D(V)}_{L^2(\mc{F}_0)}^2 = \langle A V, V\rangle_{\mc{L}^2}
\]
for $V \in \mathcal{D}(A)$, see \cite[p.61]{Takahashi&Tucsnak04}, Lemma 4.1 in \cite{Takahashi&Tucsnak04} implies
\begin{equation*}
	\norm{\nabla V(t)}_{L^2(\mc{F}_0)} \leq \frac{C}{\sqrt{t}} \norm{V_{0}}_{\mathcal{L}^2}.
\end{equation*}
Hence, previous results in \cite{Takahashi&Tucsnak04} imply Theorem \ref{theo Stokes} when $q = p= 2$.

To generalize this result to arbitrary values for $p$ and $q$, we provide here an original decomposition of $V(t)$.

\bigskip

\subsection{Spherical-harmonic decomposition of $\mathcal L^p$ spaces}
To motivate the spherical-harmonic decomposition of $\mathcal L^p$, assume for instance that 
$V \in \mathcal L^p$  is smooth and denote $(\ell_V,\omega_V) \in \mathbb R^2 \times \mathbb R$
the only pair such that $V(x) = \ell_V + \omega_V x^{\perp}$ in $B_0.$   
As $V$ is divergence-free, there exists $\tilde \Psi \in \mathcal C^{\infty}(\mathbb R^2)$ such that  $V =\nabla^{\perp} \tilde \Psi.$
Fixing $\tilde \Psi(0) =0$ yields:
$$
\tilde \Psi(x) = \dfrac{1}{2}\omega_V |x|^2 + \ell_V \cdot x^{\perp}\,, \quad \text{in $\overline{B_0}\,.$}
$$
Consequently, introducing radial coordinates $(r,\theta)$ and expanding $\tilde\Psi$ in Fourier series:
$$
	\tilde \Psi(r,\theta) = \sum_{k =0}^{\infty} \Psi_k(r) \cos(k \theta) + \sum_{k=1}^{\infty} \Phi_k(r)\sin(k\theta), \quad \forall \, (r,\theta) \in (0,\infty) \times (-\pi,\pi)\,, 
$$
we observe that, setting $\ell_{V,1} = \ell_V\cdot e_1$, $\ell_{V,2} = \ell_V \cdot e_2$: where $e_1$ and $e_2$ is the canonical orthonormal basis of $\mathbb{R}^2$:
\begin{eqnarray*}
&&\Psi_0(r) = \dfrac{\omega_V}{2} r^2\,, \quad \Psi_1(r) = \ell_{V,2} r\,,  \quad  \Phi_1(r) = -\ell_{V,1} r \,,  \qquad \forall \, r \in (0,1)\,, \\
&&\Psi_k(r) = 0\,, \quad \Phi_k (r) = 0\,, \quad \forall \, k \geq 2\,, \qquad \forall \, r \in (0,1)\,.
\end{eqnarray*}
Then, informations on $\omega_V$ and $\ell_V$ are contained in the zero and first modes
of $\Psi$ respectively, so that these modes are handled separately from the others. 
In particular, we  focus on $\partial_r \Psi_0,\Phi_1,\Psi_1,$ that we denote by $W,\Phi,\Psi$ respectively
and regroup the other terms into a remainder.  In what follows, we still denote
$(r,\theta)$ radial coordinates and introduce $(e_r,e_{\theta})$ the associated local basis. 
Accordingly, we denote by $V_r$ and $V_{\theta}$ the radial and tangential components of a vector $V.$
To state our result precisely, we also introduce, for $p \in [1, \infty]$, the set  
$$
	L^p_{\sigma}(\mc{F}_0) = \{V \in \mc{L}^p,  \, V = 0 \text{ on } B_0 \}.
$$
Though this space contains functions defined on $\R^2$, we will often identify the elements of $L^p_{\sigma}(\mc{F}_0)$ with their restrictions on $\mc{F}_0$.

\begin{proposition} \label{prop_decompLp}
Let $p \in [1,\infty]$ and $V \in \mathcal L^p,$ then there exists a unique 4-uplet $(W, \Psi,\Phi,V_R)$
such that:
\begin{itemize}
\item[(i)] $V(r,\theta) = W(r) \min(r,1) e_{\theta}(\theta) + \nabla^{\perp} [\Psi(r) \cos(\theta)]+ \nabla^{\perp} [\Phi(r) \sin(\theta) ] + V_R(r,\theta),$
\item[(ii)] $W = W(r) \in L^p((0,\infty),r{\rm d}r)\,,$ and $W$ is constant on $(0,1)$:  $W(r) = \ell_W=\omega_V$ for $r \in (0,1)$.
\item[(iii)] $(\Psi,\Phi) = (\Psi(r), \Phi(r)) \in W^{1,p}_{\text{loc}}(0,\infty)$ with
$$
\int_{0}^{\infty} \left[ \left| \dfrac{\Psi(r)}{r} \right|^p  + \left| \dfrac{\Phi(r)}{r} \right|^p + |\partial_r \Psi(r)|^2 + |\partial_r \Phi(r)|^p \right] r{\rm d}r < \infty\,,
$$ 
and the functions $\Psi/r,\Phi/r,  \partial_r \Psi,\partial_r \Phi$ are constant on $(0,1)$: $\Psi(r)/r = \partial_r \Psi(r) = \ell_2 = \ell_{V,2}$ and $\Phi(r)/r = \partial_r \Phi(r) =- \ell_1= - \ell_{V,1}$ for $r \in (0,1)$.
\item[(iv)] $V_R  = V_R(x) \in L^{p}_{\sigma}(\mathcal F_0)$ and the following identities hold true: 
\begin{equation} \label{eq_vR}
\int_{0}^{2\pi} V_R(r,\theta) \cdot e_r \cos(\theta)\, \text{d$\theta$} = \int_{0}^{2\pi} V_R(r,\theta) \cdot e_r \sin(\theta)\, \text{d$\theta$} = 
\int_{0}^{2\pi} V_R(r,\theta) \cdot e_{\theta}\, \text{d$\theta$} = 0\,, \quad \forall \, r \in (1,\infty)\,.
\end{equation}
\end{itemize}
Furthermore, there exists a constant $C(p)$ depending only on $p$ such that 
\begin{eqnarray}
 \|W\|_{L^p((0,\infty),r{\rm d}r)} &+& \|V_R\|_{L^p(\mathbb R^2)}  \notag \\ 
&+& \|\partial_r \Psi\|_{L^p((0,\infty),r{\rm d}r)} + \|\Psi/r\|_{L^p((0,\infty),r{\rm d}r)}  \notag \\
&+&  \|\partial_r \Phi\|_{L^p((0,\infty),r{\rm d}r)}  +  \|\Phi/r\|_{L^p((0,\infty),r{\rm d}r)}
\leq C(p) \|V\|_{\mathcal{L}^p}\,. \label{eq_decompLp}
\end{eqnarray}
There exists also a constant $C(p)$ depending only on $p$ so that conversely:  
\begin{eqnarray}
\|V\|_{\mathcal{L}^p} &\leq&  C(p) \Big( \|W\|_{L^p((0,\infty),r{\rm d}r)} + \|V_R\|_{L^p(\mathbb R^2)} + \|\partial_r \Psi\|_{L^p((0,\infty),r{\rm d}r)} + \|\Psi/r\|_{L^p((0,\infty),r{\rm d}r)}  \notag \\
&&+  \|\partial_r \Phi\|_{L^p((0,\infty),r{\rm d}r)}  +  \|\Phi/r\|_{L^p((0,\infty),r{\rm d}r)}\Big)\,, \label{eq_reciproque}
\end{eqnarray}
and
\begin{eqnarray}
\|\nabla V\|_{L^p(\mathcal F_0)} &\leq&  C(p) \Big( \|\partial_r W\|_{L^p((1,\infty),r{\rm d}r)} +  \|W/r\|_{L^p((1,\infty),r{\rm d}r)}  
+ \|\nabla V_R\|_{L^p(\mathcal F_0)}   \notag\\ 
&&+ \|\partial_{rr} \Psi\|_{L^p((1,\infty),r{\rm d}r)} + \|\partial_r \Psi/r\|_{L^p((1,\infty),r{\rm d}r)} +  \| \Psi/r^2\|_{L^p((1,\infty),r{\rm d}r)}  \notag \\
&&+  \|\partial_{rr} \Phi\|_{L^p((1,\infty),r{\rm d}r)} + \|\partial_r \Phi/r\|_{L^p((1,\infty),r{\rm d}r)} +  \| \Phi/r^2\|_{L^p((1,\infty),r{\rm d}r)}\Big)\,.  \label{eq_reciproquegrad}
\end{eqnarray}
\end{proposition}

\begin{proof}
Let $p \in [1,\infty].$ We first note that, given $V_R \in L^{p}_{\sigma}(\mathcal F_0),$ it is possible to define by duality the functions: 
$$
r \mapsto \int_{0}^{2\pi} V_R(r,\theta) \cdot e_r \cos(\theta) \, \text{d$\theta$}\,, \quad 
r \mapsto \int_{0}^{2\pi} V_R(r,\theta) \cdot e_r \sin(\theta)\, \text{d$\theta$}\,, \quad 
r \mapsto \int_{0}^{2\pi} V_R(r,\theta) \cdot e_{\theta}\, \text{d$\theta$}  \,,
$$
on $(0,\infty)$. This yields $L^1_{\text{loc}}(1,\infty)$  functions which might satisfy \eqref{eq_vR}.
Also, once $\Phi,\Psi$ and $W,V_R$ are constructed with the regularity of (ii)--(iv), then (i) yields:
\begin{eqnarray}
V_r(r,\theta) &=&   \dfrac{\Psi(r)}{r}\sin(\theta) - \dfrac{\Phi(r)}{r} \cos(\theta) + V_R(r,\theta) \cdot e_{r}\,,  \label{eq_vr}\\
V_{\theta}(r,\theta) &=& W(r) \min(1,r) + \partial_r \Psi(r) \cos(\theta) + \partial_r \Phi(r) \sin(\theta) + V_R(r,\theta) \cdot e_{\theta}\,. \label{eq_vtheta}
\end{eqnarray}
This implies \eqref{eq_reciproque} and \eqref{eq_reciproquegrad}.
 
\medskip
 
To prove existence and uniqueness of $W,\Phi,\Psi,$ we assume $V \in \mathcal L^p $. 
With this further assumption, identities \eqref{eq_vr} and \eqref{eq_vtheta} together with \eqref{eq_vR} imply that the only possible candidates $W,\Phi,\Psi$ are the
following functions: 
\begin{eqnarray}
W(r)& := &\dfrac{1}{2\pi \min(1,r)} \int_{0}^{2\pi} V_{\theta}(r,\theta) \, \text{d$\theta$}\,,   \label{eq_w0}\\
\Phi(r)  &:= & - \dfrac{r}{\pi} \int_0^{2\pi} V_{r}(r,\theta) \cos(\theta) \, \text{d$\theta$}\,, \label{eq_phi1}\\
\Psi(r)  &:=& \dfrac{r}{\pi} \int_0^{2\pi} V_{r}(r,\theta) \sin(\theta) \, \text{d$\theta$}\,. \label{eq_psi1} 
\end{eqnarray}
Differentiating the formulas \eqref{eq_phi1}--\eqref{eq_psi1} and recalling that $V$ is divergence-free then yields:
$$
\partial_r \Phi(r) =   \dfrac{1}{\pi} \int_0^{2\pi} V_{\theta}(r,\theta) \sin(\theta) \, \text{d$\theta$}, 
\quad 
\partial_r \Psi(r) =   \dfrac{1}{\pi} \int_0^{2\pi} V_{\theta}(r,\theta) \cos(\theta) \, \text{d$\theta$}\,,
$$
(where these identities have to be understood in the sense of $\mathcal{D}'(0,\infty)$) and we have then 
\begin{equation}\label{eq_VR}
V_R := V -  ( W\min(1,r) e_{\theta} + \nabla^{\perp} [\Psi \cos(\theta)]+ \nabla^{\perp} [\Phi \sin(\theta)] ).
\end{equation}

\medskip

In the ball, we deduce from these definitions and from $V=\ell_V + \omega_V x^{\perp}$ that for all $r\in (0,1)$:
\[
W(r)=\omega_V, \quad \Psi(r)/r = \partial_r \Psi(r) = \ell_{V,2}, \quad \Phi(r)/r = \partial_r \Phi(r) = - \ell_{V,1}, \quad V_R=0.
\]

From the definition \eqref{eq_w0} of $W,$ Jensen inequality implies that:
$$
|W(r)|^p \leq \frac{1}{2 \pi}  \int_0^{2\pi} |V_{\theta}|^p(r,\theta) \, \text{d$\theta$} \qquad \forall \, r >1.
$$
Combining with the remark that $W(r) = \omega_V$ for $r<1,$
we obtain there exists a constant $C$ for which $\|W\|_{L^p((0,\infty),r{\rm d}r)} \leq C(p) \|V\|_{\mathcal{L}^p}.$
Similarly, we prove that 
\begin{align*}
	&\|\partial_r \Psi\|_{L^p((0,\infty),r{\rm d}r)} + \|\Psi/r\|_{L^p((0,\infty),r{\rm d}r)}  \leq C(p)  \|V\|_{\mathcal{L}^p}\,,
	\\
	& \|\partial_r \Phi\|_{L^p((0,\infty),r{\rm d}r)} + \|\Phi/r\|_{L^p((0,\infty),r{\rm d}r)}  \leq C(p)  \|V\|_{\mathcal{L}^p}\,.
\end{align*}
Finally, straightforward computations yield that $V_R$ is divergence-free, vanishes in $B_0$ and satisfies \eqref{eq_vR}. 
As, combining previous estimates and  \eqref{eq_VR} also yields  that $V_R \in \mathcal L^p,$ we conclude that $V_R \in L^{p}_{\sigma}(\mathcal F_0)$ and that \eqref{eq_decompLp} holds true. This ends the proof of Proposition \ref{prop_decompLp}.
\end{proof}
 
 Of course, Proposition \ref{prop_decompLp} and Theorem \ref{theo Stokes} hold true if we replace the norms $\| \cdot \|_{\mathcal{L}^p}$ by $\| \cdot \|_{L^p(\R^2)}$. We have chosen to keep the notations of Takahashi and Tucsnak \cite{Takahashi&Tucsnak04}, where they prove that $A$ is a self-adjoint maximal monotone operator for the scalar product \eqref{eq_ps}.

Let us also emphasize that in Proposition \ref{prop_decompLp}, all the functions $(W, \partial_r \Psi, \Psi/r, \partial_r \Phi, \Phi/r, v_R)$ defined on $\R^2$ are constant on $B_0$, so that we can identify these extensions with the pairs given by their restriction to $\mc{F}_0$, denoted will small caps, and their mean value on the ball $B_0$, denoted $\ell$. For instance, we will write $W \doteq (w, \ell_W)$. Moreover, in all the text, we will identify $\ell_{W}(t)=\ell_{W(t)}$.

\subsection{Decomposition in spherical harmonics of the Stokes semigroup}

In the rest of this section and in Section \ref{sect2}, we only consider smooth initial data, namely $V_0 \in \mathcal{L}^2 \cap \mathcal{C}^{\infty}_c(\mathbb R^2)$. Indeed, it is sufficient to show Theorem \ref{theo Stokes} for smooth initial data, because $\mathcal{L}^2 \cap \mathcal{C}^{\infty}_c(\mathbb R^2)$ is dense in $\mathcal{L}^q$, for the $\mathcal{L}^q$ norm with $q\in (1,\infty)$. So the estimates \eqref{Lp-Lq}-\eqref{est grad 2} could be extended, thanks to the linearity of the Stokes system.

In this paragraph, we prove that the spherical-harmonic decomposition of $\mathcal L^p$ is well-adapted to compute
solutions of \eqref{S1}--\eqref{S-Solideci}. We prove:

\begin{proposition} \label{prop_calculsysteme}
Given $V_0 \in \mathcal{L}^2 \cap \mathcal{C}^{\infty}_c(\mathbb R^2)$, the spherical-harmonic decomposition provided by Proposition \ref{prop_decompLp} of the unique solution 
 $V \in C([0,\infty);\mathcal{L}^2)$ of \eqref{S1}--\eqref{S-Solideci} satisfies:
\begin{itemize}
	\item $W  \doteq ( w, \ell_W)$, where $w \in \mathcal{C}([0,\infty),  L^2((1,\infty),r{\rm d}r)) \cap \mathcal{C}^{\infty}((0,\infty) \times [1,\infty))$ verifies: 
	\begin{eqnarray}
		\partial_t w + \nu \left( - \frac{1}{r} \partial_r (r \partial_r w)  + \frac{1}{r^2} w \right)=0 
			&& 
		\text{ for } (t,r) \in (0, \infty) \times (1, \infty); 
		\label{eq_w0first}
		\\
		w(t,1)=\ell_W(t)
			&&
			\text{ for } t \in (0, \infty);
		\label{eq_w0-2}
		\\
		\ell_W'(t) = \dfrac{2 \nu \pi }{\mathcal J} (\partial_r w(t,1)-w(t,1))
			&&
		\text{ for } t \in (0, \infty) ; \label{eq_w0last}
	\end{eqnarray}
	
	\item  $\partial_r \Psi \doteq (\partial_r \psi, \ell_2)$ and $\Psi/r \doteq (\psi/r, \ell_2)$, where $\partial_r\psi,\psi/r \in \mathcal{C}([0,\infty) ; L^2((1,\infty),r{\rm d}r)),$ $\psi \in  \mathcal{C}^{\infty}((0,\infty) \times [1,\infty))$ 
	and there exists a pressure $q_1 \in \mathcal{C}^{\infty}((0,\infty) \times [1,\infty))$ satisfying $\partial_r q_1 \in \mathcal{C}((0,\infty); L^2((1,\infty),r{\rm d}r))$ such that:
	\begin{eqnarray}
		\partial_t \psi + \nu \left( - \frac{1}{r} \partial_r (r \partial_r \psi)  + \frac{1}{r^2}  \psi\right) = -r \partial_r q_1 
			&& 
		\text{ for } (t,r) \in (0, \infty) \times (1, \infty);
		\label{eq_1:psi}
		\\
		\partial_t \partial_r \psi + \nu  \partial_r\left(  - \frac{1}{r} \partial_r (r \partial_r \psi)  + \frac{1}{r^2}  \psi\right)=-\frac{q_1}{r} 
			&&
		\text{ for } (t,r) \in (0, \infty) \times (1, \infty)\,;
		 \label{eq_2:psi}
		 \\
		\psi(t,1)=\partial_r \psi(t,1)=\ell_2(t)
			&&
		\text{ for } t \in (0, \infty);
		\label{eq_3:psi}
		\\
		\frac{m}{\pi} \ell'_2(t)= -q_1(t,1) -   \nu \left(- \frac{1}{r} \partial_r (r \partial_r \psi)  + \frac{1}{r^2}  \psi\right)(t,1)
			&& 
		\text{ for } t \in (0, \infty);\,
		\label{eq_4:psi}
	\end{eqnarray}
	
	 \item  $\partial_r \Phi \doteq (\partial_r \phi, -\ell_1)$ and $\Phi/r \doteq (\phi/r, -\ell_1)$, where $(\partial_r\phi,\phi/r) \in \mathcal{C}([0,\infty) ; L^2((1,\infty),r{\rm d}r)),$ $\phi \in  \mathcal{C}^{\infty}((0,\infty) \times [1,\infty))$ 
	and there exists a pressure $p_1 \in \mathcal{C}^{\infty}((0,\infty) \times [1,\infty))$ satisfying $\partial_r p_1 \in \mathcal{C}((0,\infty); L^2((1,\infty),r{\rm d}r))$ such that:
	\begin{eqnarray}
		\partial_t \phi + \nu \left( -  \frac{1}{r} \partial_r (r \partial_r \phi)  + \frac{1}{r^2}  \phi\right) = r \partial_r p_1 
			&& 
		\text{ for } (t,r) \in (0, \infty) \times (1, \infty);
		\label{eq_1:phi}
		\\
		\partial_t \partial_r \phi +\nu  \partial_r\left(- \frac{1}{r} \partial_r (r \partial_r \phi)  + \frac{1}{r^2}  \phi\right)=\frac{p_1}{r} 
			&&
			\text{ for } (t,r) \in (0, \infty) \times (1, \infty);
		\label{eq_2:phi}
		\\
		\phi(t,1)=\partial_r \phi(t,1)=-\ell_1(t)
			&&
			\text{ for } t \in (0, \infty);
		\label{eq_3:phi}
		\\
		\frac{m}{\pi} \ell'_1(t)= -p_1(t,1)+  \nu \left(- \frac{1}{r} \partial_r (r \partial_r \phi)  + \frac{1}{r^2}  \phi\right)(t,1)
		&& 
			\text{ for } t \in (0, \infty);
		\label{eq_4:phi}
	\end{eqnarray} 

	\item $V_R \doteq ( v_R, 0)$, where $v_R \in \mathcal{C}([0,\infty), L^2(\mathcal F_0)) \cap \mathcal{C}^{\infty}((0,\infty) \times \bar{\mathcal F}_0)$ and
	there exists $p_R \in \mathcal{C}^{\infty}((0,\infty)\times \mathcal{F}_0) $ such that:
	\begin{eqnarray}
		\partial_t v_R -  \nu \Delta v_R + \nabla p_R =0 && \text{ for } (t,x) \in (0, \infty)\times \mc{F}_0\,; 
		\label{eq_vR1}
		\\
		\div \, v_R=0 
			&&
		\text{ for } (t,x) \in (0, \infty)\times \mc{F}_0\,;
		 \\
		v_R(t,x)=0
		&&
		\text{ for } (t,x) \in (0, \infty)\times \partial B_0. 
		 \label{eq_vR3}
	\end{eqnarray}
\end{itemize}
\end{proposition}

We postpone the proof of Proposition \ref{prop_calculsysteme} to Appendix \ref{App_calculsysteme}. This mainly consists of tedious computations. 

\medskip
The main interest of Proposition \ref{prop_calculsysteme} is that it reduces the study of the Stokes semigroup to the study of scalar equations for the modes involving non-trivial boundary conditions and one Stokes equation with homogeneous boundary conditions on the obstacle. Indeed:
\begin{itemize}
	\item System \eqref{eq_w0first}--\eqref{eq_w0last} is a scalar heat equation with dynamic boundary condition.
	\item  System \eqref{eq_vR1}--\eqref{eq_vR3} is a Stokes equation with a fixed obstacle and Dirichlet boundary condition.
	\item  Systems \eqref{eq_1:psi}--\eqref{eq_4:psi} and \eqref{eq_1:phi}--\eqref{eq_4:phi} are similar one to each other. Actually, $(\phi, p_1,\ell_1)$ solves \eqref{eq_1:phi}--\eqref{eq_4:phi} if and only if $(-\phi, p_1,\ell_1)$ solves \eqref{eq_1:psi}--\eqref{eq_4:psi}. System  \eqref{eq_1:psi}--\eqref{eq_4:psi} involves two scalar heat equations \eqref{eq_1:psi}--\eqref{eq_2:psi} which contain the term $q_1$ reminiscent from the pressure. It also involves intricate boundary conditions \eqref{eq_3:psi}--\eqref{eq_4:psi} which couples Dirichlet ($\psi(t,1)$), Neumann ($\partial_r \psi(t,1)$) and dynamic (see \eqref{eq_4:psi}) boundary conditions.
\end{itemize}
Whereas systems \eqref{eq_w0first}-\eqref{eq_w0last} and \eqref{eq_vR1}-\eqref{eq_vR3} are classical and widely studied in the literature, systems \eqref{eq_1:psi}-\eqref{eq_4:psi} and \eqref{eq_1:phi}-\eqref{eq_4:phi} do not seem known and are the main challenge of our study.
Actually, we show that systems \eqref{eq_1:psi}-\eqref{eq_4:psi} and \eqref{eq_1:phi}-\eqref{eq_4:phi} reduce to a heat equation with dynamic boundary conditions. Concerning $\psi$ for instance, our strategy consists of removing the pressure term and reduce \eqref{eq_1:psi}-\eqref{eq_4:psi} to a scalar equation for the new unknown:
\begin{equation}
	\label{Def-z1}
	Z(r) := \partial_{r}\Psi(r) + \dfrac{ \Psi(r)}{r} =  \dfrac{1}{r} \partial_r[r\Psi(r)]\,, \quad \forall \, r \in (0,\infty),
\end{equation}
which, in particular, is a constant function on the ball $B_0$, denoted by $\ell_Z$, and for which we have
\begin{equation}
	\label{z-on-Ball}
	 Z(r) = \ell_Z \quad \forall \, r \in (0,1), \qquad \ell_Z = 2 \ell_2.
\end{equation}
Note that, using the definition \eqref{Def-z1} of $Z$ and the fact the $\Psi(r)/r= \ell_2$ on the unit ball, identity \eqref{z-on-Ball} immediately implies $\partial_r \Psi = \ell_2$ on the unit ball, thus being completely compatible with the boundary conditions \eqref{eq_3:psi}.

Indeed, using this new unknown, we get: 

\begin{proposition} \label{prop_calculz}
Given $V_0 \in \mathcal{L}^2 \cap \mathcal{C}^{\infty}_c(\mathbb R^2),$ let $(W,\Psi,\Phi,V_R)$ be the spherical-harmonic decomposition given by Proposition \ref{prop_decompLp}
of the solution $V \in \mathcal{C}([0,\infty);\mathcal{L}^2)$ of  \eqref{S1}--\eqref{S-Solideci}. Then, setting $Z \doteq ( z, \ell_Z)$ as in \eqref{Def-z1}, 
(or $Z= -  \partial_r ( r \Phi)/r$), 
\begin{itemize}
\item $z \in \mathcal{C}([0,\infty); {L}^2((1,\infty),r{\rm d}r)) \cap \mathcal{C}^{\infty}((0,\infty) \times [1,\infty))$ 
\item $(z, \ell_Z)$ is a solution to:
\begin{eqnarray}
	\label{eq_1:z}
		\partial_t z - \nu \Bigl(\partial_{rr}+\frac{1}{r}\partial_r \Bigl) z =0 
		&& 
		\text{ for } (t,r) \in (0, \infty) \times (1, \infty);
		\\
		z(t,1) = \ell_Z(t)
		&& 
		\text{ for } t \in (0, \infty);
		\\
		\ell_Z'(t)= \alpha_0 \nu \partial_r z(t,1)
		&& 
		\text{ for } t \in (0, \infty); 
	\label{eq_2:z}
\end{eqnarray}
with 
\begin{equation}
	\label{Alpha-0}
	\alpha_0=\frac{4 \pi}{\pi + m }.
\end{equation} 
\end{itemize}
\end{proposition}

\begin{proof}
Up to a change of sign, we focus on $Z =  \partial_r ( r \Psi)/r.$  Thanks to the regularity proved in Proposition \ref{prop_calculsysteme}, we have $(\partial_r \psi,\psi/r) \in\mathcal{C}([0,\infty) ; L^2((1,\infty),r{\rm d}r)).$ Consequently $z = \partial_r \psi + \psi/r$ enjoys the same regularity. The smoothness of $z$ is straightforward.

\smallskip

Differentiating \eqref{eq_1:psi} with respect to $r$ and subtracting \eqref{eq_2:psi}, the pressure $q_1(t)$ satisfies, for each time $t >0$:
\[
	-\frac{1}{r}\partial_r (r \partial_r q_1) + \frac{1}{r^2} q_1 =0, \quad \text{ for } r \in (1, \infty).
\]
Hence $q_1(t,r)=\alpha_1(t)r+\frac{\beta_1(t)}{r}$. Of course, the condition $\partial_r  q_1 \in L^2((1, \infty), r \, {\rm d}r) $  implies that:
\begin{equation}
	\label{Pressure-q-1}
	q_1(t,r)=\frac{\beta_1(t)}{r},
\end{equation}
 and therefore, for all $t >0$ and $r \geq 1$,
\[
	- \partial_r q_1= \frac{q_1}r.
\]
With this identity, the pressure can be removed simply by adding  \eqref{eq_2:psi} to $1/r$ times \eqref{eq_1:psi}:
\[
	\partial_t\left[ \left(\partial_r + \frac{Id}r\right)\psi \right] + \nu \left(\partial_r + \frac{Id}r\right)\left(- \frac{1}{r} \partial_r (r \partial_r \psi) + \frac{1}{r^2} \psi\right)=0.
\]
Using \eqref{Def-z1},
 \begin{equation}
 	\label{Identity-z-psi}
 	\partial_r z = - \left(- \frac{1}{r} \partial_r (r \partial_r \psi) + \frac{1}{r^2} \psi\right),
\end{equation}
and the new variable $z$ in \eqref{Def-z1} solves \eqref{eq_1:z}.

Concerning the boundary conditions, \eqref{eq_3:psi} reads as
\[
	z(t,1)=2\ell_2(t) = \ell_Z(t)
\]
and, using \eqref{Identity-z-psi} and \eqref{Pressure-q-1}, \eqref{eq_4:psi} yields
\[
	\frac{m}{\pi} \ell'_2(t)= -\beta_1(t)+ \nu \partial_r z(t,1).
\]

Moreover, still using \eqref{Identity-z-psi} and \eqref{Pressure-q-1}, \eqref{eq_1:psi} for $r=1$ and \eqref{eq_3:psi} gives 
\[
	\ell'_2(t) - \nu \partial_r z(t,1)=\beta_1(t).
\]

Combining the previous equations, $(z, \ell_Z)$ solves \eqref{eq_1:z}-\eqref{eq_2:z}.
\end{proof}

\begin{remark} \label{rem_defz}
In what follows, given $V \doteq ( v,\ell,\omega)$ a solution to \eqref{S1}--\eqref{S-Solideci} on $(0,\infty),$ we keep the convention:
\begin{equation*}
Z_{\Psi}(t,r) :=  \partial_{r}\Psi(t,r) + \dfrac{ \Psi(t,r)}{r}\,, \ Z_{\Phi}(t,r) := - \left(  \partial_{r}\Phi(t,r) + \dfrac{ \Phi(t,r)}{r}\right) \,, \ \forall (t,r) \in [0,\infty) \times (0,\infty)\,. 
\end{equation*}
We emphasize that, for $t >0$, $V(t,\cdot)$ has continuous normal and tangential traces through $\partial B_0$, and thus then $Z_{\Phi}(t,\cdot)$ and $Z_{\Psi}(t,\cdot)$
have continuous traces through the interface $r=1.$
\end{remark}

\begin{remark}
	As we recalled in the introduction, the classical approach would rather consist in the elimination of the pressure in the Navier Stokes system by taking
the curl of the Navier Stokes equation,  yielding that way an equation for the vorticity of the velocity-field. {\em But this is not the method we choose here}. 
Indeed, in an exterior domain, one should complete the vorticity equation, and this would yield non-dissipative boundary conditions of Robin type.
\end{remark}

\subsection{Some elliptic problems}
To conclude this section, we prove some technical lemmas that will be useful later on. Indeed, in order to compute the decay of the Stokes semigroup, we study the decay
 of solutions to the heat equation \eqref{eq_1:z}-\eqref{eq_2:z}. This gives the decay of the new
 unknown $z$ whether it is computed with respect to $\phi$ or $\psi.$ However, to our 
 purpose, we need then to invert the definition of $z$ in order to get also the decay of $\phi$
 and $\psi$ in suitable spaces. This is the content of the following proposition:
 
 \begin{proposition} \label{prop_z2psi}
 Given $p \in (1,\infty]$ and $(z,\ell) \in L^p((1,\infty),r{\rm d}r) \times \mathbb R,$ there exists a unique $\psi \in W^{1,p}_{loc}(1,\infty)$ solution 
 to the following boundary value problem:
 \begin{eqnarray}
	 \partial_r \psi(r) + \dfrac{\psi(r)}{r} &= & z(r)\,, \qquad \text{ for } r \in  (1, \infty)\,, \label{eq_z2psi1} \\
	\psi(1) &= & \ell\,,   \label{eq_z2psi2}
\end{eqnarray}
 and there exists a constant $C(p)$ depending only on $p$ for which:
 \begin{equation}
 \|\partial_r \psi\|_{L^p((1,\infty),r{\rm d}r)} + \left\| \dfrac{\psi}{r}\right\|_{L^p((1,\infty),r{\rm d}r)} 
  \leq  C(p)\left(   \|z\|_{L^p((1,\infty),r{\rm d}r)}  + |\ell |\right) \,.  \label{est_z2psi}
 \end{equation}
 \end{proposition}

 \begin{proof}
	 Let $p \in (1, \infty]$ and $(z,\ell)$ satisfy the assumptions of Proposition \ref{prop_z2psi}.  It is straightforward that the unique solution to \eqref{eq_z2psi1}-\eqref{eq_z2psi2} reads:
 $$
 \psi(r) = \dfrac{\ell}{r} + \dfrac{1}{r} \int_{1}^{r} s z(s) \, {\rm d}s\,, \quad \forall \, r \geq 1\,.
 $$

 If $p=\infty$, we establish easily \eqref{est_z2psi} from this formula.
 If $p \in (1, \infty)$, up to a regularizing argument we skip for conciseness, we multiply \eqref{eq_z2psi1} by $|\psi|^{p-2}\psi/r^{p-1}$ on $[1,R],$
for arbitrary $R > 1$:
\begin{eqnarray*}
	\int_{1}^R z |\psi|^{p-2} \frac{\psi}{r^{p-1}}  r \, dr 
	&=&  \left[  \dfrac{|\psi(r)|^p}{r^{p-2}}  \right]_1^R - (p-1) \left( 
	 \int_{1}^R   \frac{\partial_{r} |\psi|^{p}}{pr^{p-2}} \, dr -   \int_{1}^R  \frac{|\psi|^{p}}{r^{p}} r\, dr \right)
	\\
	&=&
	\frac{1}{p} \left[  \frac{ |\psi(r)|^{p}}{r^{p-2}}  \right]_1^R + 2 \left( 1 - \dfrac{1}{p}\right)   \int_{1}^R   \frac{ |\psi|^{p}}{r^{p-1}} \, dr\,, \\
	& \geq &  - \frac{|\ell |^p}{p} +   2 \left( 1 - \dfrac{1}{p}\right)   \int_{1}^R   \frac{ |\psi|^{p}}{r^{p-1}} \, dr\,.
\end{eqnarray*}
Hence, for all $p \in (1, \infty)$,
\begin{eqnarray*}
	 \left\| \dfrac{\psi}{r}\right\|_{L^p((1,R),r{\rm d}r)}^p 
	 & \leq &
	 C(p)  \left\| z\right\|_{L^p((1,\infty),r{\rm d}r)}	 \left\| \left(\dfrac{\psi}{r}\right)^{p-1}\right\|_{L^{p'}((1,R),r{\rm d}r)}+C(p)  |\ell |^p
	 \\
	 & \leq &
	 C(p) 	 \left\| z\right\|_{L^p((1,\infty),r{\rm d}r)}	 \left\| \dfrac{\psi}{r}\right\|_{L^{p}((1,R),r{\rm d}r)}^{p-1} + C(p) |\ell |^p.
\end{eqnarray*}
This yields
$$
	\left\| \dfrac{\psi}{r}\right\|_{L^p((1,R),r{\rm d}r)} 
	  \leq  C(p)\left(   \|z\|_{L^p((1,\infty),r{\rm d}r)}  + |\ell |\right).
$$
Letting then $R \to \infty$ we obtain \eqref{est_z2psi}.
 \end{proof}

Let us now state another elliptic estimate that will be useful in the following:
 \begin{proposition} \label{prop_drz2z}
 Let $p \in (1,\infty)\setminus \{2\}$ and assume that  $z \in L^p((1,\infty),r{\rm d}r)$ and $\partial_r z \in  L^p((1,\infty),r{\rm d}r).$ 
 There exists a constant $C(p)$ depending only on $p$ such that:
 \begin{equation} \label{est_drz2z}
 \left \|  \dfrac{z}{r} \right\|_{L^p((1,\infty),r{\rm d}r)} \leq C(p) \left( \|\partial_r z\|_{L^p((1,\infty),r{\rm d}r)} + \varepsilon_p |z(1)| \right)
 \end{equation}
 where $\varepsilon_p = 1$ if $p> 2$ and $\varepsilon_p  = 0$ if $p < 2.$
 \end{proposition}
 
 \begin{proof} 
 As $z$ belongs to $W^{1,p}_{loc}(1,\infty)$, we infer that it is continuous and we integrate by parts on $[1,R]$:
 \begin{eqnarray*}
	\Bigl\| \frac{z}{r} \Bigl\|_{L^p((1,R),r{\rm d}r)}^p 
	&=& 
	- \frac1{p-2} \int_1^R |z|^p \partial_r \Bigl(\frac{1}{r^{p-2}}\Bigl) \, {\rm d}r= -\frac1{p-2} \Bigl[ \frac{|z|^p}{r^{p-2}} \Bigl]_1^R +\frac{p}{p-2}  \int_1^R  \frac{  |z|^{p-2} z \partial_r z}{r^{p-1}} \, r{\rm d}r\\
&\leq& \frac1{p-2} \Bigl( |z(1)|^p -\frac{|z(R)|^p}{R^{p-2}} \Bigl) + \frac{p}{|p-2|}   \|\partial_r z\|_{L^p((1,\infty),r{\rm d}r)} \Bigl\| \frac{z}{r} \Bigl\|_{L^p((1,R),r{\rm d}r)}^{p-1} .
\end{eqnarray*}
Then, the following depends on the sign of $p-2$:
\begin{itemize}
 \item if $p>2$, we directly have that
 \[
 \Bigl\| \frac{z}{r} \Bigl\|_{L^p((1,R),r{\rm d}r)}^p \leq  \frac1{p-2} |z(1)|^p + \frac{p}{p-2}   \|\partial_r z\|_{L^p((1,\infty),r{\rm d}r)} \Bigl\| \frac{z}{r} \Bigl\|_{L^p((1,R),r{\rm d}r)}^{p-1} ,\]
 which gives \eqref{est_drz2z} with $\varepsilon_p =1$.
 \item if $p<2$, we get 
 \[
 \Bigl\| \frac{z}{r} \Bigl\|_{L^p((1,R),r{\rm d}r)}^p \leq  \frac1{2-p} \frac{|z(R)|^p}{R^{p-2}} + \frac{p}{2-p}   \|\partial_r z\|_{L^p((1,\infty),r{\rm d}r)} \Bigl\| \frac{z}{r} \Bigl\|_{L^p((1,R),r{\rm d}r)}^{p-1} .
 \]
  To establish \eqref{est_drz2z} with $\varepsilon_p =0$, it is sufficient to find a sequence $R_n \to \infty$ such that  $(|z(R_n)|^p R_n)$ tends to zero. This can obviously be done since $r \mapsto r |z(r)|^p$ is assumed to belong to $L^1(1, \infty)$.
\end{itemize}
\end{proof}

We finally provide elliptic estimates that will be useful when getting estimates on the $0$-mode:

 \begin{proposition} \label{prop_ellipticw0}
Let $p \in (1,\infty)$ and assume that $w \in  L^p((1,\infty),r{\rm d}r)$ satisfies
 \begin{align*}
	& \partial_{rr} w(r) + \dfrac{\partial_r w(r)}{r}-  \dfrac{w(r)}{r^2} =  f(r)\,, \qquad \text{ for } r\in  (1, \infty)\, ;  \\
	& \partial_r w (1) - w(1) = a,  \qquad   w(1) = b,    \label{eq_w2}
\end{align*}
for some $f \in L^{p}((1,\infty),r{\rm d}r)$, $a,\,b$ in $\mathbb R$.
Then, there  exists a constant $C(p)$ depending only on $p$ for which:
\begin{equation} \label{est_wnabla2}
 \|\partial_{rr} w \|_{L^p((1,\infty),r{\rm d}r)}   +\Big\|\frac{\partial_r w}{r}  -\frac{w(r)}{r^2} \Big\|_{L^p((1,\infty),r{\rm d}r)} \leq  C(p)\left(   \|f\|_{L^p((1,\infty),r{\rm d}r)}  +  |a|  \right). \,
\end{equation}
Furthermore, if $p \neq 2$, 
\begin{equation} \label{est_w-autres2}
	\Big\|\frac{\partial_r w}{r} \Big\|_{L^p((1,\infty),r{\rm d}r)}+ \Big\|\frac{w}{r^2}\Big\|_{L^p((1,\infty),r{\rm d}r)}   \leq  C(p)\left(   \|f\|_{L^p((1,\infty),r{\rm d}r)}  +  |a|  + \varepsilon_p |b | \right),
\end{equation}
with $\varepsilon_p =1$ if $p >2$ and $\varepsilon_p = 0$ if $p <2$.
\end{proposition} 

\begin{proof}
We define $\tilde{w} = w(r)/r$ for $r \geq 1.$
Then, $\tilde{w}$ satisfies   
 \begin{eqnarray}
r\partial_{rr} \tilde{w}(r) + 3 \partial_r \tilde{w}(r)  &= & f(r)\,, \qquad \text{ for } r\in  (1, \infty)\, ;    \label{eq_wtilde1}\\
\partial_r \tilde{w} (1) &= & a \,,  \label{eq_wtilde2}
\end{eqnarray}
Following the method of the proof of Proposition \ref{prop_z2psi}, we multiply \eqref{eq_wtilde1} by $|\partial_r \tilde{w}|^{p-2} \partial_r \tilde{w}$ on $[1,R].$
After integration by parts, this yields:
\begin{eqnarray*}
\int_{1}^R f|\partial_r \tilde{w}|^{p-2} \partial_r \tilde{w} \,  r {\rm d}r &=&   \frac{1}{p}\left[ r^2 |\partial_r \tilde{w}|^p \right]_1^R + \left(3- \dfrac{2}{p} \right) \int_{1}^{R} |\partial_r \tilde{w}|^p \, r{\rm d}r\,, \\
& \geq & - \frac{|a|^p}{p} +  \left(3- \dfrac{2}{p} \right) \int_{1}^{R} |\partial_r \tilde{w}|^p \, r{\rm d}r\,.
\end{eqnarray*}
We conclude that:
\begin{equation} \label{est_wtilde}
\|\partial_r \tilde{w}\|_{L^p((1,\infty),r{\rm d}r)} \leq C(p) \left(   \|f\|_{L^p((1,\infty),r{\rm d}r)}  +  |a|  \right). 
\end{equation}
Expanding $\partial_r \tilde{w},$ we remark that  $\partial_{rr} w = f - \partial_r \tilde{w}$ so that \eqref{est_wtilde} implies
\eqref{est_wnabla2}. 

If $p \neq 2$, we then apply Proposition \ref{prop_drz2z} to $\partial_r \tilde w$. This yields 
\begin{equation}
\label{estimate_wtilde-r2}
		  \Big\|\frac{\tilde w}{r} \Big\|_{L^p((1,\infty),r{\rm d}r)} =  \Big\|\frac{w}{r^2} \Big\|_{L^p((1,\infty),r{\rm d}r)}   \leq  C(p)\left(   \|f\|_{L^p((1,\infty),r{\rm d}r)}  +  |a|  + \varepsilon_p |b| \right).
\end{equation}
Since $\partial_r \tilde w = \partial_r w /r - w /r^2$, estimates \eqref{est_wtilde} and \eqref{estimate_wtilde-r2} immediately yield \eqref{est_w-autres2}.
\end{proof}

\section{Study of solutions to (\ref{S1})--(\ref{S-Solideci}) }\label{sect2}

The ultimate goal of this section is to prove Theorem \ref{theo Stokes} and Theorem \ref{thm_AsymptoticStokes}. In all this section, we assume that $\nu = 1$ for simplicity. 
This can be done without loss of generality by setting $(V_\nu(t,x),P_\nu(t,x)):=(V(t/\nu ,x),P(t/\nu,x)/\nu)$.
Because of the computations we presented in the previous section, we first analyze separately the decay of solutions
to the Stokes equation with a fixed obstacle and then, we compute the long-time behavior of solutions to both heat equations with dynamic boundary conditions.
We conclude by combining all these computations.

\subsection{Decay of solutions to (\ref{eq_vR1})--(\ref{eq_vR3})} \label{sec_decayvR}

System \eqref{eq_vR1}--\eqref{eq_vR3} has already been studied in the frame of $L^{p}_{\sigma}(\mathcal F_0)$ spaces \cite[Theorem 1.2]{DS99a}, \cite{DS99}:
\begin{theorem} \label{Thm-Dan-Shibata} 
For each $q\in (1,\infty)$, the Stokes operator of the linear problem \eqref{eq_vR1}-\eqref{eq_vR3} generates a semigroup $S_{R}(t)$ on $L^q_{\sigma}(\mc{F}_{0})$. Moreover, this semigroup satisfies the following decay estimates for $v_R(t,\cdot)=S_{R}(t)v_{R}(0,\cdot)$:

$\bullet$ For  $p \in [ q , \infty]$, there exists $K_{1,R} = K_{1,R}(p,q) >0$ such that for every $v_R(0,\cdot) \in L^q_{\sigma}(\mc{F}_{0})$,
	\begin{equation}\label{Lp-LqR}
		\|v_R(t,\cdot)\|_{L^p_\sigma( \mathcal{F}_0)} \leq K_{1,R}\, t^{\frac{1}{p} - \frac{1}{q}}\|v_R(0,\cdot)\|_{L_\sigma^q( \mathcal{F}_0)}\,,
		\qquad \text{for all}\quad  t>0.
	\end{equation} 

$\bullet$ If $q \leq 2$, for $p \in [q, 2]$, there exists $K_{2,R} = K_{2,R}(p,q)>0$ such that for every $v_R(0,\cdot) \in L^q_{\sigma}(\mc{F}_{0})$,
	\begin{equation}\label{est grad 1R}
		 \|\nabla v_R(t,\cdot)\|_{L^p( \mathcal{F}_0)} \leq K_{2,R}\, t^{-\frac{1}{2} + \frac{1}{p} - \frac{1}{q}}\|v_R(0,\cdot)\|_{L^q_\sigma( \mathcal{F}_0)}\,,
	\qquad \text{for all}\quad  t>0.
	\end{equation}

$\bullet$ For  $p \in [\max\{2, q\}, \infty)$, there exists $K_{3,R} =K_{3,R}(p,q) >0$ such that for every $v_R(0,\cdot) \in L^q_{\sigma}(\mc{F}_{0})$,
	\begin{equation}\label{est grad 2R}
		 \|\nabla v_R(t,\cdot)\|_{L^p( \mathcal{F}_0)} \leq 
			\left\{ \begin{array}{ll} 
				 K_{3,R} \, t^{-\frac{1}{2} + \frac{1}{p} - \frac{1}{q}}\|v_R(0,\cdot)\|_{L^q_\sigma( \mathcal{F}_0)}\,, &\qquad \text{for all}\quad  0<t<1, \\
				 K_{3,R} \, t^{ - \frac{1}{q}}\|v_R(0,\cdot)\|_{L^q_\sigma( \mathcal{F}_0)}\,, &\qquad \text{for all}\quad  t \geq 1.
			 \end{array}\right.
	\end{equation} 
\end{theorem}
For localized initial data it is possible to obtain a much sharper description of the long-time behavior of $v_R$ 
by following the spirit of our spherical-harmonic decomposition. To this end, we need a general result on the decay of solutions
to heat equations with dynamic boundary conditions. This result is detailed in the following subsection. 
So, we postpone the more precise computation of the long-time behavior of $v_R$ to the end of this section (see Theorem  \ref{thm_vRfirstorder}).

\subsection{Semigroup estimates.} \label{sec_sgest}

We proceed with the computation of the long-time behavior of solutions to \eqref{eq_w0first}--\eqref{eq_w0last} and \eqref{eq_1:z}--\eqref{eq_2:z}.
We note that both equations are examples of the family of systems:
\begin{eqnarray}
		\partial_t y + \left( - \frac{1}{r} \partial_r (r \partial_r y)  + \frac{k^2}{r^2} y \right) = 0 \,,
		 & & \text{ for } (t,r) \in (0,\infty) \times (1,r)\,; \label{eq_1:y}
		\\
		y(t,1) = \ell_Y(t) \,,
		 & & \text{ for } (t,r) \in (0,\infty) \times (1,r)\,; \label{eq_1-b:y}
		 \\
		\ell_Y'(t) = \tilde{\alpha} (\partial_r y(t,1)- k y(t,1))\,,
		& &
		\text{ for } t \in (0, \infty) \, ; \label{eq_3:y}
\end{eqnarray}
with parameters $\tilde{\alpha} >0$ and $k \in \mathbb N \cup \{0\}.$ Indeed $(z, \ell_Z)$ solution of  \eqref{eq_1:z}--\eqref{eq_2:z} is a solution to \eqref{eq_1:y}--\eqref{eq_3:y} in the 
case 
$$
k = 0 \qquad \tilde{\alpha} = \dfrac{4\pi}{\pi+m}
$$
whereas $(w, \ell_W)$ solution of \eqref{eq_w0first}--\eqref{eq_w0last} is a solution to \eqref{eq_1:y}--\eqref{eq_3:y} in the case 
$$
k=1 \qquad \tilde{\alpha} = \dfrac{2\pi}{\mathcal J}\,.
$$

\medskip

To compute the decay of solutions to \eqref{eq_1:y}--\eqref{eq_3:y}, we use classical methods for parabolic equations (see \cite{EscobedoZuazua,Vazquez,VazquezZuazua,MunnierZuazua}).  In our context, due to the presence of the solid, we shall refer extensively to the works \cite{MunnierZuazua,MunnierZuazua2}
of A. Munnier and E. Zuazua  which study thoroughly the equation
\begin{equation}
	\label{Eq-on-V}
	\left\{
		\begin{array}{ll}
			\partial_t \text{v} - \Delta_{\R^n} \text{v} = 0, \quad & \text{ for } (t,x) \in  (0,\infty) \times \R^n \backslash B(0,1), 
			\\[8pt]
			\text{v}(t,x)= \ell_{\text{v}}(t) , & \text{ for } (t,x) \in (0, \infty) \times \mathcal S^{n-1},
			\\[4pt]
			 \ell_{\text{v}}' (t) = \alpha \displaystyle{\int_{\mathcal S^{n-1}}} \partial_r \text{v}(t,x)  \, d\sigma,  & \text{ for } t \in (0, \infty),
		\end{array}
	\right.
\end{equation}
where $\alpha>0$ is a fixed real number. 
Formally, for arbitrary $k \in \mathbb N,$ $(y, \ell_Y)$ is a solution to \eqref{eq_1:y}--\eqref{eq_3:y} if and only if the pair
$(\text{v},\ell_{\text{v}})$ defined by 
\begin{equation} 	\label{Correspondance-Z-V-k}
	\ell_{\text{v}}(t) = \ell_Y(t)\,,
	\qquad
	\text{v}(t,r, \omega) := \dfrac{y (t,r)}{r^k}, 
	\qquad 
	\forall \, r > 1 \,,  \quad  \forall \, \omega \in \mathcal S^{n-1},\,
\end{equation}
is a solution of equation \eqref{Eq-on-V} for 
\begin{equation}
	\label{Alpha}
	n = 2k+2\,, \quad \alpha = \dfrac{\tilde{\alpha}}{|\mathcal S^{n-1}|}\,.
\end{equation}

In this subsection, we fix  $k \in \mathbb N \cup \{0\}$ and $\tilde{\alpha}>0$ and study the long-time behavior of the solution of system \eqref{eq_1:y}--\eqref{eq_3:y}.
By \eqref{Alpha}, this fixes also values for $n$ and $\alpha.$
\medskip

In order to study system \eqref{Eq-on-V}, A. Munnier and E. Zuazua introduce the functional spaces
\begin{equation*}
	 \mf{L}^p(\mathbb R^n) = \{ Y \in L^p(\R^n),\ \nabla Y=0 \text{ in }B(0,1) \}, \qquad (p \in [1, \infty]),
\end{equation*}
endowed with the norm:
\begin{align*}
	  \| Y \|_{\mf{L}^p(\mathbb R^n)}^p&=\| y \|_{L^p(\mathbb R^n \setminus B(0,1) )}^p + \frac{1}{\alpha}|\ell_{Y}|^p , & \text{ when $p<\infty$}\,,\\
	  \| Y \|_{\mf{L}^\infty(\mathbb R^n)}    &= \max ( \| y \|_{L^{\infty}(\mathbb R^n \setminus B(0,1) )}, |\ell_Y| ) , & \text{corresponding to $p = \infty$}\,, 
\end{align*}
where $\ell_{Y}$ is the mean value of $Y$ in the ball:
\begin{equation*}
	\ell_{Y}= \frac{1}{|B(0,1)|}\int_{B(0,1)} Y(x) \, {\rm d}x.
\end{equation*}
As before, in what follows, we identify $(\text{v},\ell_\text{v}) \in L^p(\mathbb R^n \setminus B(0,1)) \times \R$ with the extension $\text{V} \in \mf{L}^p(\mathbb R^n)$ given by 
$\text{V} = \mathbf{1}_{B(0,1)} \ell_{\text{v}} + \mathbf{1}_{\mathbb R^n \setminus B(0,1)} \text{v}$, and we shall write $\text{V} \doteq ( \text{v},\ell_\text{v})$ to denote this extension. 

We also introduce a radial variant of $\mf{L}^p(\mathbb R^2)$-spaces:
\begin{equation*}
	 \mf{L}^p := \{ Y \doteq ( y,\ell_Y) \text{ radial function, such that  } Y \in {\mf L}^p(\R^2) \}\,.
\end{equation*}
This space is endowed with the norm:
\begin{align*}
	  \| Y \|_{\mf{L}^p}^p&=\| y \|_{L^p(\mc{F}_0)}^p +  \frac{2\pi}{\tilde{\alpha}}|\ell_{Y}|^p , & \text{ when $p<\infty$}\,,\\
	  \| Y \|_{\mf{L}^\infty}    &= \max ( \| y \|_{L^{\infty}(\mc{F}_0)}, |\ell_Y| ) , & \text{corresponding to $p = \infty$}\,.
\end{align*}
In the case $p=2,$ this space is a Hilbert space associated with the scalar product:
$$
(Y,\tilde{Y}) = \int_{\mathcal F_0} y \tilde{y} + \dfrac{2\pi}{\tilde{\alpha}} \ell_{Y} \ell_{\tilde{Y}}.
 $$
 For $p \neq 2,$  extending this scalar product by a density argument enables to identify the dual of $\mf{L}^{p}$ with $\mf{L}^{{p'}}$ where
 ${p'}$ is the conjugate exponent of $p.$

\medskip

 With these notations,  A. Munnier and E. Zuazua prove in \cite{MunnierZuazua,MunnierZuazua2}:
\begin{theorem}[Decay estimates for \eqref{Eq-on-V}, \cite{MunnierZuazua,MunnierZuazua2}]\label{thm_munnierzuazua_Cauchy}
	Given $(\text{\em v}_0,\ell_{\text{\em v}_0}) \in \mf{L}^2(\mathbb R^n),$  there exists a unique solution 
	$(\text{\em v},\ell_{\text{\em v}}) \in \mathcal{C}([0,\infty); \mf{L}^2(\mathbb R^n))$
	of \eqref{Eq-on-V} such that $(\text{\em v}(0,\cdot), \ell_{\text{\em v}}(0))  = (  \text{\em v}_0, \ell_{\text{\em v}_0}).$ This solution satisfies:
	\begin{equation} \label{eq_semigroupdimn}
		\|(\text{\em v}(t,\cdot),\ell_{\text{\em v}}(t))\|_{\mf L^2(\mathbb R^n)} \leq \| (\text{\em v}_0,\ell_{\text{\em v}_0})\|_{\mf{L}^2(\mathbb R^n)}\,, \quad \forall \, t \geq 0\,.
	\end{equation}
	Moreover, if $(\text{\em v}_0,\ell_{\text{\em v}_0}) \in \mf{L}^q(\mathbb R^n),$ for some $q \in [1, \infty]$, for all $p \in [q, \infty]$, there exists a constant $C(p,q)$ such that
	\begin{equation} \label{eq_semigroup-pq-mn}
		t^{\frac{n}{2} (1/q - 1/p)} \|(\text{\em v}(t,\cdot),\ell_{\text{\em v}}(t))\|_{\mf L^p(\mathbb R^n)} \leq  \| (\text{\em v}_0,\ell_{\text{\em v}_0})\|_{\mf{L}^q(\mathbb R^n)}\,, \quad \forall \, t \geq 1\,.
	\end{equation}
\end{theorem}

\begin{theorem}[First term in the asymptotic expansion of solutions of \eqref{Eq-on-V}, \cite{MunnierZuazua,MunnierZuazua2}]\label{thm_munnierzuazua_Exp}

	Given $(\text{\em v}_0,\ell_{\text{\em v}_0}) \in \mf{L}^2(\mathbb R^n)$ such that 
	$
		\text{\em v}_0 \in L^2( \R^n \setminus B(0,1); \exp(|x|^2/4){\rm d}x), 
	$
	 setting
	\begin{equation*}
		M = \int_{\R^n \setminus B(0,1)} \text{\em v}_0(x)\, {\rm d}x + \frac{1}{\alpha} \ell_{\text{\em v}_0}\,, 
	\end{equation*}
	we get
	\begin{itemize}
		\item for all $t >0$ and $p \in [1, \infty]$, $(\text{\em v}(t,\cdot),\ell_{\text{\em v}}(t)) \in \mf{L}^p(\mathbb R^n)$
		\item for all $p\in [1, \infty]$, there exists a constant $C_p$ such that for all $t >0$,
		\begin{eqnarray*}
				t^{\frac{n}{2}(1- 1/p)} \norm{\text{\em v}(t,\cdot) - M G(t) }_{L^p(\R^n \setminus B(0,1))} & \leq  & C_p R_{1,p}(t), 
				\\
				t^{\frac{n}{2}} \left|\ell_{\text{\em v}}(t) - \frac{M}{(4 \pi t)^{\frac{n}{2}}} \right| & \leq & C R_2(t), 
		\end{eqnarray*}
		where
		\begin{equation*}
			G(t,x) = \frac{1}{(4 \pi t)^{\frac{n}{2}}} \exp\left( - \frac{|x|^2}{4t} \right), 
		\end{equation*}
		and, denoting by $\delta_{n,2}$ the Kronecker symbol:
		\begin{eqnarray*}
			R_{1,p}(t) &= &
				\left\{
				\begin{array}{ll}
					(\delta_{n,2}|\log(t)| +1 ) t^{-1/2} & \text{ if } p \in [1, 2], 
					\\[3pt]
					(\delta_{n,2}|\log(t)| + 1) t^{-1/2+\theta_{n,p}} & \text{ if } p \geq 2,
				\end{array}
				\right. 
			\quad \text{ with }
				\theta_{n,p} = \frac{n}{2} \frac{(p-1)(p-n)}{p (2p + n(p -1))}, 
			\nonumber
			\\[4pt]
			 R_2(t) & = & (\delta_{n,2} |\log(t)|^{1/2}+ 1) t^{-1/(n+2)}.
			 \nonumber
		\end{eqnarray*}
	\end{itemize}
\end{theorem}

We do not give a comprehensive proof of Theorems \ref{thm_munnierzuazua_Cauchy}--\ref{thm_munnierzuazua_Exp} and let the reader refer to  \cite{MunnierZuazua2,MunnierZuazua}
for further details. 
Let us only recall that the proof of Theorem \ref{thm_munnierzuazua_Cauchy} is based on the remark that \eqref{Eq-on-V} reduces to the abstract ODE:
$\partial_t \text{V} + A_{mz} \text{V} = 0,$
where $A_{mz}$ is the unbounded operator with domain
\begin{equation}
	\mathcal{D}(A_{mz}) = \{   \text{V} \doteq ( \text{v}, \ell_{\text{v}}) \in H^2(\R^n \backslash B(0,1)) \times \R \text{ with } \text{v}_{| |x| = 1} = \ell_{\text{v}} \},
\end{equation}
such that:
\begin{equation}
	\label{StokesOperator}
	 A_{mz} (\text{v},\ell_{\text{v}}) =  \left( \begin{array}{cc}  -\Delta \cdot & 0 \\[4pt] - \alpha \displaystyle{\int_{\mathcal{S}^{n-1}}}\, \partial_r \cdot d \sigma & 0\end{array} \right)\begin{pmatrix}\text{v}\\ \ell_{\text{v}} \end{pmatrix}
	=  \left( \begin{array}{cc}  -\Delta \text{v}  \\[4pt] - \alpha \displaystyle{\int_{\mathcal S^{n-1}}}\, \partial_r \text{v} d \sigma \end{array} \right)\,.
\end{equation}
A. Munnier and E. Zuazua show that this operator is maximal monotone which implies the existence of a contraction semigroup on $\mf{L}^2(\mathbb R^n)$
representing the unique solution to \eqref{Eq-on-V}. Further classical smoothing properties of this semigroup also yield that, for $(\text{v}_0,\ell_{\text{v}_0}) \in \mathcal{D}(A_{mz}),$ the unique solution to  \eqref{Eq-on-V} satisfies:
\begin{equation} \label{eq_regsolv}
	{\text V} \in \mathcal{C}^1([0,\infty); {\mf L}^2(\mathbb R^n) ) \cap \mathcal{C}([0,\infty); {\mathcal D}(A_{mz}) )\,.
\end{equation}
We remark that \eqref{Eq-on-V} is rotational invariant. Hence, considering radial data and noting that transformation 
\eqref{Correspondance-Z-V-k} is a bi-continuous one-to-one and onto mapping from $\mf{L}^2$ to radial functions in $\mf{L}^2(\mathbb R^n)$
(with $n=2k+2)$, Theorem \ref{thm_munnierzuazua_Cauchy} implies:

\begin{theorem} \label{thm_wpy} 
	Given $Y_0 \in \mf{L}^2$  there exists a unique solution 
	$Y \in \mathcal{C}([0,\infty); \mf{L}^2)$
	of \eqref{eq_1:y}--\eqref{eq_3:y} such that $Y(0,\cdot)  = Y_0.$ This solution satisfies:
	\begin{equation*} \label{eq_semigroupy}
		\|Y(t,\cdot)\|_{\mf L^2} \leq \| Y_0\|_{\mf{L}^2}\,, \quad \forall \, t \geq 0\,.
	\end{equation*}
\end{theorem}

This theorem implies again that the solution to \eqref{eq_1:y}--\eqref{eq_3:y} is given by a contraction semigroup on $\mf{L}^2$
denoted by $S_y$ in what follows. 
The results in \cite{MunnierZuazua,MunnierZuazua2} are not sufficient for our purposes. Indeed, we also have to compute decay rates in $\mf{L}^p-\mf{L}^q$ spaces, similar to the ones in \eqref{eq_semigroup-pq-mn}. But, when $n \neq 2$ (equivalently $k \neq 0$) and $p \neq 2$, the transformation \eqref{Correspondance-Z-V-k} is not an isometry between $\mf{L}^p(\R^n)$ and $\mf{L}^p,$ so that the ``change of dimension'' argument does not yield the expected result. Besides, we will also derive estimates on the $\partial_r y, \, y/r, \, \partial_{rr}y,\ \partial_r y/r$, and $y/r^2$ in ${L}^p(\mc{F}_0)$, for which no precise estimates were given in  \cite{MunnierZuazua,MunnierZuazua2}, except in the case $p = 2$.

In the following subsection, we adapt the arguments of \cite{MunnierZuazua,MunnierZuazua2} to system \eqref{eq_1:y}--\eqref{eq_3:y} to estimate the decay of $S_y$ in $\mf{L}^p$. We then explain how to derive estimates on the derivatives of solutions of \eqref{eq_1:y}--\eqref{eq_3:y} in $\mf{L}^p$.

\subsubsection{${\mf L}^p-{\mf L}^q$ estimates on $y$}

Inspired in \cite{MunnierZuazua,MunnierZuazua2} , we prove the following ${\mf L}^p-{\mf L}^q$ decay estimates for solutions of  \eqref{eq_1:y}--\eqref{eq_3:y}:
\begin{theorem}
	\label{ThmMunnierZuazuaExt}
	For all $q \in [1, \infty)$, system \eqref{eq_1:y}--\eqref{eq_3:y} is well-posed in $\mf{L}^q$: given $Y_0 \in \mf{L}^q$ there is one unique solution $Y$ of \eqref{eq_1:y}--\eqref{eq_3:y} in $\mathcal{C}([0,\infty); \mf{L}^q)$. This solution satisfies:
	\begin{equation}
		\label{EstLq}
		\norm{Y(t)}_{\mf{L}^q} \leq \norm{Y_0}_{\mf{L}^q}.
	\end{equation}	
	We furthermore have the following $\mf{L}^p$ estimates: for all $p \in  [q, \infty]$, $Y$ belongs to $\mathcal{C}((0, \infty); \mf{L}^p)$ and there exists a constant $C$ such that
	\begin{equation}
		\label{DecayLp-Lq}
		t^{1/q - 1/p} \norm{Y(t)}_{\mf{L}^p} \leq C \norm{Y_0}_{\mf{L}^q}, \quad t > 0.
	\end{equation}
	Furthermore, if $Y_0$ also belongs to $\mf{L}^\infty$, we also have $	\norm{Y(t)}_{\mf{L}^\infty} \leq \norm{Y_0}_{\mf{L}^\infty}$.
\end{theorem}

Before going into the proof of Theorem \ref{ThmMunnierZuazuaExt}, let us emphasize that estimates \eqref{EstLq}--\eqref{DecayLp-Lq} are different from the ones in \eqref{eq_semigroup-pq-mn} when $k = 1$, \emph{ i.e. } $n=4$, that corresponds to $( w, \ell_W)$ solutions of \eqref{eq_w0first}--\eqref{eq_w0last}. To be more precise, in that case, using the transformation \eqref{Correspondance-Z-V-k} for $r>1$, \eqref{eq_semigroup-pq-mn} would then read: for all $q \in [1, \infty]$, $p \in [q, \infty]$, there exists a constant $C$ such that for all $W_0 \doteq ( w_0, \ell_{W_0})$ satisfying $w_0/r \in {L}^q(\mathbb{R}^4\setminus B(0,1))$, the solution $W \doteq ( w, \ell_W)$ of \eqref{eq_w0first}--\eqref{eq_w0last} satisfies, for all $t >0$,
\begin{equation}
	\label{Decay-W-By-Munnier-Zuazua}
	t^{2(1/q - 1/p)} \Big\| \frac{w(t)}{r} \Big\|_{{L}^p(\mathbb R^4\setminus B(0,1))} \leq  C \Big\|\Big( \frac{w_0}{r} , \ell_{W_0} \Big)\Big\|_{\mf{L}^q(\mathbb R^4)}.
\end{equation}
Hence, the solution $W$ of \eqref{eq_w0first}--\eqref{eq_w0last} will simultaneously satisfy the decay estimates \eqref{DecayLp-Lq} and \eqref{Decay-W-By-Munnier-Zuazua}. Actually, as we explain below, both results can be proved following the same strategy based on suitable multipliers, the only difference being Sobolev's embeddings.

\begin{proof}
Let $Y_0 \in \mf{L}^2 $ and  $Y \doteq ( y,\ell_Y) \in \mathcal{C}([0,\infty); \mf{L}^2) $
be the unique solution to \eqref{eq_1:y}--\eqref{eq_3:y} given by Theorem \ref{thm_wpy}. Up to assume that $Y_0$ is sufficiently smooth and vanish sufficiently rapidly at infinity we can apply the regularizing effect of the semigroup in $\mathbb R^n$ (see \eqref{eq_regsolv}) so that, going back in $\mathbb R^2$ 
we have ${Y} \in \mathcal{C}([0,\infty);H^1(\mathbb R^2))$ and $y \in \mathcal{C}([0,\infty) ; H^2(\mathcal F_0))\,.$
Then, the idea is to multiply equation \eqref{eq_1:y} by $j'(y)$ for smooth non-decreasing convex functional $j = j(y)$ with at most linear growth at infinity. After integration by parts, this yields 
\begin{equation}
\label{Lyapunov-general}
	\frac{d}{dt} \left(\int_{\mc{F}_0} j(y) + \frac{2\pi}{ \tilde{\alpha}} j(\ell_Y) \right) + \int_{\mc{F}_0} j''(y) |\nabla y|^2 + \dfrac{2k\pi}{\tilde{\alpha}} j'(\ell_Y) \ell_Y  + k^2 \int_{\mathcal F_0} \dfrac{j'(y)y}{|x|^2} = 0.
\end{equation}
After a classical regularization argument, one can show that such estimate can be extended to the convex functionals $j(y) = |y|^q,$ for $q \in [1, \infty),$ and this yields:
\begin{equation}
	\label{Contraction-p}
	\frac{d}{dt}  \left( \| Y \|_{\mf{L}^q}^q \right) \leq 0. 
\end{equation}
Similarly, using functionals of the form $j(y) = (y-K)_+$, after a suitable regularization argument, one derives
\[
	\frac{d}{dt}  \left( \| Y \|_{\mf{L}^\infty} \right) \leq 0. 
\]
Based on the contraction property \eqref{Contraction-p}, the semigroup $S_y(t)$ can be uniquely extended by density to initial data in $\mf{L}^q$ as an operator from $\mf{L}^q$ to itself. We thus have the well-posedness of \eqref{eq_1:y}--\eqref{eq_3:y} in \emph{any} $\mf{L}^q$, $ q \in [1, \infty)$ (${\mf L}^2 \cap H^2(\mathbb R^2 \setminus B(0,1))$ is not densely embedded in $\mf{L}^\infty$, thus our argument does not apply in that case).
This yields also the decay estimates \eqref{EstLq} for all $q \in [1, \infty]$ and $y_0 \in \mf{L}^q$. Note that the decay estimates \eqref{EstLq} also coincide with the decay estimates  \eqref{DecayLp-Lq} for any $p=q <\infty.$

\medskip

Actually, one can go even further. Taking $j(y) = |y|^p$ for $p \geq 2$, estimate \eqref{Lyapunov-general} implies (forgetting the two last terms which are non-negative):
\begin{equation}
\label{Lyapunov-p}
	\frac{d}{dt} \left( \| Y \|_{\mf{L}^p}^p \right) + \frac{ 4(p-1)}{ p} \int_{\mc{F}_0}  |\nabla (|y|^{p/2})|^2 \leq 0.
\end{equation}
Using then suitable Sobolev embeddings and interpolation estimate (actually, this is the only step where the dimension plays a role), one gets Lemma 2.2 in \cite{MunnierZuazua} (the proof is done in \cite{MunnierZuazua2}), and in particular \cite[(2.17)]{MunnierZuazua}: there exists a constant $C$ such that for all functions $Y \doteq (y , \ell_Y) \in \mf L^1$ with $y \in H^1(\mc{F}_0)$:
\begin{equation*}
	\| Y \|_{\mf{L}^2}^4  \leq C  \|Y \|_{\mf{L}^1}^{2} \| \nabla y \|_{L^2(\mc{F}_0)}^2.
\end{equation*}
Applying it to $|Y|^q$, we get the existence of a constant $C$ such that for all $q \geq 1$,
\[
	\left(\| Y \|_{\mf{L}^{2q}}^{2q} \right)^{2} \leq C \left(\|Y \|_{\mf{L}^q}^{q} \right)^2 \| \nabla (|y|^{q}) \|_{L^2(\mc{F}_0)}^2.
\]
Plugging this estimate in \eqref{Lyapunov-p} for $p = 2q$ and using the fact that the $\mf{L}^q$-norm of $y$ decays according to \eqref{Contraction-p}, 
\[
	\frac{d}{dt} \left(\| Y(t) \|_{\mf{L}^{2q}}^{2q} \right) + \frac{2 q -1}{ 8 C q \|Y_0 \|_{\mf{L}^q}^{2q}} \left(\|Y(t) \|_{\mf{L}^{2q}}^{2q}\right)^2 \leq 0.
\]
Of course, this implies that there exists a constant $C$ independent of $q \in [1, \infty)$ such that
\[
	\frac{d}{dt} \left(\| Y(t) \|_{\mf{L}^{2q}}^{2q} \right) + \frac{1}{ C  \|Y_0 \|_{\mf{L}^q}^{2q}} \left(\|Y(t) \|_{\mf{L}^{2q}}^{2q}\right)^2 \leq 0.
\]
This yields the following decay property: there exists a constant $C>0$ independent of $q>0$ such that for all $q \in [1, \infty)$,
\begin{equation}
	\label{Decay-2q-q}
	t \|Y(t) \|_{\mf{L}^{2q} }^{2q} \leq C \left(\|Y_0\|_{\mf{L}^q}^q \right)^2 .
\end{equation}
Then, the iteration argument of \cite{Veron79} based on \eqref{Decay-2q-q}  applies and yields
\[
	t^{1/q } \norm{Y(t)}_{\mf{L}^\infty} \leq C \norm{Y_0}_{\mf{L}^q}, \quad t > 0.
\]
Other estimates in \eqref{DecayLp-Lq} are deduced for arbitrary $p \in [q,\infty)$ by interpolating the cases $p=q$ and $p= \infty$.
\end{proof}

As we mentioned in the above proof, the semigroup $S_y$ associated with system \eqref{eq_1:y}--\eqref{eq_3:y} extends  to a semigroup on $\mf L^q$ for all $q \in [1, \infty)$ that we still denote the same for simplicity. Consequently, Corollary \cite[Corollary 2.5, p.5]{Pazy}
implies that it is associated to a closed linear operator. In this case the operator reads $A_q$ where
	\begin{equation*}
		\mathcal{D}(A_{q}) = \{ Y \doteq  (y, \ell_Y) \in \mf{L}^{p} \text{ with } A_q Y \in \mf{L}^q \}.
	\end{equation*}
and 
\begin{equation*}
		A_q Y = A_q (y, \ell_Y) =  \left( \begin{array}{cc}  -\Delta  + \dfrac{k^2}{r^2} \cdot & 0 \\[4pt] - \dfrac{\tilde{\alpha}}{2\pi} \displaystyle{\int_{\mathcal{S}^1}}\, \partial_r \cdot d \sigma & \tilde{\alpha}k \end{array} \right)\begin{pmatrix}y \\ \ell_Y\end{pmatrix}
		=  \left( \begin{array}{cc}  -\Delta y  + \dfrac{k^2 y }{r^2} \\[4pt] - \dfrac{\tilde{\alpha}}{2\pi} \displaystyle{\int_{\mathcal{S}^1}}\, \partial_r  y d \sigma + k \tilde{\alpha} \ell_Y \end{array} \right).
\end{equation*}

\medskip

\subsubsection{${\mf L}^p-{\mf L}^q$ estimates on $\partial_t Y$}
In the case $p=2,$ as $A_{2}$ is self-adjoint (see \cite[App. A]{MunnierZuazua2}), Theorem 7.7 in \cite{Brezis} states that, if $Y_0 \in \mf{L}^2$, the solution $Y$ of \eqref{eq_1:y}--\eqref{eq_3:y} belongs to $C^\infty((0, \infty); \cap_{\ell \in \N} \mathcal{D}(A^\ell_{2}))$ and 
\begin{equation*}
	\norm{\partial_t Y(t)}_{\mf{L}^2}=\norm{A_{2}Y(t)}_{\mf{L}^2} \leq \frac{C}{t} \norm{Y_0}_{\mf{L}^2}.
\end{equation*}
Extending this result to the $\mf{L}^q$ case, for $q \in (1,\infty)$ turns out to be slightly more intricate. 

\begin{theorem}
	\label{Thm-L-p-Anal}
	For all $q \in (1, \infty)$, there exists a constant $C = C(q)$ such that for all $Y_0 \in \mf{L}^q$, the solution $y$ of \eqref{eq_1:y}--\eqref{eq_3:y} satisfies, for all $t >0$,
	\begin{equation}
		\label{Est-du-Brezis-p}
		\norm{\partial_t Y(t)}_{\mf{L}^q} \leq \frac{C}{t} \norm{Y_0}_{\mf{L}^q}.
	\end{equation}
\end{theorem}

\begin{proof}
	The proof of such result is rather classical, but we did not find precise reference in our precise setting. 
	We follow the proof of Theorem 3.6 in \cite[Chapter 7]{Pazy}. First, we recall that $\mf{L}^q$ is a Banach space whose dual is identified with $\mf{L}^{q'},$ for $q' = q/(q-1),$ when taking the duality pairing 
	\[
		\langle Y_1, Y_2 \rangle_{\mf{L}^q, \mf{L}^{q'}}  =  \int_{\mc{F}_0} y_1 \overline{y_2} + \frac{2\pi}{ \tilde{\alpha}} \ell_{Y_1} \overline{\ell_{Y_2}},
	\]
	for $Y_1 \doteq ( y_1, \ell_{Y_1})$, $Y_2 \doteq ( y_2, \ell_{Y_2})$. Note that, in this proof only, we extend $\mf{L}^p$ to functions having complex values. 
	We focus on the case $q \geq 2$. 
	
\medskip
	
For $Y \in \mathcal{D}(A_q)$, $Y^* = |Y|^{q-2} \overline{Y} $ belongs to $\mf{L}^{q'}$ and satisfies
	\[
		\langle Y, Y^* \rangle_{\mf{L}^q, \mf{L}^{q'}} = \| Y \|_{\mf{L}^q}^q \text{ and } \| Y^* \|_{\mf{L}^{q'}} = \| Y \|_{\mf{L}^q}^{q-1}.
	\]
	Besides, easy computations yield 
	\[
		\langle A_q Y, Y^* \rangle_{\mf{L}^q, \mf{L}^{q'}} = \frac{q}{2} \int_{\mc{F}_0} |y|^{q-2} |\nabla y|^2 + \left(\frac{q}{2} -1 \right) \int_{\mc{F}_0} |y|^{q-4}   ( \overline{y} \nabla y)^2 + k^2 \int_{\mathcal{F}_0} \dfrac{|y|^q}{r^2} + 2 \pi  k |\ell_Y|^q. 
	\]
	In particular, both first terms can be expressed easily in terms of $|y|^{\frac{q}{2}-2} \overline{y} \nabla y.$ So, we introduce the vectors $\vec{a}= \vec{a}(x)$, and $\vec{b} =  \vec{b}(x)$ of $\R^2$ defined by $|y|^{\frac{q}{2}-2} \overline{y} \nabla y = \vec{a} + i \vec{b}.$ We get
	\[
		\langle A_q Y, Y^* \rangle_{\mf{L}^q, \mf{L}^{q'}} = (q-1) \int_{\mc{F}_0} 	|\vec{a}|^2 + \int_{\mc{F}_0} |\vec{b}|^2 + (q-2) i \int_{\mc{F}_0} \vec{a} \cdot \vec{b}  + k^2 \int_{\mathcal{F}_0} \dfrac{|\text{v}|^q}{r^2} + 2 \pi k |g|^q.
	\]
	In particular, 
	\begin{equation*}
		\Re\left(\langle A_q Y, Y^* \rangle_{\mf{L}^q, \mf{L}^{q'}} \right) \geq (q-1) \| \vec{a} \|_{L^2(\mc{F}_0)}^2 + \| \vec{b} \|_{L^2(\mc{F}_0)}^2,
	\end{equation*}
	whereas 
	\[
		|\Im\left(\langle A_q Y, Y^* \rangle_{\mf{L}^q, \mf{L}^{q'}} \right)| \leq |q-2| \| \vec{a} \|_{L^2(\mc{F}_0)} \| \vec{b} \|_{L^2(\mc{F}_0)}.
	\]
This implies 
	\begin{equation}
		\label{NumericalRange}
		\frac
		{|\Im\left(\langle A_q Y, Y^* \rangle_{\mf{L}^q, \mf{L}^{q'} }\right)|}
		{\Re\left(\langle A_q Y, Y^* \rangle_{\mf{L}^q, \mf{L}^{q'}} \right) } 
		\leq \frac{1}{2}\frac{|q-2|}{\sqrt{ q-1}}.
	\end{equation}
From Theorem \ref{ThmMunnierZuazuaExt}, $-A_q$ generates a $C_0$ semigroup of contractions on $\mf{L}^q$ hence Theorem 3.1 in \cite[Chapter 1]{Pazy} implies that for all $\lambda >0$, $\lambda$ is in the resolvent set of $-A_q$. 
	
\medskip
	
 For $q \geq 2$, from \eqref{NumericalRange}, the numerical range $S(-A_q)$ is contained in the sector $\Sigma_{\theta_0} = \{\lambda \in \C\setminus \{0\} \, : \, |\text{arg}\, \lambda| > \pi - \theta_0 \}$ where 
	\[
		\theta_0 = \arctan\left( \frac{1}{2} \frac{|q-2|}{\sqrt{ q-1}}\right) \in [0, \pi /2).
	\]
	In particular, choosing $\theta_1 \in (\theta_0, \pi/2)$, denoting by $\Sigma_{\theta_1} = \{\lambda \in \C\setminus \{0\} \, : \, |\text{arg}\, \lambda| > \pi - \theta_1 \}$  the corresponding sector of $\mathbb{C}$, and using the fact that $\R_+^*$ is in the resolvent set, Theorem 3.9 in \cite[Chapter 1]{Pazy} implies the existence of a $C_\theta = C_\theta (q)$ such that 
	\begin{equation}
		\label{ResolventEst}
		\|(\lambda I + A_q)^{-1} \|_{\mathscr{L}(\mf{L}^p)} \leq \frac{C_\theta}{|\lambda|} , \quad \forall \, \lambda \in \mathbb{C} \setminus \Sigma_{\theta_1}.
	\end{equation}
	
 Now, the regularizing properties of the semigroup generated by $-A_q$ are a consequence of Theorem 5.2 in \cite[Chapter 2]{Pazy} and the above resolvent estimate. However, here again, we need to be careful since Theorem 5.2 in \cite[Chapter 2]{Pazy} requires that $0$ belongs to the resolvent set of $A_q$, which is not the case here.
 Set $\theta_2 \in ( \pi/2, \pi - \theta_1)$. For each $\varepsilon >0$, we introduce the curve $\Gamma_\varepsilon$, defined for $\varepsilon >0$ by the path composed as follows: 
	\[
		\Gamma_\varepsilon = 
			\left\{
				\begin{array}{ll}
				-\rho \exp(- i \theta_2), \quad & \rho \in (- \infty , -\varepsilon),
				\\
				\varepsilon \exp( i \theta), \quad & \theta \in (- \theta_2, \theta_2),
				\\
				\rho \exp(i \theta_2), \quad & \rho \in (\varepsilon, \infty),
			\end{array}\right.		 
	\]
	oriented in the increasing directions of the parameters.	
	Then, for $t >0$, we use the formula
	\[
		S_y(t) = \frac{1}{ 2 \pi i} \int_{\Gamma_\varepsilon} e^{\lambda t} (\lambda I +A _q)^{-1} \, {\rm d} \lambda.
	\]
	This integral converges due to the resolvent estimates \eqref{ResolventEst} and can be differentiated with respect to time since
	\[
		\frac{1}{ 2 \pi} \int_{\Gamma_\varepsilon} e^{- \Re(\lambda) t} \left\| \lambda (\lambda I +A _q)^{-1}\right\|_{\mathscr{L}_c(\mf{L}^p)}  \, {\rm d} \lambda
		\leq 
		 \frac{C_{\theta_2}}{t} +C_t \varepsilon,
	\]
	where the constant $C_{\theta_2}$ does not depend on $t$ and $\varepsilon >0$ and the constant $C_t$ depends on $t$ but not on $\varepsilon >0$.
	Of course, letting then $\varepsilon \to 0$, this yields
	\[
		\left\| 
		\frac{1}{ 2 \pi i} \int_{\Gamma_\varepsilon} e^{\lambda t} \lambda (\lambda I +A _q)^{-1} \, {\rm d} \lambda
		\right\|_{\mathscr{L}_c(\mf{L}^p)} 
		=
		\| \partial_t S_y(t) \|_{\mathscr{L}_c(\mf{L}^q)} 
		\leq 
		 \frac{C_{\theta_2}}{t}.
	\]
	This completes the proof of \eqref{Est-du-Brezis-p} for $q \geq 2$.  The case $q \in (1,2)$ can be deduced by a simple duality argument.
\end{proof}

\begin{remark}
	Actually, following the proof of Theorem 5.2 in \cite[Chapter 2]{Pazy}, one can prove that $-A_q$ generates an analytic semigroup on $\mf{L}^q$ for all $q \in (1, \infty)$.
\end{remark}

In the two next subsections, we apply the semigroup estimates we have proved to systems \eqref{eq_1:z}-\eqref{eq_2:z} and \eqref{eq_w0first}-\eqref{eq_w0last} .
 
\subsection{Decay of solutions to (\ref{eq_1:z})--(\ref{eq_2:z})} \label{sec_decayz}
We first consider the solution $Z \doteq ( z, \ell_Z)$ of \eqref{eq_1:z}--\eqref{eq_2:z}. As we noticed previously, this corresponds to the computations
of the previous subsection in the case
$$
k=0 \,, \quad \tilde{\alpha} = \dfrac{4\pi}{\pi+m}\,.
$$
We obtain in this way the following  decay estimates on solutions: 

\begin{theorem} \label{thm_decayz}
Given $q \in (1,\infty)$ and radial $Z_0 \in \mf{L}^q,$ there exists a unique solution $Z \in \mathcal{C}([0,\infty) ; \mf{L}^q)$ to \eqref{eq_1:z}-\eqref{eq_2:z}
such that $Z(0,\cdot) = Z_0$.  This solution satisfies the further decay estimates:
\begin{itemize}
	\item for all $p \in [q,\infty]$ we have $Z \in \mathcal{C}((0,\infty) ; \mf{L}^p)$ and there exists a constant $K_{1,1} = K_{1,1}(p,q)$ such that:
\begin{equation} \label{est_zLp}
\|Z(t,\cdot)\|_{\mf{L}^p} \leq K_{1,1}\, t^{\frac{1}{p}- \frac{1}{q}} \|Z_0\|_{\mf{L}^q}\,, \quad \forall \, t >0\,,
\end{equation}
\item  if $q < 2$ for all $p \in [q,2),$ we have $(\partial_r z,z/r) \in \mathcal{C}((0,\infty) ; {L}^p(\mathcal F_0))$ and there exists $K_{2,1} = K_{2,1}(p,q)$ such that:
\begin{equation} \label{est_gradzLp1}
\|\partial_r z(t,\cdot)\|_{L^p(\mathcal F_0)}  + \left\| \dfrac{z(t,\cdot)}{r} \right\|_{L^p(\mathcal F_0)}  \leq K_{2,1}\, t^{-\frac{1}{2} + \frac{1}{p} - \frac{1}{q}}  \|Z_0\|_{\mf{L}^q}\,, \quad \forall \, t >0\,,
\end{equation}
\item  if $q \in (1, \infty)$, for all $p \in [\max\{2,q\},\infty)$ with $p >2$ we have $(\partial_r z,z/r) \in \mathcal{C}((0,\infty) ; {L}^p(\mathcal F_0))$ and there exists $K_{3,1} = K_{3,1}(p,q)$ such that:
\begin{equation} \label{est_gradzLp2}
\|\partial_r z(t,\cdot)\|_{L^p(\mathcal F_0)}  + \left\| \dfrac{z(t,\cdot)}{r} \right\|_{L^p(\mathcal F_0)}  \leq
\left\{
\begin{array}{ll}
 K_{3,1}\, t^{-\frac{1}{2} + \frac{1}{p} - \frac{1}{q}}  \|Z_0\|_{\mf{L}^q}\,, & \forall \, t <1 \,, \\
 K_{3,1}\, t^{ - \frac{1}{q}}  \|Z_0\|_{\mf{L}^q}\,, & \forall \, t >1 \,. \\
\end{array}
\right.
\end{equation}
\end{itemize}
These decay estimates are also satisfied for $q=1$ and $p \in (1, \infty)\setminus \{2 \}$.
\end{theorem}

\begin{proof}
Existence of solutions and $\mf{L}^p-\mf{L}^q$ decay estimates are straightforward applications of Theorem \ref{ThmMunnierZuazuaExt}
in the case $k=0$ and $\tilde{\alpha} = 4\pi/(\pi+m) >0.$ We focus on estimates \eqref{est_gradzLp1}--\eqref{est_gradzLp2}. Actually, we only need to prove the case $p=q$, as other cases are then obtained by combining the estimates  \eqref{est_gradzLp1}--\eqref{est_gradzLp2} for $p = q \neq 2$ between $t/2$ and $t$ with \eqref{est_zLp} between $0$ and $t/2$. Indeed it will follow from the semigroup property: 
\begin{equation*}
\begin{split}
 \|\partial_r z(t,\cdot)\|_{L^p(\mathcal F_0)}  + \left\| \dfrac{z(t,\cdot)}{r} \right\|_{L^p(\mathcal F_0)}  &\leq K_{2,1}(p,p)\, (t/2)^{-\frac{1}{2}}  \|Z(t/2,\cdot)\|_{\mf{L}^p}\\&\leq  K_{2,1}(p,p)K_{1,1}(p,q) \, (t/2)^{-\frac{1}{2} + \frac{1}{p} - \frac{1}{q}}  \|Z_0\|_{\mf{L}^q}.
\end{split}
\end{equation*}

\medskip

For radial $Z_0 \in \mf{L}^q$, estimate \eqref{Est-du-Brezis-p} implies 
	\[
		\norm{\partial_{rr} z(t) + \frac{\partial_r z(t)}{r} }_{L^q(\mc{F}_0)}  + | \partial_r z(t,1) | \leq \frac{C}{t} \| Z_0 \|_{\mf{L}^q}. 
	\]
We are now in position to apply Proposition \ref{prop_z2psi} to $\partial_r z$ which yields that (see \eqref{est_z2psi}) 
\begin{equation} \label{Est-Delta-p}
\| \partial_{rr} z(t)\|_{L^q(\mathcal{F}_0)} + \left\|\frac{\partial_r z(t)}{r}\right\|_{L^q(\mathcal{F}_0)} \leq \frac{C}{t} \| Z_0 \|_{\mf{L}^q}\,.  
\end{equation}
From the Gagliardo Niremberg inequality in exterior domains, see \cite{CrispoMaremonti}, we have then: for $q\in [1, \infty]$, for all $z$ such that $\partial_{xx} z,$ $\partial_{xy} z$ and $\partial_{yy} z$ belong to $L^q(\mc{F}_0)$, 
\begin{equation} \label{eq_CM2}
		\| \nabla z \|_{L^q(\mc{F}_0)} 
		\leq C \left( \|\partial_{xx} z\|_{L^q(\mc{F}_0)}  {+ \|\partial_{xy} z\|_{L^q(\mc{F}_0)}} + \|\partial_{yy} z\|_{L^q(\mc{F}_0)}\right)^{1/2} \| z\|_{L^q(\mc{F}_0)}^{1/2}.
\end{equation}
Since we are focusing on the case of radial solutions, estimates \eqref{Est-Delta-p}--\eqref{eq_CM2} and the fact that for radial functions
	\[
		 \|\partial_{xx} z\|_{L^q(\mc{F}_0)} { + \|\partial_{xy} z\|_{L^q(\mc{F}_0)}} +  \|\partial_{yy} z\|_{L^q(\mc{F}_0)} \leq C\left( \| \partial_{rr} z \|_{L^q(\mc{F}_0)} + \norm{ \frac{\partial_{r} z}{r} }_{L^q(\mc{F}_0)}\right),
	\]
imply 
\begin{equation*} 
\| \partial_{r} z (t)\|_{L^q(\mc{F}_0)}  \leq 	\frac{C}{\sqrt{t}} \| Z_0 \|_{\mf{L}^q}.
\end{equation*}
To conclude the proof of Theorem \ref{thm_decayz}, we prove the boundedness of the mapping $\partial_r z \mapsto z/r.$
As $z \in L^q(\mathcal F_0),$ this is already contained in Proposition \ref{prop_drz2z}, provided we get a suitable estimate on $z(t,1) = \ell_Z(t)$. But, using \eqref{est_zLp} for $p = \infty$, we get
$$
	|\ell_Z(t)| = |z(t,1)| \leq C_q t^{-\frac{1}{q}} \|Z_0\|_{\mf{L}^q}. 
$$
Thus, \eqref{est_drz2z} implies: 
$$
	\left\| \frac{z(t)}{r} \right\|_{L^q(\mathcal F_0)} \leq C_q \left( t^{-\frac{1}{2}}  + \varepsilon_q  t^{-\frac{1}{q}}\right) \|Z_0\|_{\mf{L}^q} 
$$
where $\varepsilon_q = 1$ if $q>2$ and $\varepsilon_q = 0$ if $q< 2.$ We obtain \eqref{est_gradzLp1} and \eqref{est_gradzLp2} comparing the size of 
the different terms on the right-hand side depending on  $q < 2 $ or $q >2$ and $t \geq 1$ or $t <1.$
\end{proof}

\subsection{Decay of solutions to (\ref{eq_w0first})--(\ref{eq_w0last})} \label{sec_decayw}
The equation  (\ref{eq_w0first})--(\ref{eq_w0last}) of $w$ is linked to the computations in Section \ref{sec_sgest} in the case $k=1$ and $\tilde{\alpha} = \mathcal{J}/{2\pi}.$ 
Thus, we compute the following time-decay of solutions:

\begin{theorem} \label{thm_decayw}
Given $q \in (1,\infty)$ and radial $W_0 \in \mf{L}^q,$ there exists a unique solution $W \in \mathcal{C}([0,\infty) ; \mf{L}^q)$ to \eqref{eq_w0first}-\eqref{eq_w0last}
such that $W(0,\cdot) = W_0$.  This solution satisfies the further decay estimates:
\begin{itemize}
\item for all $p \in [q,\infty]$ we have $W \in \mathcal{C}((0,\infty) ; \mf{L}^p)$ and there exists a constant $K_{1,0} = K_{1,0}(p,q)$ such that:
\begin{equation} \label{eq_decayw}
\|W(t,\cdot)\|_{\mf{L}^p} \leq K_{1,0}\, t^{\frac{1}{p}- \frac{1}{q}} \|W_0\|_{\mf{L}^q}\,, \quad \forall \, t >0\,,
\end{equation}
\item  if $q < 2$ for all $p \in [q,2),$ we have $(\partial_r w,w/r) \in \mathcal{C}((0,\infty) ; {L}^p(\mathcal F_0))$ and there exists $K_{2,0} = K_{2,0}(p,q)$ such that:
\begin{equation} \label{eq_decaygradw1}
\|\partial_r w(t,\cdot)\|_{L^p(\mathcal F_0)}  + \left\| \dfrac{w(t,\cdot)}{r} \right\|_{L^p(\mathcal F_0)}  \leq K_{2,0}\, t^{-\frac{1}{2} + \frac{1}{p} - \frac{1}{q}}  \|W_0\|_{\mf{L}^q}\,, \quad \forall \, t >0\,,
\end{equation}
\item   if $q \in (1, \infty)$, for all $p \in [\max\{2,q\},\infty)$ satisfying $p>2$  we have $(\partial_r w,w/r) \in \mathcal{C}((0,\infty) ; {L}^p(\mathcal F_0))$ and there exists $K_{3,0} = K_{3,0}(p,q)$ such that:
\begin{equation} \label{eq_decaygradw2}
\|\partial_r w(t,\cdot)\|_{L^p(\mathcal F_0)}   + \left\| \dfrac{w(t,\cdot)}{r} \right\|_{L^p(\mathcal F_0)}  \leq
\left\{
\begin{array}{ll}
 K_{3,0}\,  t^{-\frac{1}{2} + \frac{1}{p} - \frac{1}{q}}  \|W_0\|_{\mf{L}^q}\,, & \forall \, t <1 \,, \\
 K_{3,0}\, t^{ - \frac{1}{q}}  \|W_0\|_{\mf{L}^q}\,, & \forall \, t >1 \,. \\
\end{array}
\right.
\end{equation}
\end{itemize}
\end{theorem}
\begin{proof}
Again, existence of solutions and $\mf{L}^p-\mf{L}^q$ decay estimates are straightforward applications of Theorem \ref{ThmMunnierZuazuaExt}
in the case $k=1$ and $\tilde{\alpha} = \mathcal{J}/2\pi >0.$ We focus now on gradient estimates in the case $p=q\neq 2$, the estimates \eqref{eq_decaygradw1}--\eqref{eq_decaygradw2} with $p \geq q$, $p \neq 2$ being a simple consequence of the semigroup property.

Let $q \in (1,\infty).$  We note that $w \in \mathcal{C}^1((0,\infty);L^q(\mathcal F_0))$ with 
$\|\partial_t W(t)\|_{\mf{L}^q} \leq C/t \|W_0\|_{\mf{L}^q}$ yields 
\begin{eqnarray*}
&&\partial_{rr} w(t) + \dfrac{\partial_r w(t)}{r} - \dfrac{w(t)}{r^2} \in \mathcal{C}((0,\infty) ; L^q(\mathcal F_0)) \\
&& \left\|\partial_{rr} w(t) + \dfrac{\partial_r w(t)}{r} - \dfrac{w(t)}{r^2} \right\|_{L^q(\mathcal F_0)} + |\partial_r w(t,1) - w(t,1)| \leq \dfrac{C}{t}  \|W_0\|_{\mf{L}^q}\,.
\end{eqnarray*}
Recalling estimates \eqref{est_wnabla2} and \eqref{est_w-autres2}, for all $t >0$,
\begin{align*}
	& \| \partial_{rr} w(t)\|_{L^q(\mathcal F_0)} \leq \frac{C}{t} \|W_0\|_{\mf{L}^q},
	\\
	& \Big\|\frac{\partial_r w(t)}{r} \Big\|_{L^q((1,\infty),r{\rm d}r)}+ \Big\|\frac{w(t)}{r^2}\Big\|_{L^q((1,\infty),r{\rm d}r)}   \leq  \frac{C}{t}  \|W_0\|_{\mf{L}^q}  + C(q) \varepsilon_q |w(t,1)|,
\end{align*}
with $\varepsilon_q =1$ if $q >2$ and $\varepsilon_q = 0$ if $q <2$.
But, for $q >2$, estimate \eqref{Decay-W-By-Munnier-Zuazua} with $p = \infty$ yields
$$
	|\ell_W(t)| = |w(t,1)| \leq Ct^{-2/q} \Big\| \Big( \frac{w_0}{r} , \ell_{W_0}\Big) \Big\|_{\mf{L}^q(\mathbb{R}^4)} \leq C t^{-2/q} \|W_0 \|_{\mf{L}^q},
$$
where the last estimate is a consequence of $q >2$.
Hence
$$
	\Big\|\frac{\partial_r w(t)}{r} \Big\|_{L^q((1,\infty),r{\rm d}r)}+ \Big\|\frac{w(t)}{r^2}\Big\|_{L^q((1,\infty),r{\rm d}r)}   \leq  C \left( t^{-1} + \varepsilon_q t^{-2/q} \right) \|W_0\|_{\mf{L}^q}.
$$
We can then bound $\|\partial_{xx} w\|_{L^q(\mathcal F_0)},$ { $\|\partial_{xy} w\|_{L^q(\mathcal F_0)}$} and $\|\partial_{yy} w\|_{L^q(\mathcal F_0)}$ 
in the same way as in the previous proof.  Applying interpolation inequality \eqref{eq_CM2} to $W$ we then obtain that $\partial_r w \in  \mathcal{C}((0,\infty) ; L^q(\mathcal F_0))$ with:
$$
	\|\partial_r w(t)\|_{L^q(\mathcal F_0)} \leq C \left( t^{-1/2} + \varepsilon_q t^{-1/q} \right) \|W_0\|_{\mf{L}^q}.
$$
To get the decay of $w/r$, we then simply use that
$$
	\Big\|\frac{w(t)}{r}\Big\|_{L^q((1,\infty),r{\rm d}r)}^2 \leq \Big\|\frac{w(t)}{r^2}\Big\|_{L^q((1,\infty),r{\rm d}r)}\|W(t)\|_{\mf{L}^q}.		
$$
\end{proof}

\subsection{Decay estimates of solutions to the Stokes system}
It remains now to combine together the results obtained in Subsections \ref{sec_decayvR}, \ref{sec_decayz} and \ref{sec_decayw}  to prove our main results regarding the long-time behavior of Stokes solutions.

\subsubsection{Proof of Theorem \ref{theo Stokes}}

\begin{proof}
Given $q \in (1,\infty),$ as $\mathcal{C}^{\infty}_c(\mathbb R^2) \cap \mathcal{L}^2$ is dense in $\mathcal{L}^q$ we remark that it is
sufficient to prove decay estimate for initial data $V_0 \in \mathcal{L}^2 \cap \mathcal{C}^{\infty}_c(\mathbb R^2).$ We emphasize $V_0 \in \mathcal{L}^q$ for all $q \in [1,\infty]$
under these assumptions. We denote $(W_0,\Phi_0,\Psi_0,V_{R,0})$ the spherical harmonic decomposition of this initial data. 
We already know that there exists a unique solution to (\ref{S1})--(\ref{S-Solideci}) in $\mathcal{C}([0,\infty) ; \mathcal{L}^2)$ for such an initial data. 

\medskip

First, we decompose the solution $S(t) V_0$ of (\ref{S1})--(\ref{S-Solideci}) into the spherical-harmonic decomposition $(W,\Phi,\Psi,V_R)$ of Proposition \ref{prop_decompLp}. According to Proposition \ref{prop_calculsysteme},
this decomposition satisfies:
\begin{itemize}
\item $W \in \mathcal{C}([0,\infty); \mf{L}^2)$ and is a solution to \eqref{eq_w0first}--\eqref{eq_w0last};
\item $V_R \in \mathcal{C}([0,\infty) ; L^2_{\sigma}(\mathcal F_0))$ and is a solution to (\ref{eq_vR1})--(\ref{eq_vR3}).
\end{itemize}
According to {Theorem \ref{thm_decayw}} and {Theorem \ref{Thm-Dan-Shibata}}, these are the respective unique solutions to \eqref{eq_w0first}--\eqref{eq_w0last} and  (\ref{eq_vR1})--(\ref{eq_vR3}) in these spaces,  with respective initial data $W_0$ and $v_{R,0}$. We can then
apply the decay estimates of  {Theorem \ref{thm_decayw}} and {Theorem \ref{Thm-Dan-Shibata}}  to these solutions.  

\medskip

Referring moreover to {Proposition \ref{prop_calculz}} and Remark \ref{rem_defz}, we have:
\begin{itemize}
\item $Z_{\Phi}  \in \mathcal{C}([0,\infty);\mf{L}^2)$ and is a solution to \eqref{eq_1:z}--\eqref{eq_2:z};
\item $Z_{\Psi}  \in \mathcal{C}([0,\infty);\mf{L}^2)$ and is a solution to \eqref{eq_1:z}--\eqref{eq_2:z}. 
\end{itemize} 
Consequently, applying {Theorem \ref{thm_decayz}}, these are the unique solutions to \eqref{eq_1:z}--\eqref{eq_2:z}  in these spaces, with respective initial data
$Z_{\Phi,0} = -(\partial_r \Phi_0 + \Phi_0/r)$ and $Z_{\Psi,0} = \partial_r \Psi_0 + \Psi_0/r\,$ and the decay estimates of {Theorem \ref{thm_decayz}} are also
satisfied by $Z_{\Phi}$ and $Z_{\Psi}.$

\medskip

 We proceed with $L^p-L^q$ estimates. Applying \eqref{eq_decompLp} we have:
 $$
 \|W_0\|_{\mf{L}^q} + \|Z_{\Phi,0}\|_{\mf{L}^q} + \|Z_{\Psi,0}\|_{\mf{L}^q} + \|V_{R,0}\|_{L^q_{\sigma}(\mathcal F_0)} \leq C_q \|V_0\|_{\mathcal{L}^q}\,.
 $$
Then, combining decay estimates of the different components in the spherical-harmonic decomposition obtained in \eqref{Lp-LqR}, \eqref{est_zLp}, and \eqref{eq_decayw}, we have for $t>0$:
 $$
	 \|W(t)\|_{\mf{L}^p} + \|Z_{\Phi}(t) \|_{\mf{L}^p} + \|Z_{\Psi}(t)\|_{\mf{L}^p} + \|V_{R}(t)\|_{L^p_{\sigma}(\mathcal F_0)} \leq C_q t^{\frac{1}{p} - \frac{1}{q}} \|V_0\|_{\mathcal{L}^q}\,.
 $$ 
On the left hand side, we have for instance $Z_{\Psi}(t) = \partial_r \Psi(t) + \Psi(t)/r$ with $\Psi(t,1) = Z_\Psi(t,1)/2$ so that:
$$
	|\Psi(t,1)| \leq  C \|Z_{\Psi}(t)\|_{\mf{L}^p}\,.
$$
 Hence, applying {Proposition \ref{prop_z2psi}} we obtain:
 $$
	 \|\partial_r \Psi \|_{L^p((0,\infty);r{\rm d}r)} + \left\| \dfrac{\Psi}{r}\right\|_{{L^p((0,\infty);r{\rm d}r)}} \leq C \|Z_{\Psi}\|_{\mf{L}^p}\,.
 $$ 
 Applying similar argument to bound  $\Phi,$ we finally obtain that:
 \begin{eqnarray*}
 \|W(t)\|_{\mf{L}^p}+ \|V_{R}(t)\|_{L^p_{\sigma}(\mathcal F_0)} & +&  \|\partial_r \Psi (t)\|_{L^p((0,\infty);r{\rm d}r)}  + \left\| \dfrac{\Psi(t)}{r}\right\|_{{L^p((0,\infty);r{\rm d}r)}}  \\
 &+& 
  \|\partial_r \Phi(t) \|_{L^p((0,\infty);r{\rm d}r)} + \left\| \dfrac{\Phi(t)}{r}\right\|_{{L^p((0,\infty);r{\rm d}r)}}  \leq C_q t^{\frac{1}{p} - \frac{1}{q}} \|V_0\|_{\mathcal{L}^q}\,.
 \end{eqnarray*}
Noting that $\|W(t)\|_{\mf{L}^p}$ is equivalent to $\|W(t)\|_{L^p((0,\infty),r{\rm d}r)}$, we apply \eqref{eq_reciproque} and conclude immediately
$$
\|V(t)\|_{\mathcal{L}^p} \leq C_{p,q} t^{\frac{1}{p} - \frac{1}{q}} \|V_0\|_{\mathcal{L}^q}\,.
$$

We now proceed with the gradient estimates. Let us recall that the case $p= 2$ is a straightforward consequence of \cite{Takahashi&Tucsnak04}  and the previous inequality:
$$
\|\nabla v(t)\|_{L^2(\mc{F}_0)} \leq Ct^{-1/2} \|V(t/2)\|_{\mathcal{L}^2} \leq C_q t^{\frac{1}{2} - \frac{1}{q}-\frac12} \|V_0\|_{\mathcal{L}^q}\,.
$$
So we focus on the case $q \in (1,\infty)$ and $p \in [q,\infty)$ with $p \neq 2$. 

Similarly to the previous computations, the method is then an application of Proposition \ref{prop_decompLp} and the decay estimates obtained in Theorems \ref{Thm-Dan-Shibata}, \ref{thm_decayz}, and \ref{thm_decayw}. The only difference we detail now is the computation of 
$$
\|\partial_{rr} \psi \|_{L^p((1,\infty),r{\rm d}r)} + \left\| \dfrac{\partial_r \psi}{r}\right\|_{{L^p((1,\infty);r{\rm d}r)}} + \left\| \dfrac{\psi}{r^2}\right\|_{{L^p((1,\infty);r{\rm d}r)}} 
$$
from the estimates for $\|\partial_r z_{\psi}\|_{L^p(\mathcal F_0)}$ and $\|z_{\psi}/r\|_{L^p(\mathcal F_0)}$ as given in \eqref{est_gradzLp1} and \eqref{est_gradzLp2}. We focus on $\psi$, the problem being completely similar for $\phi.$
Differentiating the definition of $z_{\psi},$ we remark that $\psi$ satisfies: 
\begin{eqnarray*}
\partial_{rr} \psi + \dfrac{\partial_r \psi}{r} - \dfrac{\psi}{r^2} &=& \partial_r z_{\psi}\,, \qquad   \text{ for }  r \in (1,\infty) \,, \\
\partial_r \psi(1) - \psi(1) &=& 0 \,.   
\end{eqnarray*}
Consequently, we apply Proposition \ref{prop_ellipticw0} which yields:
$$
\|\partial_{rr} \psi \|_{L^p((1,\infty),r {\rm d}r)} + \left\| \dfrac{\partial_r \psi}{r} - \dfrac{\psi}{r^2} \right\|_{L^p((1,\infty),r {\rm d}r)}  \leq C \|\partial_r z_{\psi}\|_{L^p(\mathcal F_0)}\,.
$$
On the other hand, we have, by definition of $z_{\psi}$,
$$
\left\| \dfrac{\partial_r \psi}{r} + \dfrac{\psi}{r^2} \right\|_{L^p((1,\infty),r {\rm d}r)}  \leq  C \left\|\dfrac{z_{\psi}}{r} \right\|_{L^p(\mathcal F_0)} .
$$
So that, we finally get:
\begin{equation*}
\|\partial_{rr} \psi \|_{L^p((1,\infty),r{\rm d}r)} + \left\| \dfrac{\partial_r \psi}{r}\right\|_{{L^p((1,\infty);r{\rm d}r)}} + \left\| \dfrac{\psi}{r^2}\right\|_{{L^p((1,\infty);r{\rm d}r)}} 
\leq C \left( \|\partial_r z_{\psi}\|_{L^p(\mathcal F_0)} + \left\|\dfrac{z_{\psi}}{r} \right\|_{L^p(\mathcal F_0)} \right) \,.
\end{equation*}
This ends the proof of Theorem \ref{theo Stokes}.
\end{proof}

\subsubsection{Duality decay estimates}

For later use,  based on Theorem \ref{theo Stokes}, we derive here additional estimates on the behavior of the semigroup corresponding to \eqref{S1}-\eqref{S-Solideci}:
\begin{corollary}  \label{cor Stokes}
	Assume that  $1 < q \leq p <\infty$ and let  $F \in L^q( \mathbb{R}^2 ;M_{2\times2}(\mathbb R))$ satisfying $F=0$ on $B_0$.
The following decay estimates hold true:
\begin{itemize}
\item if $2\leq q \leq p < \infty,$ there exists $K_4 = K_4(p,q)>0$ such that: 
	\begin{equation}\label{est div 1}
		 \| S(t) \mathbb P \div \, F \|_{\mathcal{L}^p} \leq K_4 (\nu t)^{-\frac{1}{2} + \frac{1}{p} - \frac{1}{q}} \|F\|_{L^q( \R^2)} 
		\qquad \text{for all}\quad  t>0.
	\end{equation}

\item if $1 < q \leq p$ and $q\leq 2$,  there exists $K_5 = K_5(p,q)>0$ such that:  
\begin{equation}\label{est div 2}
\| S(t) \mathbb P \div \, F \|_{\mathcal L^p}  \leq 
\left\{ \begin{split} 
 K_5 (\nu t)^{-\frac{1}{2} + \frac{1}{p} - \frac{1}{q}}  \|F\|_{L^q( \R^2)} &\qquad \text{for all}\quad  0<t< \frac1{\nu}, \\
 K_5 (\nu t)^{ -1+ \frac{1}{p}}  \|F\|_{L^q( \R^2)} &\qquad \text{for all}\quad   t \geq  \frac1{\nu}.
 \end{split}\right.
\end{equation} 
\end{itemize}
\end{corollary}

In this corollary the divergence $\div$ is computed  along rows of the matrix $F$.

Before going into the proof, let us emphasize that, in our case $\nabla S(t)$ is not the dual operator of $S(t) \mathbb P \div.$ Indeed
if $F$ is smooth with compact support, there holds, for all $\Phi \in \mathcal{L}^2 \cap \mathcal{C}^{\infty}_c(\mathbb R^2)$: 
\begin{eqnarray*}
(\nabla S(t) \phi, F) & : =&  \dfrac{m}{\pi} \int_{B_0} \nabla S(t)\Phi : F  + \int_{\mathcal F_0} \nabla S(t) \Phi : F   \\
				&=&  \left(  1 - \dfrac{m}{\pi}\right)\int_{\partial B_0} S(t) \Phi \cdot Fn \, \text{d$\sigma$}  + (S(t)\Phi, \div F)\\
				&=&  \left( 1 - \dfrac{m}{\pi} \right)\int_{\partial B_0} S(t) \Phi  \cdot Fn \, \text{d$\sigma$}  + (\Phi , S(t) \mathbb{P} \div F)\,.
\end{eqnarray*}
Hence Corollary \ref{cor Stokes} only concerns the restriction of the dual of $\nabla S(t)$ to functions $F$ which vanish at the boundary.

\begin{proof}
The following proof contains a construction of the operator $S(t) \mathbb P \div$ on the closed
 subset  of $L^q(\mathbb R^2 ; M_{2\times2}(\mathbb R))$ of functions vanishing on $B_0.$ 
We prove our result in the case $2\leq q \leq p < \infty$ only. The other cases can be done similarly.  

Let $F \in L^q(\mathbb R^2 ; M_{2\times2}(\mathbb R^2))$ such that $F=0$ on $B_0$.
Up to a regularizing argument, we assume that $F \in \mathcal{C}^{\infty}_c(\mathcal F_0 ; M_{2\times2}(\mathbb R)).$ 
Then, $V(t) := S(t) \mathbb P \div F \in \mathcal{C}((0,\infty) ; \mathcal{L}^p)$ for all $p  \in (1,\infty)$ by a straightforward
application of {Theorem \ref{theo Stokes}}.  For all $t>0$ and $\tilde{V} \in \mathcal{L}^{2} \cap \mathcal{C}^{\infty}_c(\mathbb R^2),$ we have, as $S$ is self-adjoint with respect to the
scalar product $(\cdot,\cdot)$ we introduced on $\mathcal{L}^2$ (see \eqref{eq_ps}):
\begin{eqnarray*}
\langle V(t) , \tilde{V} \rangle_{\mathcal L^{p},\mathcal{L}^{p'}}  
&=& ( \mathbb{P} \div F , S(t)\tilde{V} )  \,, \\[4pt]
&=& \int_{\mathcal{F}_0}  \div F  \cdot S(t) \tilde{V}\,, \qquad \text{ (as $F$ vanishes on $B_0$)\,,}\\[2pt]
&=&
-\int_{\mathcal{F}_0} F   : \nabla S(t) \tilde{V}  \quad \phantom{As} \text{ (as $Fn$ vanishes on $\partial B_0$)\,.}
\end{eqnarray*}
Finally, we obtain:
$$
\left|\langle V(t) , \tilde{V} \rangle_{\mathcal L^{p},\mathcal{L}^{p'}} \right| \leq \|F\|_{L^q(\mathbb R^2)} \| \nabla S(t) \tilde{V}\|_{L^{{q'}}(\mathcal F_0)}\,,
$$
where we apply decay estimates we obtained in Theorem \ref{theo Stokes}: as ${p'} \leq  {q'} < 2$ we have from \eqref{est grad 2}  
$$
\|\nabla S(t) \tilde{V} \|_{L^{{q'}}(\mathcal F_0)} \leq C t^{-\frac{1}{2} + \frac{1}{{q'}} - \frac{1}{{p'}}} \|\tilde{V}\|_{\mf{L}^{{p'}}} \leq C t^{-\frac{1}{2} + \frac{1}{p} - \frac{1}{q}} \|\tilde{V}\|_{\mathcal{L}^{{p'}}}.
$$
So that, we obtain:
$$
\left|\langle V(t) , \tilde{V} \rangle_{\mathcal L^{p},\mathcal{L}^{p'}}\right| \leq  C \|F\|_{L^q(\mathbb R^2)}  t^{-\frac{1}{2} + \frac{1}{p} - \frac{1}{q}} \|\tilde{V}\|_{\mathcal{L}^{{p'}}}.
$$
As $\mathcal{L}^{p'} \cap \mathcal{C}^{\infty}_c(\mathbb R^2)$ is dense in $\mathcal{L}^{{p'}},$ this inequality implies by duality that $V(t) \in \mathcal{L}^p$ with norm
lower than $C t^{-\frac{1}{2} + \frac{1}{p} - \frac{1}{q}}  \|F\|_{L^q(\mathbb R^2)}\, .$
\end{proof}

In the previous corollary we restrict $p$ to finite values. In the case $p=\infty,$ we do not obtain a control of the whole solution. Nevertheless, we can obtain a result that would correspond to the case $p = \infty$ in \eqref{est div 1} for the translation speed $\ell_{V(t)}.$
This result is a new application of the added mass effect and relies on the fact that Kirchoff potentials are easily computed in our case.

\begin{corollary} \label{est_lv}
	Let $q \in [2,\infty)$ and $F \in L^q( \mathbb{R}^2 ;M_{2\times2}(\mathbb R))$ satisfying $F=0$ on $B_0$,
	The following decay estimate holds true for $V(t) := S(t)\mathbb P \div F$:
	\begin{equation} \label{eq_lv}
		|\ell_{V(t)}| \leq {K_{\ell}}(q)(\nu t)^{-\left(\frac 12 + \frac 1q \right)} \|F\|_{L^q(\mathbb R^2)}\, , \quad \forall \, t >0
	\end{equation}
	where $K_{\ell}(q)$ depends only on $q$. 
\end{corollary}

\begin{proof}
Let the assumptions of the corollary be satisfied.
At first, we recall that we have $V(t) \in \mathcal L^q$ for all $t \in (0,\infty)$ as has been shown in the previous
corollary. We show how to prove that the first component $\ell_{V,1}$ of $\ell_V(t)$ satisfies 
\eqref{eq_lv}. Similar estimate for the other component $\ell_{V,2}$ is obtained applying comparable arguments.

\medskip

Let  $\bar{\psi} \in \mathcal C^{\infty}(\mathcal F_0)$ be given in polar coordinates by:
$$
	\bar{\psi}(r,\theta) = \dfrac{\cos(\theta)}{r} \,, \quad \forall \, (r,\theta) \in (1,\infty) \times (-\pi,\pi)\,.
$$
Given $t>0$ we note that $V := V(t) \doteq ( v(t), \ell_V(t))$ is divergence free on any subdomain $B(0,R) \setminus B_0$ of $\mathcal F_0.$
This yields:
$$
\int_{\partial B(0,R)} v \cdot n \, \bar{\psi} \, \text{d$\sigma$} +  \int_{\partial B_0} v \cdot n \, \bar{\psi} \, \text{d$\sigma$}  = \int_{B(0,R) \setminus B_0} v \cdot \nabla \bar{\psi}
$$
Letting $R \to \infty,$ we obtain (the exterior boundary term vanishes as $v \in \mathcal L^q$):
\begin{equation} \label{eq_propxi}
	-\int_{0}^{2\pi} v_r (1, \theta) \cos(\theta)\, \text{d$\theta$}  = \int_{\mathcal F_0} v \cdot \nabla \bar{\psi}\,.
\end{equation}
We observe then that, on the one hand, we have $\nabla \bar{\psi} = \nabla^{\bot} \bar{\phi}$ where 
$$
	\bar{\phi}(r,\theta) = \dfrac{\sin(\theta)}{r} \,, \quad \forall \, (r,\theta) \in (1,\infty) \times (-\pi,\pi)\,,
$$ 
on the other hand:
$$
	\int_{0}^{2\pi} v_r(t,1, \theta) \cos(\theta)\, \text{d$\theta$}  = \pi \ell_{V,1}(t)\,.
$$
Setting finally:
$$
	\Xi := 1_{\mathcal F_0}  \nabla \bar{\psi} - 1_{B_0} e_1
$$
we have that $\Xi \in \mathcal L^p$ for arbitrary $p>1$ and that \eqref{eq_propxi} reads:
\begin{equation*} 
	-(\pi+m) \ell_{V,1}(t)  = m \ell_V(t) \cdot \ell_{\Xi} +  \int_{\mathcal F_0} v(t) \cdot \xi  = (V(t),\Xi)\,, \quad \forall \, t >0\,.
\end{equation*}

Given this identity, we reproduce the computations done in the proof of the previous corollary.
We obtain:
\begin{equation}
	\label{Added-Mass}
	-(\pi + m) \ell_{V,1}(t) = \int_{\mathcal F_0} F : \nabla S(t) \Xi\,,
\end{equation}
which implies 
$$
|\ell_{V,1}(t)|  \leq \dfrac{1}{|\pi + m|} \|F\|_{L^q(\mathbb R)} \|\nabla S(t) \Xi\|_{L^{q'}(\mathcal F_0)}\,.
$$

\medskip

The proof now reduces to find a bound on $\| \nabla S(t) \Xi\|_{L^{q'}(\mathcal F_0)}.$ 
To this end, we remark that the spherical-harmonic decomposition of $\Xi$ reduces to the first mode 
$\bar \phi(r, \theta) = \min\{r,1/r\} \sin(\theta).$ Going back to the computation of Section 2, we note that $S(t) \Xi$ is given by its first-mode, corresponding to $Z_{\Phi} =- \partial_r \Phi - \Phi/r$ where $\Phi(r) = \min\{r, 1/r\}$.   This mode satisfies \eqref{eq_1:z}-\eqref{eq_2:z} with initial condition:
$$
	Z_{\Phi,0} = - 2\cdot 1_{B_0} \,. 
$$
Consequently, for $q >2$, $Z_{\Phi,0} \in \mathcal L^1$ and we apply Theorem \ref{thm_decayz} with ``$p$''$=q'$ and ``$q$''$=1$, which yields (see \eqref{est_gradzLp1})
$$
	\|\partial_r z_{\Phi}(t,\cdot)\|_{L^{q'}(\mathcal F_0)}  + \left\| \dfrac{z_{\Phi}(t,\cdot)}{r} \right\|_{L^{q'}(\mathcal F_0)}  \leq K\,  (\nu t)^{-\frac 12+ \frac 1{q'} - 1}=K\,  (\nu t)^{-\frac 12- \frac 1{q} } \,, \quad \forall \, t >0\,.
$$

We  go back  to $\Xi$ as in the proof of Theorem {\ref{theo Stokes}} so that:
$$
	\|\nabla S(t) \Xi\|_{L^{q'}(\mathcal F_0)} \leq K\, (\nu t)^{-\frac 12 - \frac 1q }\,, \quad \forall \, t >0\,.
$$ 
Plugging this estimate in \eqref{Added-Mass} we obtain the expected result for $q >2$.

The case $q=2$, corresponding to $q'=2$, does not immediately follows from Theorem \ref{thm_decayz}, but rather from the fact that
$$
	\| \nabla^\perp \left( \Phi\left(\frac{t}{2}, r \right) \sin( \theta) \right) \|_{\mc{L}^2} \leq C \| Z_\Phi\left( \frac{t}{2}, \cdot \right)\|_{\mf{L}^2} \leq C (\nu t)^{-1/2} \| Z_{\Phi,0} \|_{\mf{L}^1}
$$
and from the $\mc{L}^2$ decay estimates on the gradient obtained in Theorem \ref{theo Stokes}:
$$
	\|\nabla S(t/2) \left(\nabla^\perp \left( \Phi\left(\frac{t}{2}, r \right) \sin(\theta) \right) \right)\|_{L^2(\mathcal{F}_0)}  \leq C (\nu t)^{-1/2} \| \nabla^\perp \left( \Phi\left(\frac{t}{2}, r \right) \sin \theta) \right) \|_{\mc{L}^2}.
$$
\end{proof}

\subsection{Asymptotic expansion of solutions to the Stokes system} This section aims at proving Theorem \ref{thm_AsymptoticStokes}. We first show that the solutions $W$ and $V_R$ corresponding to the modes $k \neq 1$ decay faster than the modes corresponding to $k = 1$. In a second step, we derive precisely the first-order in the long-time behavior of this first mode.

\subsubsection{Faster decay on $W$}

\begin{theorem} \label{thm_w0firstorder}
Given a radial $W_0 \in \mf{L}^1 \cap  L^2(\R^2,\exp(|x|^2/4){\rm d}x),$ the unique solution to \eqref{eq_w0first}--\eqref{eq_w0last}
satisfies $W(t) \in \mf{L}^p$ for all $t >0$ and $p\in [1,\infty].$ Furthermore, for all $p \in [2,\infty],$  there exists a constant $C_p >0$ such that
$$
	\|W(t)\|_{\mf{L}^p} \leq C_p t^{\frac{1}{p} - \frac{3}{2} } \|W_0 r\|_{\mf{L}^1}\,,
$$ 
and there exists a constant $C$ such that
$$
 |\ell_{W(t)}|  \leq C t^{-2} \|W_0 r\|_{\mf{L}^1}.
$$
\end{theorem}

\begin{proof}
 Let us first remark that $W_0 \in \mf{L}^1 \cap  L^2(\R^2,\exp(|x|^2/4){\rm d}x)$ obviously implies that $W_0 r \in \mf{L}^1$.

We first focus on the decay estimate, as $W(t) \in \mf{L}^p$ for $t >0$ is obvious. Given $t>0$ and $p \geq 2$ we apply \eqref{eq_decayw} 
with $q=2$  between $t/2$ and $t$ and then, we apply \eqref{Decay-W-By-Munnier-Zuazua} between $0$ and $t/2$:
$$
\|W(t)\|_{\mf{L}^p} \leq C  t^{\frac{1}{p} - \frac{1}{2}}\|W(t/2)\|_{\mf{L}^2} \leq C t^{\frac{1}{p} - \frac{3}{2}} \Big\| \frac{W_0}{r} \Big\|_{\mf{L}^1(\mathbb{R}^4)} \leq C t^{\frac{1}{p} - \frac{3}{2}} \|W_0 r\|_{\mf{L}^1}\,.
$$
Concerning the second estimate, it suffices to use \eqref{l-omega} and \eqref{Decay-W-By-Munnier-Zuazua}:
$$
	 |\ell_{W(t)}| \leq  \Big\| \frac{W(t)}{r} \Big\|_{\mf{L}^\infty(\mathbb R^4)} \leq  Ct^{-2} \Big\| \frac{W_0}{r} \Big\|_{\mf{L}^1(\mathbb R^4)} = C t^{-2}  \|W_0 r\|_{\mf{L}^1}.
$$
This ends the proof.
\end{proof}

As a consequence, we can already note that $\omega_{S(t)V_0}=\ell_{W(t)}$ (see Proposition \ref{prop_decompLp}) verifies \eqref{eq_asymStokes4}.

\subsubsection{Faster decay on $V_R$}\label{sec_decayvR2}
\begin{theorem} \label{thm_vRfirstorder}
Given $V_{R,0} \in L^2_{\sigma}(\mathcal F_0)\cap L^2(\R^2,\exp(|x|^2/4){\rm d}x)$, for all $p \in [ 2,\infty]$,  there exists a constant $C= C(p,v_{R,0})$ such that:
$$
	\|V_R(t,\cdot)\|_{L^p_\sigma( \mathcal{F}_0)}  \leq C \dfrac{|\log(t)|}{t^{3/2-1/p}}  \,, \quad \forall \, t >1\,. 
$$ 
\end{theorem}

\begin{proof}
	 In order to prove Theorem \ref{thm_vRfirstorder}, we expand $V_R$ solution of \eqref{eq_vR1}--\eqref{eq_vR3} on its Fourier basis:
	$$
		V_R(t,r,\theta) = \nabla^{\perp} \left[  \sum_{k \geq 2} \left( \psi_k(t,r) \cos(k\theta) +  \phi_k(t,r) \sin(k\theta) \right)\right], 
	$$
	where $\psi_k(t,1) = \partial_r \psi_k(t,1)= \phi_k(t,1) = \partial_r \phi_k(t,1) = 0$ thanks to the homogeneous Dirichlet boundary conditions satisfied by the restriction $v_R$ of $V_R$ on $\mc{F}_0$. Note that $v_R$ does not contain any $0$ or $1$ mode due to the orthogonality condition \eqref{eq_vR}.

	As in the case $k=0,1$, we can show that for all $k  \geq 2$ the new unknown $z_k = z_{\psi,k}(t,r) := 1/r^{k}  \partial_r[ r^k \psi_k(t,r)]$ or $z_k = z_{\phi,k}(t,r) =  -1/r^{k}\partial_r[ r^k \phi_k(t,r)]$ satisfies:
	\begin{eqnarray*}
			\partial_t z_k + \left( - \frac{1}{r} \partial_r (r \partial_r z_k)  + \frac{(k-1)^2}{r^2} z_k \right) = 0 & &  	\qquad \text{ for } (t,r) \in (0, \infty) \times (1, \infty) ;
			\\
			z_k(t,1) =  0 & & \qquad \text{ for } t \in (0, \infty).
	\end{eqnarray*}
	One can then use the asymptotic formula given by Theorem \ref{thm_munnierzuazua_Exp} for 
	$$
		\tilde{v}_R(t,r,\theta) = \sum_{k \geq 2} \left( z_{\psi,k}(t,r) \cos((k-1) \theta) + z _{\phi, k}(t,r) \sin((k-1)\theta)\right)  
	$$
	which is a solution of \eqref{Eq-on-V} for $n=2,$ arbitrary $\alpha>0$ and vanishing initial mass $M = 0$. This immediately yields that, provided 
	\begin{equation*}
		\tilde v_{R}(0) \in L^2(\mathcal{F}_0,\exp(|x|^2/4)\, {\rm d}x),
	\end{equation*}
	which holds true since  $V_{R,0}$ is assumed to belong to $  L^2(\R^2,\exp(|x|^2/4){\rm d}x)$,
	we have
	$$
		t^{1-1/p} \norm{\tilde{v}_R(t,\cdot)}_{L^p(\mathcal{F}_0)} \leq C R_{1, p}(t).
	$$
	In particular, for $p = 2$, this implies that
	$$
		t \sum_{k \geq 2} \int_1^\infty r (|z_{\psi,k}(t,r)|^2+|z_{\phi,k}(t,r)|^2 \, dr \leq C R_{1,2}(t)^2.
	$$
	But recall that, for $k \geq 2$, 
	$$
		z_{\psi,k} = \partial_r \psi_k + \frac{k}{r} \psi_k, 
	$$
	and $\psi_k(t,1) = 0$. Hence for all $R>1$, 
	$$
		\int_1^R |z_{\psi,k}|^2 \, r dr = \int_1^R |\partial_r \psi_k|^2 r dr + k^2 \int_1^R \frac{|\psi_k|^2}{r}\, rdr + k \int_1^R \partial_r \left(|\psi_k|^2 \right)\, dr.
	$$
	As $\psi_{k}(1)=0$, passing to the limit $R \to \infty$, we get
	$$
		\int_1^\infty |\partial_r \psi_{k}|^2 r dr +   k^2 \int_1^\infty \frac{|\psi_k|^2}{r}\, rdr \leq \int_1^\infty |z_{\psi,k}|^2 \, r dr, 
	$$	
	and thus,
	\begin{equation*}
		t^{1/2} \norm{v_R(t)}_{L^2(\mathcal{F}_0)} \leq C R_{1,2}(t).
	\end{equation*}
	Using then the semigroup estimates \eqref{Lp-LqR}, 
	 we get, for $p \geq 2$,
	\begin{equation*}
		t^{1-1/p} \norm{v_R(t)}_{L^p(\mc{F}_0)} \leq C R_{1,2}(t).
	\end{equation*}
	This concludes the proof of Theorem \ref{thm_vRfirstorder}, as $V_R$ simply vanishes in $B(0,1)$.
\end{proof}

\subsubsection{Proof of Theorem \ref{thm_AsymptoticStokes}}

\begin{proof}
Let $V_0 \in \mathcal{L}^1 \cap  L^2(\R^2, \exp(|x|^2/4){\rm d}x)$ and $V(t)$ the unique associated solution to   (\ref{S1})--(\ref{S-Solideci}).
Note that $V_0 \in \mathcal{L}^q$ for all $q \in (1,\infty)$ so that we already know that $V(t) \in \mathcal{L}^p$ for all $p \in [2,\infty)$ for all $t>0$
from Theorem \ref{theo Stokes}. Let now $p \in [2,\infty)$ and $(W,\Phi,\Psi,V_R)$ the spherical-harmonic decomposition of $V$. 

The components  $\Phi,\Psi$ and $W$ are computed as means of $V$ in $\theta$ so that they inherit the asymptotic decay in $r$ of the data $V_0.$ 
Combining this remark with Proposition \ref{prop_decompLp}, this yields that:
\begin{eqnarray*}
&& W(0,\cdot) \in \mf{L}^1 \cap L^2(\R^2,\exp(|x|^2/4)\, {\rm d}x)\,,  \quad \quad  V_R(0,\cdot) \in L^2_{\sigma}(\mc{F}_0)\cap  L^2(\R^2,\exp(|x|^2/4){\rm d}x)\,\\[4pt]
&& (Z_{\Phi}(0,\cdot), Z_{\Psi}(0,\cdot))  \in \mf{L}^1 \cap  L^2(\R^2,\exp(|x|^2/4)\, {\rm d}x) \,.
\end{eqnarray*}
Consequently, {Theorems \ref{thm_vRfirstorder}} and \ref{thm_w0firstorder} imply respectively:
$$
\|V_R(t,\cdot)\|_{L^p_{\sigma}(\mc{F}_0^2)} = \mathcal{O}\left(\dfrac{|\log(t)|}{t^{3/2-1/p}}\right)  \,,
\qquad
 \|W(t,\cdot)\|_{\mf{L}^p} =  \mathcal{O} (t^{1/p-3/2})\,.
$$
We focus now on $Z_{\Phi}$ and $Z_{\Psi}.$ Using Theorem \ref{thm_munnierzuazua_Exp} with $\alpha= 2/(\pi+m)$, we immediately get:
\begin{eqnarray}
	t^{1- 1/p} \norm{z_\Phi(t,\cdot) - M_\Phi G(t) }_{L^p(\mathcal F_0)}  \leq   C_p R_{1,p}(t),
	& &
	t \left|\ell_{Z_\Phi} (t) - \frac{M_\phi}{4 \pi t} \right| \leq  C R_2(t),\label{Asymptotic-z1}
	\\
	t^{1- 1/p} \norm{z_\Psi(t,\cdot) - M_\psi G(t) }_{L^p(\mathcal F_0)}  \leq  C_p R_{1,p}(t),
	& &
	t \left|\ell_{Z_\Psi}(t) - \frac{M_\psi}{4 \pi t} \right|  \leq  C R_2(t), 
	\label{Asymptotic-z2}
\end{eqnarray}
with $G$ and $(R_{1,p},R_2)$ as given in Theorem \ref{thm_munnierzuazua_Exp} in the case $n=2$ and 
\begin{eqnarray*}
	M_{\Phi} & :=  & 2\pi \int_{1}^{\infty} Z_{\Phi}(0,r)r {\rm d}r + (\pi + m) \int_0^1 Z_{\Phi}(0,r) r {\rm d}r \,, 
	\\
	M_{\Psi} & :=  & 2\pi \ \int_{1}^{\infty} Z_{\Psi}(0,r)r {\rm d}r +(\pi + m) \int_0^1 Z_{\Psi}(0,r) r {\rm d}r  \,. 
\end{eqnarray*}
Recalling that 
\begin{align*}
	& Z_{\Phi}(t,\cdot) = - \frac1r \partial_r (r \Phi (t,\cdot) ) \quad \hbox{ for } (t,r) \in [0,\infty) \times (0,\infty);
	\\
	& Z_\Phi(t,r ) = \ell_{Z_{\Phi}}(t) = - 2  \Phi(t,r)/r =  2\ell_1(t)\quad \hbox{ for } (t,r) \in (0,\infty) \times (0,1),\\
	& \Phi(0,1 ) = -\ell_{1}(0) , \quad \text{(by the continuity of $\Phi$)}
\end{align*}
and using $\Phi(0,\cdot)/r, \partial_r \Phi(0,\cdot) \in L^1\cap L^2((0,\infty),\exp(|r|^2/4)r {\rm d}r)$, which implies the existence of a sequence $R_n \to \infty$ such that $R_n \Phi(0,R_n)$ goes to $0$ as $n \to \infty$,
$$
	M_\Phi = (\pi-m) \Phi(0,1)= (m- \pi) \ell_{1}(0).
$$
Similarly,
$$
	M_\Psi = (m-\pi) \Psi(0,1)=(m- \pi) \ell_{2}(0).
$$
By Proposition \ref{prop_decompLp} , we recall for $r\in (0,1)$ that $\ell_{S(t)V_0,1}= - \Phi(t,r)/r =\ell_{Z_\Phi}(t )/2$. From \eqref{Asymptotic-z1}-\eqref{Asymptotic-z2} and the previous formulas, we have obtained \eqref{eq_asymStokes2} and \eqref{eq_asymStokes5}.

Solving $\Phi$ and $\Psi$ in terms of $Z_{\Phi}(t,\cdot) \doteq ( z_\Phi(t,\cdot), \ell_{Z_{\Phi}}(t))$ and  $Z_{\Psi}(t,\cdot) \doteq ( z_\Psi(t,\cdot), \ell_{Z_{\Psi}}(t))$, we are then led to define $\hat{\Psi}(t,r)$ on $t \geq 0$, $r \in (0, \infty)$ as the extension of $\hat\psi$ solution of 
$$
	\frac1r \partial_r (r \hat{\psi} (t,r) ) = G(t,r) \quad \text{ for } (t,r) \in (0,\infty) \times (1, \infty), \qquad \hat{\psi}(t,1) = \frac{1}{8 \pi t} \quad \text{ for } t \in (0,\infty).
$$
by 
$$
	\hat{\Psi}(t,r) = \frac{r}{8 \pi t} \quad \text{ for } (t,r) \in (0,\infty) \times (0,1).
$$
Note that this function can be computed explicitly: 
\begin{equation*}
	\hat{\Psi}(t,r) = 
		\left\{
			\begin{array}{ll}
				\ds \frac{1}{2\pi r} \left(\exp\left( - \frac{1}{4t} \right)- \exp\left( - \frac{r^2}{4t} \right)+ \frac{1}{4t} \right)\, \quad & \hbox{ for } (t,r) \in (0,\infty)\times (1,\infty),
				\\
				\ds \frac{r}{ 8 \pi t}\quad & \hbox{ for } (t,r) \in (0,\infty)\times (0,1),
			\end{array}
		\right.
\end{equation*}

Using Proposition \ref{prop_z2psi}, we get for all $p>1$,
\begin{eqnarray*}
 \| \partial_r \psi(t, \cdot) - M_\Psi \partial_r \hat \psi(t,\cdot) \|_{L^p(\mc{F}_0)} +\Big\| \frac{\psi(t, \cdot) - M_\Psi \hat \Psi(t,\cdot)}{r} \Big\|_{L^p(\mc{F}_0)}  & \leq & C_p\Bigl( R_{1, p}(t)t^{1/p-1} + R_{2}(t) t^{-1}\Bigl),
	\\
	 \| -\partial_r \phi(t, \cdot) - M_\Phi \partial_r \hat{\psi}(t,\cdot) \|_{L^p(\mc{F}_0)} +\Big\| \frac{-\phi(t, \cdot) - M_\Phi \hat \Psi(t,\cdot)}{r} \Big\|_{L^p(\mc{F}_0)}  & \leq & C_p\Bigl( R_{1, p}(t)t^{1/p-1} + R_{2}(t) t^{-1}\Bigl).
\end{eqnarray*}
With the expression of $R_{1,p}$ and $R_2$, we can check that $-\frac12 + \theta_{2,p}+\frac1p-1>-\frac14-1$ for all $p \in [2,\infty]$. Hence for $t>1$, we have:
\begin{eqnarray*}
	t^{1-1/p} \|\nabla^\perp \left( \psi(t, \cdot) \cos(\theta)- M_\Psi \hat{\psi}(t,\cdot)\cos(\theta)\right) \|_{L^p(\mc{F}_0)} & \leq & 2C_p R_{1, p}(t),
	\\
	t^{1-1/p} \|\nabla^\perp \left( \phi(t, \cdot)\sin(\theta)  + M_\Phi \hat{\psi}(t,\cdot)\sin(\theta) \right)\|_{L^p(\mc{F}_0)} & \leq & 2C_p R_{1, p}(t).
\end{eqnarray*}
Remark then that, denoting
	\begin{equation*}
		\tilde \psi(t,r) = 
 		\frac{1}{2\pi r} \left(1- \exp\left( - \frac{r^2}{4t} \right) \right), \, \quad  \hbox{ for } (t,r) \in (0,\infty)\times (1,\infty),
	\end{equation*}
we have for all $p \in (1, \infty]$,
$$
	\| \nabla^\perp ( (\tilde \psi(t,r) - \hat{\psi}(t,r) )\cos (\theta)) \|_{L^p(\mathcal F_0)}  \leq C \left|1 - \exp\left( - \frac{1}{4t} \right) - \frac{1}{4t}\right| \leq \frac{C_p}{t^2}.
$$

We then obtain
$$
	t^{1-1/p} \| v(t) - \nabla^{\perp} \left[ (m-\pi) \tilde{\psi}(t,\cdot) (\ell_2(0) \cos(\theta) - \ell_1(0) \sin(\theta))\right]  \|_{L^p(\mathcal F_0)} \leq  C_p R_{1,p}(t).
$$
This yields the expected result.
\end{proof}

\section{Long-time behavior of solutions to the Navier-Stokes problem} \label{sect3}

In this section, we prove Theorem \ref{Thm-q-dans(1,2)} and Theorem \ref{theo kato}. We first apply Kato's method \cite{Kato84} of successive
approximations yielding decay estimates for initial data $V_0 \in \mathcal{L}^2$. In a second subsection, we then extend these estimates to
the case of initial data $V_0 \in \mathcal{L}^q$ with $q \in (1,2]$ in order to get Theorem \ref{Thm-q-dans(1,2)}. We finally explain how a bootstrap argument yields Theorem \ref{theo kato}.

To simplify notations, we replace the constants $K_1(p,q)$, $K_4(p,q)$ and $K_\ell(q)$ defined respectively in \eqref{Lp-Lq}, \eqref{est div 1} and \eqref{eq_lv} by $K_1(p,q) \nu^{1/p - 1/q}$, $K_4(p,q) \nu^{-1/2+1/p-1/q}$, and $K_\ell(q) \nu^{-1/2- 1/q}$, so that the viscosity parameter will not appear in our computations.

\subsection{$\mc{L}^p$ decay estimates for $\mathcal L^2$ initial data}\label{sect31}
We recall that we transferred our system in the body frame applying the change of variable \eqref{translation}.
So, the equations  (\ref{NS1})-(\ref{Solideci})  became \eqref{NS11}-\eqref{Solideci1}. Our first proposition reads:

\begin{proposition}\label{prop kato} Let   $V_{0}\in \mc{L}^2$. There exists $\lambda_{0}>0$ such that, if 
\begin{equation}
	\label{Smallness-v0}
	\|  V_{0} \|_{\mc{L}^2}   \leq \lambda_{0},
\end{equation}
then, the unique global weak solution $V$ of \eqref{NS11}-\eqref{Solideci1} satisfies the following: for all $p \in [2,\infty),$ 
there exists constants  $ H(p,\lambda_0)$ and $H_\ell(\lambda_0)$ such that
\begin{equation*}
	\sup_{t >0} t^{\frac12-\frac1p}  \| V(t) \|_{\mc{L}^p} \leq H(p, \lambda_{0}), \qquad \hbox{ and } \qquad \sup_{t >0} t^{\frac12} |\ell_V(t)| \leq H_\ell(\lambda_{0}),
\end{equation*}
\end{proposition}

\begin{proof} We split the proof of Proposition \ref{prop kato} into six steps.

\medskip

{\bf Step 1: integral formulation.} Following \cite{Takahashi&Tucsnak04}, we rewrite the Navier-Stokes equations \eqref{NS11}-\eqref{Solideci1} in the following abstract form:
\[
	\partial_{t}  V + A V = \mathbb{P} F
\]
where 
$$
	F (V)= \left\{
	\begin{array}{rcll}
		(\ell_V- V)  \cdot \nabla V  & \text{ on $\mathcal{F}_{0}$}\,\\
		 0 &\text{ on $B_{0}$},
	 \end{array}\right.
$$ 
$ \mathbb{P}$ denotes the continuous projector from $L^p$ to $\mc{L}^p$, and $\ell_{V}$ is defined for $V\in \mc{L}^p$ by \eqref{l-omega}. Then, Duhamel formula gives the following integral formulation of the above equations:
\begin{equation}\label{step 1}
	 V(t) = S(t)  V_{0} + \int_{0}^t S(t-s) \mathbb{P}F(V(s))\, {\rm d}s.
\end{equation}

T. Kato suggests to construct a solution by successive approximations: let the sequence $(y_n)_{\in\mathbb N}$  be defined by
\begin{equation}
	\label{KatoIteration}
\left\{
\begin{array}{rcll}
 	Y_0 &:=& S(t) V_0\,, & \\ 
	 Y_n &:=& S(t)  V_{0} + \mathcal{K} Y_{n-1} \,, & \forall \, n \in \mathbb N \,, 
\end{array} 
 \right.
 \quad \hbox{ where } \quad \mathcal{K} Y (t)= \int_{0}^t S(t-s) \mathbb{P} F(Y)(s)\, {\rm d}s.
 \end{equation}
Our aim is to prove that this sequence satisfies 
uniformly estimates of Proposition \ref{prop kato} and converges for small initial data. To simplify notations, in the following we set $\ell_{Y_n} = \ell_n$.

\medskip 

Concerning the nonlinear term, we note that, $\mathbb P F(Y)$ is well-defined as soon as $Y \in \mathcal L^p$ for $p > 2$ satisfies $\nabla Y \in L^2(\mathcal F_0).$ Indeed, we can then split $F(Y)|_{\mc{F}_0} = -Y \cdot \nabla Y + \ell_Y \cdot \nabla Y,$ the first term being in $L^q(\mathcal F_0)$ (where $q = 2p/(2+p)$) and the second one in $L^2(\mathcal F_0).$  We have then :
$$
\mathbb P F(Y) = -\mathbb P_q [1_{\mathcal F_0} Y \cdot \nabla Y ] + \mathbb P_2[1_{\mathcal F_0}  \ell_Y \cdot \nabla Y].
$$
Furthermore, we remark that, if $Y \in \mathcal L^{p_0}$ (with $p_0 \in [1,\infty)$) satisfies $y \in H^1(\mathcal F_0)$ then:
$$
F(Y) = \div \tilde{F}(Y) \text{ where } 
\tilde{F}(Y)= \left\{
\begin{array}{rcll}
	(\ell_Y-Y)  \otimes  Y   & \text{ on $\mathcal{F}_{0}$}\,,\\
	 0 &\text{ on $B_{0}\,.$}
 \end{array}\right.
$$
This property is satisfied since $\tilde{F}(Y)n$ vanishes on $\partial B_0$ as $B_0$ is a disk.
The operator $\mathcal{K}$ can then be defined indifferently as:
$$
	 \int_{0}^t S(t-s) [\mathbb{P} F(Y)(s)]\, {\rm d}s  \quad \text{ or }  \quad  \int_{0}^t [S(t-s) \mathbb{P} \div ] \tilde{F}(Y)(s) \, {\rm d}s \,,
$$
where $S(t-s) \mathbb{P} \div$ is defined by duality.
In order to get uniform estimates on the functions $Y_n$ and their limit, we work with the second form (Step 2 to 5). In Step 6, we apply the first form to prove that our construction coincides with the unique global weak solution constructed
in  \cite{Takahashi&Tucsnak04}. 

\bigskip

{\bf Step 2: estimates of $t^{\frac38} \|Y_{n}\|_{\mc{L}^8},\  \|Y_{n}\|_{\mc{L}^2},\ t^{\frac 12} |\ell_{n}(t) |$.} The goal of this step is to show the following Lemma:

\begin{lemma}\label{lem Kato1}
There exists a constant $\lambda_0>0$ such that for all $V_0\in \mc{L}^2$ satisfying \eqref{Smallness-v0} there exists $\mu_0>0$ such that:
 \[
	 \sup_{t>0} \{t^{\frac 38} \|Y_{n}(t)\|_{\mc{L}^8}\} \leq \mu_0 ,\quad  \sup_{t>0} \{ \|Y_n(t)\|_{\mc{L}^2}\} \leq \mu_0 \,, \quad \sup_{t>0} {t^{\frac 12}|\ell_{n}(t)|} \leq \mu_0.
 \]
Besides, $\mu_0$ can be chosen arbitrary small, independent of $V_0$, up to restrict the size  of $\lambda_0.$
\end{lemma}

\begin{proof}
We are going to find by induction a sequence $G_n$ such that for all $n$,
\begin{eqnarray}
	\sup_{t >0} \{ t^{\frac12-\frac18}  \| Y_{n}(t) \|_{\mc{L}^8}\}  = \sup_{t>0} t^{\frac38}  \| Y_{n}(t) \|_{\mc{L}^8}  &  \leq & G_{n} \label{Gn},
	\\
	\sup_{t >0} \{ t^{\frac12-\frac12}  \| Y_{n}(t) \|_{\mc{L}^2} \} = \sup_{t>0}  \| Y_{n}(t) \|_{\mc{L}^2}  &  \leq& G_{n} \label{Gn'}.
	\\
	\sup_{t >0} \{ t^{\frac12}  | \ell_{n}(t) | \}  &  \leq& G_{n} \label{Gn''}.
\end{eqnarray}
It is clear from \eqref{Lp-Lq} that \eqref{Gn}-\eqref{Gn''} is verified for $Y_{0}$, where 
\begin{equation}\label{gn}
G_{0}= \max\{K_{1}(8,2),K_{1}(2,2),K_{1}(\infty,2)\} \| V_{0}\|_{\mc{L}^2}.
\end{equation}
 In the sequel, we denote by $C_{0}$ the following positive constants:
\begin{equation*}
C_{0} :=  \max\Bigl(K_{4}(8,4), K_{4}(2,2),K_{\ell}(4)\Bigl)  \int_{0}^1(1-\tau)^{-\frac 58}\tau^{-\frac 34}\, {\rm d}\tau,
\end{equation*}
where $K_{4}$ is the constant in \eqref{est div 1} and $K_{\ell}$ in \eqref{eq_lv}.

Next, we assume that the properties are true for the rank $n$, and we show it for rank $n+1$: using \eqref{est div 1} with $p=8$, $q=4$, we get
\[
t^{\frac38}  \| Y_{n+1} (t)\|_{\mc{L}^8} \leq  G_{0} + t^{\frac38} K_{4}(8,4)  \int_{0}^t (t-s)^{-\frac{5}{8}} \| Y_{n}(s)  \|_{\mc{L}^{8}}^2\, {\rm d}s +  t^{\frac38} K_{4}(8,4)  \int_{0}^t (t-s)^{-\frac{5}{8}} | \ell_{n}(s) | \|  Y_{n}(s) \|_{\mc{L}^4}\, {\rm d}s.
\]
By interpolation, we have:
\begin{equation} \label{eq_interp}
\|Y_n\|_{\mc{L}^4} \leq  \|Y_n\|^{1/3}_{\mc{L}^2}\|Y_n\|^{2/3}_{\mc{L}^8}\,.
\end{equation}
So, we use that:
\[
\| Y_{n}(s) \|_{\mc{L}^8}^2\leq (s^{-\frac38} G_{n})^2, \qquad  \| Y_{n}(s) \|_{\mc{L}^4}\leq s^{-\frac14} G_{n}, \qquad |\ell_n(s)| \leq s^{-\frac12} G_{n}\,,
\]
to get
\begin{eqnarray*}
t^{\frac38}  \| Y_{n+1}(t) \|_{\mc{L}^{8}} 
&\leq & G_{0} + t^{\frac38} K_{4}(8,4) \int_{0}^t (t-s)^{-\frac{5}{8}} s^{-\frac34} |G_{n}|^2  \, {\rm d}s +  t^{\frac38} K_{4}(8,4) \int_{0}^t (t-s)^{-\frac{5}{8}} s^{-\frac 34} |G_{n}|^2 \, {\rm d}s\\
&\leq & G_{0} + 2 |G_n|^2  K_{4}(8,4) \int_{0}^1 (1-\tau)^{-\frac{5}{8}} \tau^{-\frac34} \, {\rm d}\tau\\[4pt]
&\leq& G_{0} + 2 C_{0} |G_{n}|^2\,.
\end{eqnarray*}
Writing the same computation and using \eqref{est div 1} with $p=q=2$ gives
$$
\| Y_{n+1}(t) \|_{\mc{L}^2} 
\leq  G_{0} +  K_{4}(2,2) \int_{0}^t (t-s)^{-\frac12} \|Y_{n}(s) \|^2_{\mc{L}^4} \, {\rm d}s +   K_{4}(2,2) \int_{0}^t (t-s)^{-\frac12} | \ell_{n} (s) | \| Y_{n}(s) \|_{\mc{L}^2}\, {\rm d}s \\
$$
Thus, we obtain 
\begin{eqnarray*}
  \| Y_{n+1}(t)  \|_{\mc{L}^{2}} 
&\leq & G_{0} + 2 K_{4}(2,2)\int_{0}^t (t-s)^{-\frac12} s^{-\frac12} |G_n|^2 \, {\rm d}s \\
&\leq & G_{0} + 2 C_{0} |G_n|^2.
\end{eqnarray*}
Finally, we apply Corollary \ref{est_lv} with $q=4,$ which yields:
\begin{eqnarray*}
t^{\frac 12} |\ell_{n+1}(t)| 
&\leq & G_{0} + t^{\frac12} K_{\ell}(4) \int_{0}^t (t-s)^{-\frac 34} \| Y_{n}(s) \|^2_{\mathcal{L}^8} \, {\rm d}s +  t^{\frac12} K_{\ell}(4)\int_{0}^t (t-s)^{-\frac 34} | \ell_{n} (s) | \| Y_{n}(s) \|_{\mc{L}^4}\, {\rm d}s \\
&\leq & G_{0} + 2 C_{0}  |G_{n}|^2 \, .
\end{eqnarray*}
Hence, we can take
$$
G_{n+1} =  G_{0} + 2C_{0}  |G_{n}|^2,\\
$$
in \eqref{Gn}--\eqref{Gn''} with $G_0$ given by \eqref{gn}.
Choosing $\lambda_{0}$ such that $G_{0} \leq 1/(8C_{0})$, we easily get by induction that for all $n \in \N$
\begin{equation}\label{gn est}
	G_{n} \leq \frac{1-(1-8C_{0}G_{0})^{\frac12}}{4C_{0}}=:\mu_{0}.
\end{equation}
Therefore, $(G_{n})$ is bounded by $\mu_{0}$ which implies that \eqref{Gn}-\eqref{Gn''} are uniform estimates. This ends the proof of Lemma \ref{lem Kato1}. According to \eqref{gn}--\eqref{gn est}, $\mu_0$ can be chosen arbitrarily small by taking $\lambda_0>0$ small enough.
\end{proof}

\medskip

{\bf Step 3: convergence of $Y_{n}$.} The goal of this step is to show that the sequence $Y_n$ constructed in the previous step strongly converges in $L^\infty(0,\infty; \mc{L}^2) \cap L^{\infty}_{loc}(0,\infty;\mathcal{L}^8),$ endowed with the norm:
$$
\|\cdot \|_{L^\infty(0, \infty; \mc{L}^2)} + \|t^{3/8}\cdot\|_{L^\infty(0, \infty; \mc{L}^8)} + \|t^{\frac 12} \ell_{\cdot}\|_{L^\infty(0, \infty)}\,, 
$$ 
and that the limit $V$ solves the integral formulation \eqref{step 1} of the Navier-Stokes equations \eqref{NS11}--\eqref{Solideci1}.

The main idea here comes from \cite{KatoFujita62}: let us define 
\begin{multline*}
W_{n+1}(t)  := Y_{n+1}(t)-Y_{n}(t)
	\\ 
	=  \int_{0}^t  S(t-s) \mathbb{P}\div 1_{\mathcal{F}_{0}} \Bigl( (\ell_{n}-Y_n)(s)\otimes (Y_{n}-Y_{n-1})(s) + (\ell_n - \ell_{n-1} + Y_{n-1}- Y_{n})(s)\otimes Y_{n-1}(s)\Bigl) \, {\rm d}s.
\end{multline*}
Again, we construct a sequence $a_n$ such that for all $n$, 
$$
a_{n} \geq  \max \left\{ \sup_{t>0} t^{\frac38}\| W_{n}(t) \|_{\mc{L}^8},  \sup_{t>0} \| W_{n}(t) \|_{\mc{L}^2}, \,  \sup_{t>0} t^{\frac 12}|\ell_{W_n}(t)|\right\}.
$$ 
Indeed, we have:
\begin{eqnarray*}
	t^{\frac38}\| W_{n+1}(t) \|_{\mc{L}^8}&\leq& K_{4}(8,4) t^{\frac38} \Bigl(\int_{0}^t (t-s)^{-\frac58} (\|Y_{n}(s)\|_{\mc{L}^8}+\|Y_{n-1}(s)\|_{\mc{L}^8}) \|W_{n}(s)\|_{\mc{L}^8} \, {\rm d}s\\
&&+  \int_{0}^t (t-s)^{-\frac 58} (|\ell_{n}(s) | \|W_{n}(s)\|_{\mc{L}^4} + |\ell_{n}(s) -\ell_{n-1}(s)| \|Y_{n-1}(s)\|_{\mc{L}^4})\, {\rm d}s\Bigr)\\
&\leq& 4 K_{4}(8,4) \int_{0}^t (t-s)^{-\frac58} s^{-\frac34} \mu_0 a_{n}\, {\rm d}s \\
& \leq & 4 C_{0}\mu_0a_{n} .
\end{eqnarray*}
Here and in what follows, we always estimate $L^4$-norms by interpolating the $L^2$-norm and $L^8$-norm (see \eqref{eq_interp}).
In the same manner, we have
\begin{eqnarray}
	\label{Est-wn-L2}
	\| W_{n+1}(t) \|_{\mc{L}^2}&\leq& K_{4}(2,2) \int_{0}^t (t-s)^{-\frac12} (\|Y_{n}(s)\|_{\mc{L}^4}+\|Y_{n-1}(s)\|_{\mc{L}^4}) \|W_{n}(s) \|_{\mc{L}^4}\, {\rm d}s\\
&&+ K_{4}(2,2) \int_{0}^t (t-s)^{-\frac12} (|\ell_{n}(s) | \|W_{n}(s)\|_{\mc{L}^2} + |\ell_{n} (s) -\ell_{n-1}(s)| \|Y_{n-1}(s)\|_{\mc{L}^2})\, {\rm d}s\,,
	\notag
\end{eqnarray} 
which implies:
\begin{eqnarray}	
\| W_{n+1}(t) \|_{\mc{L}^2} &\leq& 4 K_{4}(2,2) \int_{0}^t (t-s)^{-\frac12} s^{-\frac12}\mu_0a_{n}\, {\rm d}s\notag \\
& \leq & 4 C_{0}\mu_0  a_{n}\,.
\end{eqnarray}
Finally, we have:
\begin{eqnarray}
	\label{Est-wn-L3}
t^{\frac 12}| \ell_{n+1}(t) - \ell_n(t) |&\leq& K_{\ell}(4) t^{\frac 12}   \Bigl(
\int_{0}^t (t-s)^{-\frac 34} (\|Y_{n}(s)\|_{\mc{L}^8}+\|Y_{n-1}(s)\|_{\mc{L}^8}) \|W_{n}(s) \|_{\mc{L}^8}\, {\rm d}s\\
&+&  \int_{0}^t (t-s)^{-\frac 34} (|\ell_{n}(s) | \|W_{n}(s)\|_{\mc{L}^4} + |\ell_{n} (s) -\ell_{n-1}(s)| \|Y_{n-1}(s)\|_{\mc{L}^4})\, {\rm d}s\Bigr)
	\notag
	\\
&\leq& 4 K_{\ell}(4) t^{\frac 12} \int_{0}^t (t-s)^{-\frac 34} s^{-\frac 34}\mu_0 a_{n}\, {\rm d}s\notag \\
&\leq & 4 C_{0}\mu_0  a_{n} \notag.
\end{eqnarray}
Therefore, we can take $a_{n}=(4C_{0}\mu_0)^{n-1}a_{1}$ where $a_{1}$ can be easily estimated thanks to Lemma \ref{lem Kato1}. According to Lemma \ref{lem Kato1} again, one can choose $\lambda_0>0$ such that $ \mu_0 <1/(4C_0)$. With this choice, $\sum (\|W_{n}\|_{L^\infty(0, \infty; \mc{L}^2)} + \|t^{1/8}W_{n}\|_{L^\infty(0, \infty; \mc{L}^8)} + \|t^{1/2}\ell_{W_n}(t)\|_{L^{\infty}(0,\infty)})$ converges uniformly and there exists a function $V \in L^\infty(0, \infty; \mc{L}^2) \cap L^{\infty}_{loc}(0,\infty;\mc{L}^8)$ such that 
\[
	Y_{n} \to  V \text{ strongly in } L^\infty(0,\infty; \mc{L}^2) \cap L^{\infty}_{loc}(0,\infty;\mc{L}^8)\,, \quad \ell_{n} \to \ell_V \text{ in }L^{\infty}_{loc}(0,\infty)\,.
\]
By construction $ V$ satisfies the decay estimates of Lemma \ref{lem Kato1}:
 \begin{equation}
 	\label{Lp-Lq-V-NavierStokes-0}
	\sup_{t >0}\{ t^{\frac38} \|V(t)\|_{\mc{L}^8}\} \leq \mu_0,\quad \ \sup_{t>0} \{ \|V(t)\|_{\mc{L}^2}\} \leq \mu_0 ,\quad  \sup_{t>0} \{ t^{\frac12} |\ell_{V}(t)|\} \leq \mu_0.
 \end{equation}

The last point to check is that $ V$ indeed is a solution of the integral equation \eqref{step 1}, \emph{i.e.}  we have to check that $\mathcal{K}Y_{n} \to \mathcal{K} V$. This computation is exactly the previous one:
\[
	\mathcal{K} V(t)-\mathcal{K} Y_{n}(t) =  \int_{0}^t S(t-s) \mathbb{P}\div 1_{\mathcal{F}_{0}} \Bigl( (\ell_V -  V )\otimes ( V-Y_{n}) + (\ell_V - \ell_n + Y_n -V )\otimes Y_{n}\Bigl)(s) \, {\rm d}s.
\]
Doing as in \eqref{Est-wn-L2} and using that $\sup_{t>0} t^{\frac38}\|  V-Y_{n}(t) \|_{\mc{L}^8},$ $ \sup_{t>0} \| V(t)-Y_{n}(t) \|_{\mc{L}^2}$ and  $\sup_{t>0} \{ t^{\frac12} |\ell_{V}(t) - \ell_n(t)|\}$ tend to zero as $n\to \infty$, one easily shows that $\mathcal{K} V-\mathcal{K}Y_{n}$ converges to $0$ in $L^\infty(0,\infty; \mc{L}^2)$.
This shows that the limit $V$ of the sequence $Y_n$ solves the integral formulation \eqref{step 1} of the Navier-Stokes equations \eqref{NS11}--\eqref{Solideci1}.

\bigskip

{\bf Step 4: The limit $V$ is the unique weak solution of \eqref{NS11}--\eqref{Solideci1} when $V_0 \in H^1(\R^2)$.} In the previous steps, we have constructed a solution to the integral formulation \eqref{step 1} of the Navier-Stokes equations \eqref{NS11}--\eqref{Solideci1}  verifying the $L^p-L^q$ decay estimates \eqref{Lp-Lq-V-NavierStokes-0}. The last point that we have to check is  that this solution $ V$ is the unique solution from the well-posedness theory of \cite{Takahashi&Tucsnak04}.
In  \cite{Takahashi&Tucsnak04}, uniqueness is obtained in the framework $ V\in L^\infty(0,T;\mc{L}^2)\cap L^2(0,T;H^1(\R^2))$. 

Of course, our solution satisfies by construction $ V\in L^\infty(0,T;\mc{L}^2)$, and thus we only have to check that $V \in L^2(0,T; H^1(\R^2))$.

We focus on the case of initial data $V_0 \in H^1(\R^2)$ (\emph{i.e.}, $V_0 \in \mathcal{D}(A^{1/2}),$ see \eqref{domainStokes}). In that case, we prove that the solution $V$ constructed above is the unique solution in $L^\infty(0,T;\mc{L}^2)\cap L^2(0,T;H^1(\R^2))$. The main issue is to show that the sequence $\|\nabla y_{n}\|_{L^\infty(0,T; {L}^2(\mathcal F_0))}$ is uniformly bounded in $n$ for any arbitrary $T>0$ fixed. For $T$ fixed, it is proved in \cite[Cor. 4.3]{Takahashi&Tucsnak04} that $S(t) V_{0}$ belongs to $\mathcal{C}([0,T] ; H^1(\mathbb R^2))$ when $ V_{0}\in H^1(\R^2)$, which implies that there exists $J_{0}>0$ such that
\[
	  \| \nabla y_{0} \|_{L^\infty(0,T; {L}^2(\mathcal F_0))}
	   \leq J_{0}.
\]
Next, we construct by induction a sequence $J_{n}$ such that for all $n\in \N$
\[
	  \| \nabla y_{n} \|_{L^\infty(0,T; {L}^2(\mathcal F_0))} \leq J_{n}.
\]
Using \eqref{est grad 1} with $p=2,\, q=8/5$ and $p=q=2$, for all $t \in [0,T]$
\begin{eqnarray*}
	 \| \nabla y_{n+1} (t) \|_{L^2(\mathcal F_0)} 
	&\leq & J_{0} +  C_{8/5} K_{2}(2,8/5) \int_{0}^t (t-s)^{-\frac12+\frac12-\frac58} \| |y_{n}(s)| |\nabla y_{n}(s)| \|_{L^{8/5}(\mathcal F_0)}\, {\rm d}s
		\\
	& &+   K_{2}(2,2) \int_{0}^t (t-s)^{-\frac12} | \ell_{n}(s) | \| \nabla y_{n}(s) \|_{L^2(\mathcal F_0)}\, {\rm d}s 
		\\
	&\leq & J_{0} + C_{8/5}K_{2}(2,8/5) \int_{0}^t (t-s)^{-\frac58} \| Y_{n}(s) \|_{\mathcal{L}^{8}} \| \nabla y_{n}(s) \|_{L^{2}(\mathcal F_0)} \, {\rm d}s
		\\
	& & +   K_{2}(2,2) \int_{0}^t (t-s)^{-\frac12} | \ell_{n}(s) | \| \nabla y_{n}(s) \|_{L^2(\mathcal F_0)} \, {\rm d}s
		\\
	&\leq & J_{0} + C_{8/5} K_{2}(2,8/5)\int_{0}^t (t-s)^{-\frac58} s^{-\frac38} \mu_0 J_{n}  \, {\rm d}s 
	\\
	& &
	+ K_{2}(2,2) \int_{0}^t (t-s)^{-\frac12} s^{-\frac12} \mu_0 J_{n} \, {\rm d}s\\
	&\leq & J_{0} + \tilde C_0 \mu_0  J_{n}
	 =: J_{n+1},
\end{eqnarray*}
where $C_{8/5} := \|\mathbb P_{8/5}\|_{\mathscr{L}_c(L^{8/5}(\mathbb R^2) \to \mathcal L^{8/5})}$ and $\tilde C_0:=(C_{8/5}K_{2}(2,8/5) + K_{2}(2,2))\int_{0}^1 (1-\tau)^{-\frac58} \tau^{-\frac12}  \, {\rm d}\tau$. Taking $\lambda_0>0$ small enough so that $ \tilde C_0 \mu_0  \leq 1/2$,
there holds:
\[
J_{n} = J_{0} \sum_{k=0}^n \Bigl(\tilde C_0 \mu_0 \Bigl)^k \leq J_{0}\frac1{1-\tilde C_{0}\mu_0} \leq 2 J_0.
\]
Hence we have, for all $n \in \N$, 
\begin{equation*}
	  \| \nabla y_{n} \|_{L^\infty(0,T; {L}^2(\mathcal F_0))}  \leq 2J_{0},
\end{equation*}
which implies that $\nabla  v$ verifies the same estimate. Inside $B(0,1)$, we have
\begin{equation*}
	 |\nabla Y_n|=| \omega_{Y_n} | = \Bigl| \int_{B(0,1)} Y_n \cdot x^\perp \, {\rm d}x\Bigl|\leq C\| Y_n \|_{\mc{L}^2}
\end{equation*}
which is uniformly bounded in time and $n$. 

Note that this is not enough to conclude that $Y_n \in L^\infty([0,T]; H^1(\R^2))$ for all $n$ and one should be careful that the boundary conditions are compatible on $\partial B_0$. In order to do that, for all $n \in \N$, we introduce, for all $\epsilon \in (0,1)$ and $t>0$,
\[
	Y_{n+1}^\varepsilon(t) = S(t) V_0 + \int_0^{t(1- \varepsilon)} S(t-s) \mathbb{P} F(Y_{n}(s)) \, {\rm d}s.
\]
Of course, arguing as above, $Y_{n+1}^\varepsilon$ satisfy exactly the same estimate as $Y_{n+1}$, uniformly with respect to $\varepsilon >0$ and $n\in \N$:
\begin{equation}
	\label{Bound-y-n-eps}
	   \| \nabla Y_{n+1}^\varepsilon \|_{L^\infty(0,T; {L}^2(\R^2))} = \frac{\pi}{2} \| \omega_{Y_{n+1}^\varepsilon} \|_{L^\infty(0,T)}  + \| \nabla y_{n+1}^\varepsilon \|_{L^\infty(0,T; L^2(B_0))} \leq C.
\end{equation}
But, since the semigroup $S(t)$ is analytic on $\mc{L}^2$, for $t>0$, $S(t) V_0 \in \mathcal{D}(A)$ (see \eqref{domainStokes}) and 
\begin{eqnarray*}
	\int_0^{t(1- \varepsilon)} S(t-s) \mathbb{P} F(Y_{n}(s)) \, {\rm d}s 
	& = & S(t \varepsilon) \int_0^{t(1- \varepsilon)} S(t(1- \varepsilon)-s) \mathbb{P} F(Y_{n}(s)) \, {\rm d}s 
	\\
	& = & S(t \varepsilon) \left(Y_{n+1}(t(1- \varepsilon)) - S(t(1- \varepsilon))V_0  \right)
	\\
	& =& S(t \varepsilon) Y_{n+1}(t(1- \varepsilon)) - S(t) V_0.
\end{eqnarray*}
Since for all $t >0$, $Y_{n+1}(t) \in \mc{L}^2$, this implies that for all $t>0$, $Y_{n+1}^\varepsilon (t)$ belongs to $\mathcal{D}(A)$ for all $t>0$. 

 Besides, as $\varepsilon \to 0$, $Y_{n+1}^\varepsilon$ converges to $Y_{n+1}$ in $L^\infty_{loc} ([0, \infty); \mc{L}^2)$ since
\begin{eqnarray*}
	\lefteqn{\|Y_{n+1}^\varepsilon(t) - Y_{n+1}(t)\|_{\mc{L}^2} }
	\\
	& \leq & \int_{t(1- \varepsilon)}^t (K_1(2,4/3) (t-s)^{-1/4} \|y_n(s) \|_{\mc{L}^4} \|\nabla y_n(s)\|_{L^2(\mathcal{F}_0)}  + K_1(2,2) |\ell_n(s)| \nabla y_n(s)\|_{L^2(\mathcal{F}_0)} ) \, {\rm d}s
	\\
	& \leq  & t^{1/2} \left( K_1(2, 4/3) \mu_0 J_0  \int_{1-\varepsilon}^1  (1-\tau)^{-1/4} \tau^{-1/4} \, {\rm d}\tau + K_1(2,2) \mu_0 J_0 \int_{1-\varepsilon}^1 \tau^{-1/2} \, {\rm d}\tau\right).
\end{eqnarray*}
Hence $Y_{n+1}$ is the strong limit in $L^\infty((0,T); \mc{L}^2)$ of the sequence of functions $Y_{n+1}^\varepsilon$ satisfying \eqref{Bound-y-n-eps} and the fact that for all $\varepsilon>0$ and $t >0$, $Y_{n+1}^\varepsilon(t) \in \mathcal{D}(A)$. Therefore, $Y_{n+1}$ belongs to $L^2([0,T]; H^1(\R^2))$.

Besides, since the bound in \eqref{Bound-y-n-eps} is uniform in $\epsilon >0$ and $n \in \N$, $V$ also belongs to $L^2((0,T);H^1(\R^2))$.
According to \cite{Takahashi&Tucsnak04}, when the initial data $V_0$ belongs to $H^1(\R^2)$, the solution $V$ constructed in the above steps is the unique weak solution of \eqref{NS11}--\eqref{Solideci1}.

\medskip

{\bf Step 5: Sensitivity of $V$ to the initial data.} So far, given $V_0 \in  \mc{L}^2$ satisfying the smallness condition \eqref{Smallness-v0}, we have constructed a solution $V$ of the integral equation \eqref{step 1}. In this step, we show that the map $V_0 \mapsto V$ is continuous from the ball of $\mc{L}^2$ with radius $\lambda_0$  to $L^\infty((0, \infty); \mc{L}^2)$. 

Let us consider $V_0^a$ and $V_0^b$ two elements of  $\mc{L}^2$ satisfying the smallness condition \eqref{Smallness-v0}, and $Y_n^a$ and $Y_n^b$ the corresponding sequences in \eqref{KatoIteration}. We set $Z_n = Y_n^a - Y_n^b$, which satisfies by construction
$$
	Z_{n+1}(t) = S(t) (V_0^a - V_0^b) + \mc{K} Y_n^a(t) - \mc{K} Y_n^b(t)= Z_0(t) + \mc{K} Y_n^a(t) - \mc{K} Y_n^b(t).
$$
Similarly as in Step 3, we are going to construct a sequence $b_n$ such that for all $n$,
$$
	b_n \geq \max\left\{\sup_{t >0}\{  t^{3/8} \|Z_n(t)\|_{\mc{L}^{8}} \}\, , \  \sup_{t >0}\{ \| Z_n(t) \|_{\mc{L}^2}\}\,,\ \sup_{t >0} \{|t^{\frac 12} \ell_{Z_n}(t) |\} \right\}.
$$
Of course, by Theorem \ref{theo Stokes}, one can take $b_0$ proportional to $\| V_0^a - V_0^b \|_{\mc{L}^2}$. Since
\begin{align*}
	\mathcal{K} Y_n^a(t)-\mathcal{K} Y_{n}^b(t) 
	& = 
	 \int_{0}^t S(t-s) \mathbb{P}\div 1_{\mathcal{F}_{0}} \left( (\ell_n^a - Y_n^a)(s)\otimes Y_n^a(s) -  (\ell_n^b - Y_n^b)(s) \otimes Y_n^b(s)   \right) \, {\rm d}s
	\\
	& =
	 \int_{0}^t S(t-s) \mathbb{P}\div 1_{\mathcal{F}_{0}} \left( 
	 (Y_n^b - Y_n^a)(s) \otimes Y_n^b(s) + Y_n^a(s) \otimes (Y_n^b (s) - Y_n^a(s))
	 \right.
	 \\
	 & \qquad \qquad \qquad +\left.  
	 (\ell_n^a - \ell_n^b)(s)\otimes Y_n^a(s) + \ell_n^b(s) \otimes (Y_n^a - Y_n^b)(s)  
	 \right) \, {\rm d}s,
\end{align*}
arguing as in Step 3,
\begin{eqnarray*}
	t^{\frac38}\| \mathcal{K} Y_n^a(t)-\mathcal{K} Y_{n}^b(t) \|_{\mc{L}^8}
	& \leq & 
	K_{4}(8,4) t^{\frac38} \Bigl( \int_{0}^t (t-s)^{-\frac58} (\| Y_{n}^a(s)\|_{\mc{L}^8}+\|Y_{n}^b(s)\|_{\mc{L}^8}) \|Z_{n}(s)\|_{\mc{L}^8} \, {\rm d}s
	\\
	&&
	+ \int_{0}^t (t-s)^{-\frac58} (|\ell_{n}^b(s) | \|Z_{n}(s)\|_{\mc{L}^4} + |\ell_{n}^a(s) -\ell_{n}^b(s)| \|Y_{n}^a(s)\|_{\mc{L}^4})\, {\rm d}s \Bigr)
	\\
	&\leq& 
	4 K_{4}(8,4) t^{\frac38}\int_{0}^t (t-s)^{-\frac58} s^{-\frac34}\mu_0  b_{n}\, {\rm d}s \\
	&\leq & 4C_{0}\mu_{0}  b_{n}\, ,
\end{eqnarray*}
where $\mu_{0}$ is the constant in Lemma \ref{lem Kato1}. And similarly as in \eqref{Est-wn-L2},
\begin{eqnarray*}
	\| \mathcal{K} Y_n^a(t)-\mathcal{K}Y_{n}^b(t) \|_{\mc{L}^2}
&\leq& K_{4}(2,2) \int_{0}^t (t-s)^{-\frac12} (\|Y^a_{n}(s)\|_{\mc{L}^4}+\|Y^b_{n}(s)\|_{\mc{L}^4}) \|Z_{n}(s) \|_{\mc{L}^4}\, {\rm d}s\\
&&+ K_{4}(2,2) \int_{0}^t (t-s)^{-\frac12} (|\ell^b_{n}(s) | \|Z_{n}(s)\|_{\mc{L}^2} + |\ell^a_{n} (s) -\ell^b_{n}(s)| \|Y^a_{n}(s)\|_{\mc{L}^2})\, {\rm d}s\,,
	\notag\\
	&\leq& 4 K_{4}(2,2) \int_{0}^t (t-s)^{-\frac12} s^{-\frac12}\mu_0 b_{n}\, {\rm d}s\notag \\
	&\leq &  4C_0 \mu_0 b_n.
\end{eqnarray*} 
Finally, we prove as in \eqref{Est-wn-L3} that:
\begin{eqnarray*}
t^{\frac 12}| \ell^a_{n}(t) - \ell^b_n(t) |&\leq& K_{\ell}(4) t^{\frac 12}   \Bigl(
\int_{0}^t (t-s)^{-\frac 34} (\|Y^a_{n}(s)\|_{\mc{L}^8}+\|Y^b_{n}(s)\|_{\mc{L}^8}) \|Z_{n}(s) \|_{\mc{L}^8}\, {\rm d}s\\
&+&  \int_{0}^t (t-s)^{-\frac 34} (|\ell^b_{n}(s) | \|Z_{n}(s)\|_{\mc{L}^4} + |\ell^a_{n} (s) -\ell^b_{n}(s)| \|Y^a_{n}(s)\|_{\mc{L}^4})\, {\rm d}s\Bigr)
	\notag
	\\
&\leq& 4 K_{\ell}(4) t^{\frac 12} \int_{0}^t (t-s)^{-\frac 34} s^{-\frac 34}\mu_0 b_{n}\, {\rm d}s\notag \\
&\leq & 4 C_{0}\mu_0 b_{n} \notag.
\end{eqnarray*}
We can then chose $b_{n+1} = b_0 + 4C_0 \mu_0 b_n$ and thus (recall that $4 C_0 \mu_0 <1$ for our choice of $\lambda_0$ since Step 3)
$$
	\forall n \in \N, \quad b_n \leq b_0 \left( \frac{1}{ 1 - 4 C_0 \mu_0 } \right).
$$
In particular, passing to the limit $n \to \infty$, we obtain
$$
	\sup_{t >0} \| V^a(t) - V^b(t)\|_{\mc{L}^2} \leq  C \| V_0^a - V_0^b \|_{\mc{L}^2}.
$$

Thus, our above construction yields a map $V_0 \mapsto V$ continuous on the ball of $\mc{L}^2$ of radius $\lambda_0$ to $L^\infty((0,\infty); \mc{L}^2)$, which coincides with the map $V_0 \mapsto V_w$ for initial data in $H^1(\R^2)$, where $V_w$ denotes the weak solution of \eqref{NS11}--\eqref{Solideci1}. Since both maps are continuous (see ``The existence part" in the proof of Proposition 2.5 in \cite[Section 6]{Takahashi&Tucsnak04} for the continuity of the map $V_0 \mapsto V_w$ from $\mc{L}^2$ to $L^\infty((0,\infty); \mc{L}^2)$), they coincide on the ball of $\mc{L}^2$ of radius $\lambda_0$.
This implies that the solution $V$ constructed in Step 3, as the limit of the sequence $Y_n,$ actually is the unique weak global solution of \eqref{NS11}--\eqref{Solideci1}.

\bigskip

{\bf Step 6: Estimates on the $\mc{L}^p$ norm of $V$ for  $2 \leq p<\infty$.}
The goal of this step is to show the following Lemma:
\begin{lemma}\label{lem Kato2}
Let $\lambda_{0}$ the constant of Lemma \ref{lem Kato1}. For all $p \in [2, \infty),$  there exists a constant $H(p, \lambda_0)$ such that for all $V_0\in \mc{L}^2$ satisfying \eqref{Smallness-v0}, the solution $V$ of \eqref{NS11}--\eqref{Solideci1} satisfies:
\begin{equation}\label{step 4}
	\sup_{t >0}\{ t^{\frac12 - \frac1p} \|V(t)\|_{\mc{L}^p}\} \leq H(p,\lambda_0).
 \end{equation}
\end{lemma}

\begin{proof}
For  $p \leq 8$ we obtain \eqref{step 4} by interpolation of the estimates of {Lemma} \ref{lem Kato1}.
Assume now $p \in [ 8, \infty)$:
\begin{eqnarray*}
t^{\frac12-\frac{1}{p}}  \| V(t) \|_{\mc{L}^{p}} 
&\leq & K_{1}(p,2) \|  V_{0}\|_{\mc{L}^2} 
+ t^{\frac12-\frac1p} K_{4}(p,4) \int_{0}^t (t-s)^{-\frac{3}{4}+\frac{1}{p}} \| V(s) \|_{\mc{L}^8}^2\, {\rm d}s \\
&&+  t^{\frac1{2}-\frac1p} K_{4}(p,4) \int_{0}^t (t-s)^{-\frac{3}{4}+\frac1p} | \ell_V (s)| \|V(s) \|_{\mc{L}^4} \, {\rm d}s \\
&\leq & K_{1}(p,2) \|  V_{0}\|_{\mc{L}^2} 
+ t^{\frac12-\frac{1}{p}} K_{4}(p,4) \int_{0}^t (t-s)^{-\frac{3}{4}+\frac{1}{p}} s^{-\frac{3}{4}} \left( \mu_0^2 + \mu_0^2 \right)     \, {\rm d}s \\
&\leq & 
K_{1}(p,2) \lambda_0  
+ 2 K_{4}(p,4) \mu_0^2  \int_{0}^1 (1-\tau)^{-\frac{3}{4}+\frac{1}{p}} \tau^{-\frac3{4}} \, {\rm d}\tau ,
\end{eqnarray*}
which gives the desired estimates \eqref{step 4} and concludes the proof of Lemma \ref{lem Kato2}.
\end{proof}

The proof of Proposition \ref{prop kato} is then completed.
\end{proof}

\begin{remark}[Remark on the smallness condition]\label{rem small condition}
The smallness condition on $\|  V_{0}\|_{\mc{L}^2}$ is not surprising, and such an assumption appears in a lot of articles when global well-posedness is required (see e.g. \cite{Kato84}). In dimension 2, several works (\cite{Schonbek,Wiegner,KajikiyaMiyakawa} in the full plane and \cite{BorchersMiyakawa} in fixed exterior domains) show that the $L^2$-norm tends to zero when $t\to 0$ for initial data in $L^2$. Of course, this allows in such situations to get a global result for any initial data in $L^2$ by proving only a local result for initial data having small $L^2$-norm. 
Unfortunately, concerning the case of a moving disk in a 2D viscous fluid, despite the energy estimate satisfied by the solutions of \eqref{NS1}--\eqref{Solideci} which immediately guarantees the global decay of the $L^2$-norm of the solution, it is still not clear that the $L^2$-norm of all solutions with initial data in $L^2$ go to zero as $t \to \infty$. This appears to be a challenging question.
\end{remark}

\subsection{The case of an initial data $\mc{L}^q$ for $q \in (1,2)$}

The goal of this section is to prove Theorem \ref{Thm-q-dans(1,2)}.

\begin{proof}[Proof of Theorem \ref{Thm-q-dans(1,2)}]
	The proof is based on the construction done in Proposition \ref{prop kato}. 
	
	{\bf Step 1. Decay estimates for $p= 2$ and $p = 4$.} Let us consider again the sequence $Y_n$ constructed in \eqref{KatoIteration}, for which we already know the decay estimates 
	of Lemma \ref{lem Kato1} and, by interpolation,
	\begin{equation}
		\label{DecayL4-L2}
		\sup_{t >0} \{ t^{1/4} \| Y_n(t) \|_{\mc{L}^4} \} \leq \mu_0.
	\end{equation}
	We then prove the following lemma:
	
	\begin{lemma}\label{Lem-Kato3}
		There exists $\lambda_0(q)$ small enough such that for any $V_0 \in \mc{L}^q\cap \mc{L}^2$ satisfying the smallness condition \eqref{Smallness-v0} with $\lambda_0 \leq \lambda_0(q)$, there exist constants $H(2,q,V_0), H(4,q,V_0)$ for which the sequence $Y_n$ constructed in \eqref{KatoIteration} satisfies
		\begin{equation}
			\sup_{t >0} \{ t^{1/q - 1/2} \| Y_n(t) \|_{\mc{L}^2} \} \leq H(2,q,V_0),
			\qquad
			\sup_{t >0} \{ t^{1/q - 1/4} \| Y_n(t) \|_{\mc{L}^4} \} \leq  H(4,q, V_0).
		\end{equation}
		Consequently, we have
		\begin{equation}
			\label{Decay-v-2-4}
			\sup_{t >0} \{ t^{1/q - 1/2} \|V(t) \|_{\mc{L}^2} \} \leq H(2,q,V_0),
			\qquad
			\sup_{t >0} \{ t^{1/q - 1/4} \| V(t) \|_{\mc{L}^4} \} \leq  H(4,q, V_0).
		\end{equation}
	\end{lemma}
	
	\begin{proof}
		We are looking for a sequence $H_n$ such that for all $n$, 
		$$
			\sup_{t >0} \{ t^{1/q - 1/2} \| Y_n(t) \|_{\mc{L}^2} \} \leq H_n, \qquad  \sup_{t >0} \{ t^{1/q - 1/4} \| Y_n(t) \|_{\mc{L}^4} \} \leq  H_n.
		$$
		Of course, Theorem \ref{theo Stokes} implies that $H_0$ can be taken as $H_0 = (K_1(2,q) + K_1(4,q)) \|V_0\|_{\mathcal L^q}$.
		
		For $n \in \mathbb{N}$, using Theorem \ref{theo Stokes} and Corollary \ref{cor Stokes},
		$$
			t^{1/q-1/2} \|Y_{n+1}(t) \|_{\mc{L}^2} \leq K_1(2,q) \| V_0\|_{\mc{L}^q} + t^{1/q- 1/2}  \int_0^t K_4(2,2) (t-s)^{-1/2} \left(\| Y_n(s) \|_{\mc{L}^4}^2 + |\ell_n(s)| \| Y_n(s)\|_{\mc{L}^2} \right)  \, {\rm d}s.
		$$
		Using the decay estimates of Lemma \ref{lem Kato1} and \eqref{DecayL4-L2}, 
		\begin{eqnarray*}
			\| Y_n(s) \|_{\mc{L}^4}^2 
			& \leq &
			(\mu_0 s^{-1/4}) (s^{1/4-1/q} H_n) = \mu_0 H_n s^{-1/q},
		\\
			 |\ell_n(s) | \| Y_n(s)\|_{\mc{L}^2} 
			& \leq &
			( \mu_0 s^{-1/2}) (s^{1/2- 1/q} H_n) = \mu_0 H_n s^{-1/q},
		\end{eqnarray*}
		we immediately deduce that
		$$
			\sup_{t>0} \{ t^{1/q-1/2} \|Y_{n+1}(t) \|_{\mc{L}^2}\} 
			\leq 
			K_1(2,q) \|V_0\|_{\mc L^q} + 2 K_4(2,2) \left(\int_0^1 (1-\tau)^{-1/2} \tau^{-1/q}\, {\rm d}\tau\right) \mu_0 H_n.
		$$
		Similar computations yield
		$$
			\sup_{t>0} \{ t^{1/q-1/4} \|Y_{n+1}(t) \|_{\mc{L}^4}\} 
			\leq 
			K_1(4,q) \|V_0\|_{\mc L^q} + 2 K_4(4,2) \left(\int_0^1 (1-\tau)^{-3/4} \tau^{-1/q} \, {\rm d}\tau\right)\mu_0 H_n.
		$$
		One can thus take
		$$
			H_{n+1} := H_0 + \widehat C_0(q) \mu_0 H_n,
		$$
		with
		$$
		\widehat C_0(q):=2 \left(K_4(2,2) \int_0^1 (1-\tau)^{-1/2} \tau^{-1/q} \, {\rm d}\tau+ K_4(4,2)\int_0^1 (1-\tau)^{-3/4} \tau^{-1/q} \, {\rm d}\tau\right)
		$$
		
		Choosing $\lambda_0(q)>0$ such that the corresponding $\mu_0$ given by Lemma \ref{lem Kato1} satisfies
		\begin{equation}
		\label{Condition-on-mu}
		\widehat C_0(q) \mu_0 \leq 1/2,
		\end{equation}
		we thus immediately obtain that for all $n$, $H_{n+1} \leq H_0 + H_n/2$, yielding $H_n \leq 2H_0$ for all $n\in \mathbb{N}$.
	\end{proof}

	{\bf Step 2. The solution $V$ satisfies the decay estimates \eqref{H(p,q)}.} The proof of this result follows the proof of Lemma \ref{lem Kato2}. For $p \in [2,4]$, \eqref{H(p,q)} can be deduced by interpolation with \eqref{Decay-v-2-4}. For $p \in [4, \infty)$, we write
	\begin{eqnarray*}
		t^{\frac1q-\frac{1}{p}}  \| V(t) \|_{\mc{L}^{p}} 
		&\leq & K_{1}(p,q) \|  V_{0}\|_{\mc{L}^q} 
		+ t^{\frac1q-\frac1p} K_{4}(p,2) \int_{0}^t (t-s)^{-1+\frac{1}{p}} \| V(s) \|_{\mc{L}^4}^2\, {\rm d}s 
		\\
		&&
		+  t^{\frac1{q}-\frac1p} K_{4}(p,2) \int_{0}^t (t-s)^{-1+\frac1p} | \ell_V (s)| \|V(s) \|_{\mc{L}^2} \, {\rm d}s 
		\\
		&\leq & K_{1}(p,q) \|  V_{0}\|_{\mc{L}^q} 	
		+ t^{\frac1q-\frac{1}{p}} K_{4}(p,2) \int_{0}^t (t-s)^{-1+\frac{1}{p}} s^{-\frac{1}{q}}   \, {\rm d}s  \left( H(4,q,V_0) + H(2,q,V_0) \right)\mu_0 
		\\
		&\leq & 
		K_{1}(p,q) \|  V_{0}\|_{\mc{L}^q}  
		+  K_{4}(p,2) \left( H(4,q,V_0) + H(2,q,V_0) \right)  \mu_0  \int_{0}^1 (1-\tau)^{-1+\frac{1}{p}} \tau^{-\frac1{q}} \, {\rm d}\tau.
	\end{eqnarray*}
	
	{\bf Step 3. The decay estimate on $\ell_{V}(t)$.} The proof of \eqref{C-ell-q} is very similar to the above one and is based on Corollary \ref{est_lv}: for $p >2$ such that $1/p  - 1/q > -1/2$,
		\begin{eqnarray*}
			t^{1/q} |\ell_V(t)|
			& \leq & 
			K_1(\infty, q) \| V_0\|_{\mc{L}^q} + t^{1/q} \int_0^t K_\ell(p) (t-s)^{-1/2-1/p} \left(\|V(s) \|_{\mc{L}^{2p}}^2 + |\ell_V(s)| \| V(s)\|_{\mc{L}^{p}} \right)  \, {\rm d}s
			\\
			& \leq & 
			K_1(\infty, q) \| V_0\|_{\mc{L}^q} \\
			&&+ K_\ell(p) (H(2p,q, V_0)H(2p,2,V_0)+ H(p,q,V_0)  \mu_0) \int_0^1 (1-\tau)^{-1/2-1/p} \tau^{1/p-1/2 -1/q} \, {\rm d}\tau. 
		\end{eqnarray*}

	{\bf Step 4. On the map $q \mapsto \lambda_0(q)$.} We remark that, by construction $q \mapsto \lambda_0(q)$ is an increasing function. Indeed, condition \eqref{Condition-on-mu} indicates that our proof of Theorem \ref{Thm-q-dans(1,2)} requires the result of Lemma \ref{lem Kato1} with $\mu_0 = \mu_0(q)>0$, where $\mu_0(q)$ is an increasing function of $q \in (1,2]$. Since the explicit formula \eqref{gn} and \eqref{gn est} indicates that $\lambda_0 \mapsto \mu_0$ is a continuous increasing function, the map $q \mapsto \lambda_0(q)$ is an increasing function of $q \in (1,2]$. We also note here that $\lambda_0(q)\to 0$ when $q\to 1$, since $\widehat C_0(q) \to \infty$. This concludes the proof of Theorem \ref{Thm-q-dans(1,2)}.
\end{proof}

\subsection{Proximity with the linearized semi-group}

In this last subsection, we compare the asymptotic structure of solutions to the Navier Stokes and Stokes
equations and prove Theorem \ref{theo kato}.

\begin{proof}[Proof of Theorem \ref{theo kato}]
Let $V_0$ satisfy the assumptions of our proposition. 

As $\tilde q \mapsto \lambda_0(\tilde q)$ is an increasing function, we note that $\|V_0\|_{\mathcal L^2} \leq \lambda_0(\tilde q)$
for all $\tilde q \in [q,2]$ so that $V(t)$ satisfies the decay estimates of \eqref{H(p,q)}-\eqref{C-ell-q} for arbitrary $\tilde q \in [q,2]$ and $p\in[2,\infty).$

\medskip

According to estimate \eqref{est div 1} with $p\in [2, \infty)$ and $q = 2$, for all $t>0$,
 \begin{eqnarray}
\|V(t) - S(t)V_0\|_{\mathcal L^p} &\leq&  \int_0^{t} \|S(t-s) \mathbb P \div ( (V(s) - \ell_V(s)) \otimes V(s) ) \|_{\mc{L}^p} \, {\rm d}s \notag\\ 
& \leq & K_4(p,2) \int_0^t (t-s)^{-1+1/p} \left( \|V(s)\|_{\mathcal L^4}^2 + |\ell_V(s)|\|V(s)\|_{\mc{L}^2}\right)   \, {\rm d}s\,. \label{Est-Diff-Carpio1a}
\end{eqnarray}

But, using \eqref{H(p,q)} with $p = 4$ and $\tilde q \in [ q, 2]$ and with $p = 2$ and $\tilde q$, we get:
$$
	\sup_{s >0} \{ s^{1/\tilde q - 1/4} \|V(s) \|_{\mc{L}^{4}} \} \leq H(4, \tilde q, V_0), \quad \sup_{s >0} \{ s^{1/ \tilde q - 1/2} \|V(s) \|_{\mc{L}^{2}} \} \leq H(2,\tilde q,V_0).
$$
and using  \eqref{C-ell-q}, we obtain:
$$
	\sup_{s >0} \{ s^{1/\tilde q} |\ell_V(s)| \} \leq H_{\ell}(\tilde q,V_0).
$$ 
Hence, for all $s >0$ and $\tilde q \in [q,2]$, we have:
\begin{equation}
	\label{Pre-est-NonlinearTerm}
	\|V(s)\|_{\mathcal L^4}^2 + |\ell_V(s)|\|V(s)\|_{\mc{L}^2} \leq H(4,\tilde q,V_0)^2 s^{1/2 - 2/\tilde q} + H(2,\tilde q, V_0) H_{\ell}(\tilde q, V_0) s^{1/2 - 2/\tilde q}.
\end{equation}

{\bf The case $q \in (4/3,2]$: proof of \eqref{eq_firstorder-c}.} In that case, combining \eqref{Est-Diff-Carpio1a} and \eqref{Pre-est-NonlinearTerm}, taking $\tilde q = q$, we immediately obtain:
$$
	\sup_{t >0} t^{-1/p - 1/2+ 2/q} \|V(t) - S(t) V_0\|_{\mathcal L^p} 
	\leq C(p,q,V_0)
	\int_0^1 (1- \tau)^{-1+1/p} \tau^{1/2 -2/q} \, {\rm d}\tau.
$$
where we used the fact that $1/2 - 2/q > -1$ for $q >4/3$.

{\bf The case $q \in (1, 4/3)$: proof of \eqref{eq_firstorder-a}.} Here, we write 
\begin{multline*}
  \int_0^t (t-s)^{-1+1/p} \left( \|V(s)\|_{\mathcal L^4}^2 + |\ell_V(s)|\|V(s)\|_{L^2(\mathcal F_0 )}\right)  \, {\rm d}s 
  \\
  = 
 \underbrace{\int_0^{t/2} (t-s)^{-1+1/p} \left( \|V(s)\|_{\mathcal L^4}^2 + |\ell_V(s)|\|V(s)\|_{\mc{L}^2}\right)   \, {\rm d}s}_{I_1(t)}
  + 
 \underbrace{ \int_{t/2}^{t} (t-s)^{-1+1/p} \left( \|V(s)\|_{\mathcal L^4}^2 + |\ell_V(s)|\|V(s)\|_{ \mc{L}^2}\right)   \, {\rm d}s }_{I_2(t)}.
\end{multline*}
Using \eqref{Pre-est-NonlinearTerm} with $\tilde q = 2$ for $s \in (0,1)$ and $\tilde q = q$ for $s \in (1,t/2)$ (recall $t>2$), and using $t - s \geq t/2$ for $s \leq t/2$,
$$
	I_1(t) \leq 		C(p,q,V_0)  t^{-1+1/p} 
	\left(
		 \int_0^1 s^{-1/2} \, {\rm d}s
		+ 
		 \int_1^{t/2} s^{1/2 - 2/q} \, {\rm d}s
	\right)
	\leq C(p,q,V_0) t^{-1+1/p}, 
$$
where we used that $1/2- 2/q < -1$ so that $\int_1^\infty s^{1/2- 2/q} \, {\rm d}s < \infty$.

Using \eqref{Pre-est-NonlinearTerm} with $\tilde q = 4/3$ for $s \in (t/2, t)$, we obtain
$$
	I_2(t) \leq C(p,q,V_0)\int_{t/2}^t (t-s)^{-1+1/p} s^{-1} \, {\rm d}s  = C(p,q,V_0) t^{-1+1/p}\int_{1/2}^1 (1-\tau)^{-1+1/p} \tau^{-1} \, {\rm d}\tau.
$$ 

{\bf The case $q = 4/3$: proof of \eqref{eq_firstorder-b}.} This case follows similarly as the previous one, except that estimating $I_1$ yields
$$
	I_1(t) \leq C(p,q,V_0) t^{-1+1/p} 
	\left(
		 \int_0^1 s^{-1/2} \, {\rm d}s
		+ 
		 \int_1^{t/2} s^{-1} \, {\rm d}s
	\right)
	\leq C(p,q,V_0) t^{-1+1/p} (1+ \log(t)). 
$$

\medskip

We now concentrate on the estimate \eqref{eq_firstorder2-a}--\eqref{eq_firstorder2-c} on $\ell_V(t)- \ell_{S(t)V_0}$. In order to do that, again, we split the integral in two parts: 
\begin{eqnarray*}
|\ell_V(t)- \ell_{S(t)v_0}| &\leq&   K_{\ell}(2) \int_0^{t/2} (t-s)^{-1} \Bigl( \|V(s)\|_{\mathcal L^4}^2 +  |\ell_V(s)|\|V(s)\|_{\mc{L}^2} \Bigl)  \, {\rm d}s \\ 
		&&+K_{\ell}(p) \int_{t/2}^t (t-s)^{-(\frac 12+\frac1p)} \Bigl( \|V(s)\|_{\mathcal L^{2p}}^2 +  |\ell_V(s)|\|V(s)\|_{\mc{L}^p} \Bigl)  \, {\rm d}s=:J_1(t)+J_2(t)\,.
\end{eqnarray*}
where $K_{\ell}$ is the constant of Corollary \ref{est_lv} and $p \in (2,\infty).$ The estimate of $J_1$ can be done as previously by using \eqref{Pre-est-NonlinearTerm}:
 $$
	J_1(t) \leq 
			\left\{ 
				\begin{array}{ll}
			C(p,q, V_0) t^{1/2 - 2/q} \quad & \text{if } q \in (4/3,2],
			\\
			C(p,q,V_0) t^{-1} (1+\log(t)) \quad & \text {if }  q = 4/3,
			\\
			C(p,q,V_0) t^{-1} \quad & \text {if }  q \in (1, 4/3)
				\end{array}
			\right.
 $$ 
For $J_2,$ remark that similarly as in \eqref{Pre-est-NonlinearTerm} we can obtain for all $\tilde q \in [q,2]$, $s>0$,
$$
	 \|V(s)\|_{\mathcal L^{2p}}^2 +  |\ell_V(s)|\|V(s)\|_{\mc{L}^p} \leq  \left( H(2p,\tilde q,V_0))^2 s^{1/p - 2/\tilde q}    + H(p,\tilde q, V_0) H_{\ell}(\tilde q ,V_0) s^{1/p - 2/\tilde q}\right)\,,
$$ 
so that:
$$
	J_2(t) \leq  C(p,\tilde q,V_0)  t^{1/2 - 2 /\tilde q}. 
$$
This ends the proof by choosing $\tilde q = q$ for $q \in (4/3,2]$ and $\tilde q = 4/3$ if $q \in (1,4/3]$.
\end{proof}

\newpage

\section{Further comments}\label{sect comments}

We list below several comments.

\bigskip

\paragraph{\em Concerning optimality of Theorem \ref{theo Stokes}} When considering the decay estimates of Theorem \ref{theo Stokes}, it is natural to ask if the results are sharp, in particular regarding the decay of the gradient estimates when $p>2$ and $t >1$, since all other decay estimates correspond to the classical ones for the heat semigroup on $\mathbb{R}^2$. However, in our proof, the decay estimate \eqref{est grad 2} differs from the one corresponding to the heat semigroup on $\mathbb{R}^2$ for all the modes. Each time, this slower decay rate for $t>1$ and $p>2$ arises due the presence of the boundary. Let us point out that P. Maremonti and V.~A. Solonnikov prove in \cite{MS97} that, when considering the Stokes equations in an exterior domain of $\mathbb{R}^3$ with homogeneous Dirichlet boundary conditions, estimate \eqref{est grad 2R}, which is the counterpart of \eqref{est grad 2}, is sharp. It is then likely that estimates of Theorem \ref{theo Stokes} are sharp as well.

\bigskip

\paragraph{\em Straightforward extensions of theorems \ref{theo kato} and \ref{theo Stokes}}$\phantom{123}$\\
\indent $\bullet$ Using the density of $\mc{L}^1 \cap  \mc{L}^2$ in $\mc{L}^q \cap \mc{L}^2$ and the decay estimates of Theorem \ref{theo Stokes}, one easily get, for all $q \in (1, 2]$, for all $V_0 \in \mc{L}^q \cap \mc{L}^2$ satisfying $\| V_0 \|_{\mc{L}^2} \leq \lambda_0(5/4)$, the unique associated solution $V(t)$ to \eqref{NS11}--\eqref{Solideci1}
satisfies:
\begin{equation}
	\label{LimSup}
	\lim_{t \to \infty} t^{1/q-1/p} \| V(t) \|_{\mc{L}^p} = 0.
\end{equation}
Indeed, for $V_0 \in \mc{L}^q \cap \mc{L}^2$ satisfying $\| V_0 \|_{\mc{L}^2} \leq \lambda_0(5/4)$, by Theorem \ref{theo kato}, $t^{1/q-1/p} \| V(t) -S(t)V_0 \|_{\mc{L}^p}$ goes to zero as $t \to \infty$ (recalling that $\lambda_0(q)>\lambda_0(5/4)$ if $q>5/4$, see the end of Introduction). We then take $\varepsilon >0$ and $\tilde V_0 \in \mc{L}^1 \cap  \mc{L}^2$ satisfying $\| V_0 - \tilde V_0\|_{\mc{L}^q} \leq \varepsilon$. According to Theorem \ref{theo Stokes}, $t^{1/q -1/p} \|S(t) V_0 - S(t) \tilde V_0\|_{\mc{L}^p} \leq K_1(p,q) \nu^{1/p-1/q}\epsilon$. But Theorem \ref{theo Stokes} also implies $\lim_{t \to \infty} t^{1/q-1/p} \| S(t) \tilde V_0\| = 0$ since $\tilde V_0$ belongs to $\mc{L}^{\tilde q}$ for some $\tilde q \in (1,q)$. Hence 
$$\limsup_{t \to \infty} t^{1/q-1/p} \| V(t) \| \leq C \varepsilon.$$ 
Since $\varepsilon$ was arbitrary, this implies \eqref{LimSup}.

$\bullet$ The proofs of Theorems \ref{Thm-q-dans(1,2)}--\ref{theo kato} are only based on the $L^p-L^q$ estimates for the Stokes problem given in Theorem \ref{theo Stokes}. As such estimates are already known in the case of a fixed exterior domain (see \cite{DS99a,DS99,MS97}), we claim that Theorem \ref{theo kato} holds true also in this case. Hence, the computations herein extend the results in \cite{GallayMaekawa,Iftimieetcie} to the case of finite energy initial data.

$\bullet$ In order to obtain the decay estimates of Theorem \ref{theo Stokes}, our approach is strongly based on the fact that the rigid body is an homogenous disk. Indeed, in polar coordinates we decompose in Fourier series. A case which can be easily treated by our analysis is when {\it the disk is non-homogenous, and the center of mass corresponds to the center of the disk} (e.g. for a density $\rho$ with radial symmetry). In this case, the equations \eqref{NS11}-\eqref{Solideci1} and \eqref{S1}-\eqref{S-Solideci} are the same, where:
\[
m = \int_{B_0} \rho(x) \, {\rm d}x; \quad \mathcal{J}=\int_{B_0} \rho(x) |x|^2 \, {\rm d}x.
\]
To our knowledge, the case of a more general shape or more general density is completely open. A similar problem, also open, would be to derive decay estimates in the case of two rigid disks.

\bigskip

\paragraph{\em Open problems} $\phantom{123}$\newline
\indent $\bullet$  Despite Theorem \ref{theo kato}, a complete description of the first term in the asymptotic behavior as $t \to \infty$ of the solutions $V$ of \eqref{NS11}--\eqref{Solideci1} is still missing.  Indeed, Theorems \ref{thm_AsymptoticStokes} and \ref{theo kato} cannot be combined since Theorem \ref{thm_AsymptoticStokes} requires the initial data to be $\mc{L}^1$ and in that case, Theorem \ref{theo kato} only yields that the $\mc{L}^p$-norm of the difference between the solution of the complete non-linear system \eqref{NS11}--\eqref{Solideci1} and the linear one, given by $S(t) V_0$, decays as $C t^{1/p-1}$, which is precisely the order of magnitude of the $\mc{L}^p$-norm of the solution of the linear Stokes equation when $V_0\in \mc{L}^1$.
At this level, let us emphasize that one of the main conceptual difficulties of this problem is that the invariant seems to be the $\mc{L}^1$-norm of the solution of \eqref{NS11}-\eqref{Solideci1}, despite the fact that the linear semigroup does not seem to be well-posed in $\mc{L}^1$. To justify this statement, we emphasize that the asymptotic given by Theorem \ref{thm_AsymptoticStokes} does not belong to $\mathcal{L}^1$. 
Showing that the non-linear term decreases faster than $t^{-1+1/p}$ is the major issue which prevents us from extracting the asymptotic first order. 

If we look carefully at the proof of Theorem \ref{theo kato}, we note that the difficulty comes from the fact that we do not manage to prove that $\| S(t) \mathbb P \div \, F \|_{\mathcal L^p}$ decays  faster for $F \in \mc{L}^q$ with $q<2$ (see \eqref{est div 2}). 
In particular, in the case of the Navier-Stokes equations in $\mathbb{R}^2$, using heat kernel estimates, which are better than estimates \eqref{est div 2} when $q<2$ and $t >1$, A. Carpio in \cite{Carpio} shows that the non-linear term is smaller than the Stokes solution for large time. But as we have noted above, the restriction in \eqref{est div 2} seems to be unavoidable. 

Nevertheless, other methods relying on the use of suitable scaling invariance and similarity variables  have been used for providing leading order terms in  \cite{Carpio,MunnierZuazua,GallayWayne}. To keep the unity of this paper we postpone these  approaches to a future work. 

$\bullet$ Another open problem is to remove the smallness condition in Theorem \ref{theo kato}, as it is done for Navier-Stokes in the full plane \cite{Schonbek,Wiegner} and in fixed exterior domains \cite{BorchersMiyakawa}. Indeed, such result would be expected in view of the energy dissipation law \eqref{formalenergy} which indicates the decay of the $\mc{L}^2$-norm of the solutions of \eqref{NS1}--\eqref{Solideci}.

\bigskip

\noindent
{\bf Acknowledgements.} The first author is partially supported by the Agence Nationale de la Recherche (ANR, France), Project CISIFS number NT09-437023 The two last authors are partially supported by the Project ``Instabilities in Hydrodynamics'' financed by Paris city hall (program ``Emergences'') and the Fondation Sciences Math\'ematiques de Paris. The second author is partially supported by the Agence Nationale de la Recherche, Project RUGO,  grant ANR-08-JCJC0104. The third author is partially supported by the Agence Nationale de la Recherche, Project MathOc\'ean, grant ANR-08-BLAN-0301-01.

\appendix 

\section{Proof of Proposition \ref{prop_calculsysteme}}
\label{App_calculsysteme}
Assume $V_0 \in \mathcal{L}^2 \cap \mathcal{C}^{\infty}_c(\mathbb R^2)$ Hille-Yosida's theorem 
implies there exists a unique solution $V \in \mathcal{C}([0,\infty);\mathcal{L}^2)$ to   \eqref{S1}--\eqref{S-Solideci}. 
 Furthermore, the unknowns  $(\ell_V,\omega_V)$  and the pressure $p$ that are constructed starting from $V$
 have, with $V,$ the following regularity (see \cite[Corollary 4.3]{Takahashi&Tucsnak04}):
\begin{eqnarray*}
v \in  \mathcal{C}([0,\infty); [H^1(\mathcal{F}_0)]^2) \cap {L}^2((0,\infty); [H^2(\mathcal{F}_0)]^2) \,,
& & 
\nabla p \in L^2((0,\infty) ; L^2(\mathcal F_0))\,,  \label{eq_regsol}\\[4pt]
 \ell_V \in H^1(0,\infty) \,, & & 
\omega_V \in H^1(0,\infty)\,. \label{eq_regsol2}
\end{eqnarray*}
We note also that further smoothing properties of the semigroup (see \cite[Theorem 7.7]{Brezis})
imply 
\begin{equation*}
	(v,p) \in [\mathcal{C}^{\infty}((0,\infty) \times \overline{\mathcal{F}}_0)]^3, \quad \nabla p \in \mathcal{C}((0,\infty);L^2(\mathcal F_0)) \,.
\end{equation*} 

Consequently, we introduce the decomposition $(W,\Phi,\Psi,V_R)$ of $V$ in spherical harmonics and a corresponding decomposition 
of the pressure $p$:
\begin{eqnarray}
	p_1(t,r) &=&  \dfrac{1}{\pi} \int_{0}^{2\pi} p(t,r,\theta) \cos \theta\, \text{d$\theta$}\,, \label{1-mode-p}
	\\
	q_1(t,r) &=&  \dfrac{1}{\pi} \int_{0}^{2\pi} p(t,r,\theta) \sin \theta\, \text{d$\theta$}\,, \label{1-mode-q}
	\\[6pt]
	p_R(t,r) &=& p(t,r,\theta) - p_1(t,r) \cos \theta - q_1(t,r) \sin \theta\,. \label{R-modes-p}
\end{eqnarray}
Note that, like $v_R,$ the remainder term $p_R$ satisfies:
\begin{equation} \label{eq_pR}
\int_0^{2\pi} p_R(t,r,\theta) \cos(\theta) \, \text{d$\theta$} = \int_0^{2\pi} p_R(t,r,\theta) \sin(\theta)\, \text{d$\theta$} = 0.
\end{equation}
Applying the continuity of the spherical-harmonic decomposition together with the continuity of $V$ yields:
\begin{eqnarray*}
W  \in \mathcal{C}([0,\infty) ; L^2((0,\infty),r{\rm d}r))\,&& (\partial_r\Psi,\Psi/r) \in \mathcal{C}([0,\infty) ; L^2((0,\infty),r{\rm d}r))\, \\
V_R \in  \mathcal{C}([0,\infty) ; L^2_{\sigma}(\mathcal{F}_0)) \, && (\partial_r\Phi,\Phi/r) \in \mathcal{C}([0,\infty) ; L^2((0,\infty),r{\rm d}r)),
\end{eqnarray*}
together with 
$$
\partial_r p_1 \in \mathcal{C}((0,\infty); L^2((1,\infty),r{\rm d}r))\,, \quad \partial_r q_1 \in \mathcal{C}((0,\infty); L^2((1,\infty),r{\rm d}r)).
$$
Referring to the formulas \eqref{eq_w0}-\eqref{eq_psi1} and \eqref{1-mode-p}-\eqref{1-mode-q} we also obtain at once the smoothness of $(W,V_R,\Phi,\Psi),$ of $(p_1,q_1)$ and of $(V_R,p_R).$ It now remains to compute the systems satisfied by these unknowns.

\medskip

We recall that the spherical-harmonic  decomposition of  $V$ reads $V = V_r e_r + V_{\theta} e_{\theta}$
\begin{equation}\label{expu}
V_r =   \dfrac{\Psi}{r} \sin \theta - \dfrac{\Phi}{r} \cos \theta + V_{R} \cdot e_r \,, 
\quad 
V_{\theta} = W\min(1,r) + \partial_r \Psi \cos \theta + \partial_r \Phi \sin \theta  + V_{R} \cdot e_{\theta}
\end{equation}
and the velocity-field on the disk is given as follows in radial coordinates: 
$$
(\ell_V + \omega_V x^\perp)_r =   \ell_{V,1} \cos (\theta )+ \ell_{V,2} \sin(\theta)  \,, 
\quad 
(\ell_V + \omega_V x^\perp)_{\theta} =   \omega r  - \ell_{V,1}\sin( \theta) + \ell_{V,2} \cos(\theta).
$$	
Identifying $V$ and the velocity-field of the disk on $\partial B_0$ ( \emph{i.e.} for $r=1$), we obtain the following boundary conditions:
$$
\begin{array}{rcllcrcll}
	w(t,1) & =  &\omega(t), &\forall \, t   \geq 0, & & &  & 
	\\
	\psi(t,1) & =  &  \ell_{V,2}(t)\,, && \partial_r \psi (t,1)& =& \ell_{V,2}(t)\,,   &  \forall \, t \geq 0,
	\\
	\varphi_1(t,1) &=&  - \ell_{V,1}(t)\,,& & \partial_r \varphi_1(t,1)&= &-\ell_{V,1}(t),  & \forall \,  t \geq 0\,.
	\\
	v_R(t,x) & =&  0 \,,& \forall \,  x \in \partial B_0 \,, & \forall \, t \geq 0. & &
\end{array}
$$

In the fluid domain, we remark that, introducing $\chi$ such that $\partial_r \chi = w$ and $\chi(0) = 0$, the spherical-harmonic decomposition reads:
$$
v = \nabla^{\perp} \left( \chi + \psi \cos(\theta) + \phi \sin(\theta) \right) + v_R 
$$
so that:
$$
\partial_t v - \nu \Delta v = \nabla^{\perp}  [\partial_t - \nu \Delta] \left(  \chi + \psi \cos(\theta) + \phi \sin(\theta)  \right)  + \partial_t v_R - \nu \Delta v_R\,,
$$
where, in polar coordinates:
$$
\Delta \psi(t,r,\theta) = \dfrac{1}{r}\partial_{r} [ r \partial_r \psi](t,r,\theta) + \dfrac{ \partial_{\theta\theta}\psi(t,r,\theta)}{r^2}\,.  
$$
We also recall that, the gradient operator reads: 
\begin{equation*}
	\nabla q =  \partial_r q e_r + \dfrac{\partial_{\theta} q}{r} e_{\theta}\,.
\end{equation*}
Finally, we remark that orthogonality conditions such as \eqref{eq_vR} or \eqref{eq_pR} transmit to time and space derivative.
Hence, replacing $\psi$ and $p$ by their values in the two last formulas, identifying then the different frequencies: constant, $\cos \theta$, $\sin \theta,$ and remainders,  we get:
\begin{eqnarray*}
\partial_t w - \nu \left(  \partial_{rr} w +  \dfrac{\partial_r w}{r} - \dfrac{w}{r^2} \right)= 0\,, && \text{ for } (t,r ) \in (0, \infty) \times (1, \infty);
\end{eqnarray*}
\begin{eqnarray*}
\partial_t \psi- \nu \left(  \partial_{rr} \psi +  \dfrac{\partial_r \psi}{r} - \dfrac{\psi}{r^2} \right)  = -r \partial_r q_1\,,  && \text{ for } (t,r ) \in (0, \infty) \times (1, \infty);
\\
\partial_t \partial_r \psi - \nu \partial_r \left(  \partial_{rr} \psi +  \dfrac{\partial_r \psi}{r} - \dfrac{\psi}{r^2} \right)=-\frac{q_1}{r}\,,  && \text{ for } (t,r ) \in (0, \infty) \times (1, \infty);
\end{eqnarray*}
\begin{eqnarray}
\partial_t \phi-  \nu \left(  \partial_{rr} \phi +  \dfrac{\partial_r \phi}{r} - \dfrac{\phi}{r^2} \right) =  r \partial_r p_1\,,  && \text{ for } (t,r ) \in (0, \infty) \times (1, \infty);\\
\partial_t \partial_r \phi -  \nu \partial_r  \left(  \partial_{rr} \phi +  \dfrac{\partial_r \phi}{r} - \dfrac{\phi}{r^2} \right) = \frac{p_1}{r}\,,  && \text{ for } (t,r ) \in (0, \infty) \times (1, \infty);
\end{eqnarray}
\begin{eqnarray}
\partial_t v_R -  \nu \Delta v_R + \nabla p_R =0\,,  && t\geq 0,\ x\in \mathcal F_0\,.
\end{eqnarray}
To end up the proof of {Proposition \ref{prop_calculsysteme}}, we write now \eqref{S-Solide1}-\eqref{S-Solide2}.
First we recall that, on $B_0$ there holds $n = - e_r$ (the normal $n$ is here computed outward the fluid domain) and $r=1$, 
so that:
$$
\begin{array}{rcl}
-2D(v) n  & = &  \partial_r v + \nabla v_r  -  [\nabla e_r]^{\top} v  \,, \\[4 pt]
	     &=&  ( 2  \partial_r v_r ) e_r + (\partial_r v_{\theta} + \partial_{\theta} v_r-  v_{\theta} )e_{\theta}  \,, \\[4 pt]
-\Sigma n  & = &  (-p +2 \nu  \partial_r v_r  ) e_r +\nu (\partial_r v_{\theta} + \partial_{\theta} v_r-  v_{\theta} )e_{\theta} .
\end{array}		   
$$
Hence, computing for instance $\ell_{V,2}' = \ell_V' \cdot e_2$  we have:
\begin{eqnarray*}
m \ell_{V,2}' &=&  \int_{0}^{2\pi} ( 2 \nu  \partial_r v_r - p ) \sin \theta \, \text{d$\theta$} + \nu \int_{0}^{2\pi}   (\partial_r v_{\theta} + \partial_{\theta} v_r-  v_{\theta} )\cos \theta \, \text{d$\theta$}\,,\\
		&=& \int_{0}^{2\pi} (2  \nu \partial_r v_r + \nu v_r -  p ) \sin \theta \, \text{d$\theta$} + \nu \int_{0}^{2\pi}   (\partial_r v_{\theta} + \partial_{\theta} v_r )\cos \theta \, \text{d$\theta$}\,.
\end{eqnarray*} 
In these last integrals, we then compute $\partial_r v_r$ and $\partial_r v_{\theta}$ with respect to $w,$ $\psi$ and $\phi,$ and $v_R$ thanks to \eqref{expu}. Recalling the orthogonality conditions  \eqref{eq_vR} and \eqref{eq_pR}, we get:
\begin{eqnarray*}
m  \ell_{V,2}' &=& \pi \left[  \nu \left( 2  \partial_r \left[\frac{\psi}{r} \right] +  \partial_{rr} \psi + \psi - \partial_r \psi\right) - q_1 \right]\,, \\
		&=& \pi \left[ \nu \left( \partial_{rr} \psi +  \dfrac{\partial_r \psi}{r} - \dfrac{\psi}{r^2} \right)- q_1\right]\,. 
\end{eqnarray*} 
The computations of $\mathcal{J} \omega_V'$ and $m \ell_{V,1}'$ are similar.


\begin{thebibliography}{99}

\bibitem{BorchersMiyakawa}
W.~Borchers and T.~Miyakawa.
\newblock {$L^2$}-decay for {N}avier-{S}tokes flows in unbounded domains, with
  application to exterior stationary flows.
\newblock {\em Arch. Rational Mech. Anal.}, 118(3):273--295, 1992.

\bibitem{Brezis}
H.~Brezis.
\newblock {\em Analyse fonctionnelle. (French) [Functional analysis] Th\'eorie et applications. [Theory and applications] }, {Collection Math\'ematiques Appliqu\'ees pour la Ma\^\i trise.
              [Collection of Applied Mathematics for the Master's Degree]}.
\newblock Masson, Paris, 1983. xiv+234 pp. 


\bibitem{Carpio}
A.~Carpio.
\newblock Asymptotic behavior for the vorticity equations in dimensions two and
  three.
\newblock {\em Comm. Partial Differential Equations}, 19(5-6):827--872, 1994.

\bibitem{CrispoMaremonti}
F.~Crispo and P.~ Maremonti.
\newblock An interpolation inequality in exterior domains.
\newblock {\em Rend. Sem. Mat. Univ. Padova}, 112 (2004), 11-39.

\bibitem{DS99a}
W.~Dan and Y.~Shibata.
\newblock On the {$L_q$}--{$L_r$} estimates of the {S}tokes semigroup in a
  two-dimensional exterior domain.
\newblock {\em J. Math. Soc. Japan}, 51(1):181--207, 1999.

\bibitem{DS99}
W.~Dan and Y.~Shibata.
\newblock Remark on the {$L_q$}-{$L_\infty$} estimate of the {S}tokes semigroup
  in a {$2$}-dimensional exterior domain.
\newblock {\em Pacific J. Math.}, 189(2):223--239, 1999.

\bibitem{EscobedoZuazua}
M.~Escobedo and E.~Zuazua.
\newblock Large time behaviour for convection-diffusion equations in $\mathbb R^N$
\newblock  {\em J. Func. Analysis},100:119--161,1991.

\bibitem{FeireislNecasova}
E.~Feireisl and {\v{S}}.~Ne{\v{c}}asov{\'a}.
\newblock On the long-time behaviour of a rigid body immersed in a viscous
  fluid.
\newblock {\em Applicable Analysis}, 90(1):59 -- 66, 2011.

\bibitem{Galdi}
G.~P. Galdi.
\newblock {\em An introduction to the mathematical theory of the
  {N}avier-{S}tokes equations}.
\newblock Springer Monographs in Mathematics. Springer, New York, second
  edition, 2011.
\newblock Steady-state problems.

\bibitem{GallayMaekawa}
T.~Gallay and Y.~Maekawa.
\newblock Long-time asymptotics for two-dimensional exterior flows with small
  circulation at infinity.
\newblock {\em arXiv:1202.4969v1 [math.AP]}, pages 1--18, 2012.

\bibitem{GallayWayne}
T.~Gallay and C.~E. Wayne.
\newblock Global stability of vortex solutions of the two-dimensional
  {N}avier-{S}tokes equation.
\newblock {\em Comm. Math. Phys.}, 255(1):97--129, 2005.

\bibitem{Iftimieetcie}
D.~Iftimie, G.~Karch and C.~Lacave.
\newblock Self-similar asymptotics of solutions to the navier-stokes system in
  two-dimensional exterior domain.
\newblock arXiv:1107.2054v1.

\bibitem{KajikiyaMiyakawa}
R. Kajikiya and T. Miyakawa.
\newblock On $L^2$ decay of weak solutions of the Navier-Stokes equations in $R^n$.
\newblock {\em Math. Z.}, 192 (1986), no. 1, 135-148.

\bibitem{Kato84}
T. Kato.
\newblock Strong $L^p$ solutions of the Navier-Stokes equations in $\R^n$, with application to weak solutions.
\newblock {\em Math. Z.}, 187 (1984), pp. 471-480.

\bibitem{KatoFujita62}
T. Kato and H. Fujita.
\newblock On the non stationary Navier-Stokes system.
\newblock {\em Rend. Sem. Mat. Univ. Padova}, 32 (1962), pp. 243-260.

\bibitem{Leray} 
J. Leray.
\newblock Sur le mouvement d'un liquide visqueux emplissant l'espace. (French)
\newblock {\em Acta Math.} 63 (1934), no. 1, 193-248.

\bibitem{MS97}
P.~Maremonti and V.~A. Solonnikov.
\newblock On nonstationary {S}tokes problem in exterior domains.
\newblock {\em Ann. Scuola Norm. Sup. Pisa Cl. Sci. (4)}, 24(3):395--449, 1997.

\bibitem{MunnierZuazua2}
A. Munnier and E. Zuazua.
\newblock Large time behavior for a simplified n-dimensional model of fluid-solid interaction.
\newblock {\em Cahiers du Ceremade}, (2004).

\bibitem{MunnierZuazua}
A. Munnier and E. Zuazua.
\newblock Large time behavior for a simplified n-dimensional model of fluid-solid interaction.
\newblock {\em Comm. Partial Differential Equations}, 30 (2005), no. 1-3, 377-417.

\bibitem{Pazy}
A.~Pazy.
\newblock {\em Semigroups of linear operators and applications to partial
  differential equations}, volume~44 of {\em Applied Mathematical Sciences}.
\newblock Springer-Verlag, New York, 1983.

\bibitem{Schonbek}
M.~E. Schonbek.
\newblock {$L^2$} decay for weak solutions of the {N}avier-{S}tokes equations.
\newblock {\em Arch. Rational Mech. Anal.}, 88(3):209--222, 1985.

\bibitem{Takahashi&Tucsnak04}
T.~Takahashi and M.~Tucsnak.
\newblock Global strong solutions for the two-dimensional motion of an infinite
  cylinder in a viscous fluid.
\newblock {\em J. Math. Fluid Mech.}, 6(1):53--77, 2004.

\bibitem{Vazquez}
J.~L.~V{\`a}zquez.
\newblock Asymptotic behaviour for the porous medium equation posed in the whole space.
\newblock {\em J. Evol. Equ.}, 3: 67--118, 2003.

\bibitem{VazquezZuazua}
J.~L.~V{\`a}zquez and E.~Zuazua.
\newblock Large time behavior for a simplified 1{D} model of fluid-solid
  interaction.
\newblock {\em Comm. Partial Differential Equations}, 28(9-10):1705--1738,
  2003.

\bibitem{Vazquez&Zuazua04}
J.~L.~V{\`a}zquez and E.~Zuazua.
\newblock Lack of collision in a simplified 1-d model for fluid-solid
  interaction.
\newblock {\em Math. Models Methods Appl. Sci.}, 16(5):637--678, 2006.

\bibitem{Veron79}
L.~V{\'e}ron.
\newblock Effets r\'egularisants de semi-groupes non lin\'eaires dans des
  espaces de {B}anach.
\newblock {\em Ann. Fac. Sci. Toulouse Math. (5)}, 1(2):171--200, 1979.

\bibitem{Wiegner}
M.~Wiegner.
\newblock Decay results for weak solutions of the {N}avier-{S}tokes equations
  on {${\bf R}^n$}.
\newblock {\em J. London Math. Soc. (2)}, 35(2):303--313, 1987.

\bibitem{Yun&Wang} 
Y. Wang and Z. Xin.
\newblock Analyticity of the semigroup associated with the fluid-rigid body problem and local existence of strong solutions.
\newblock {\em J. Funct. Anal.} 261 (2011), no. 9, 2587-2616.

  \end{thebibliography}
\end{document}